\documentclass{amsart}
\usepackage{concmath}
\usepackage{amsmath}
\usepackage{amsthm}
\usepackage{amstext}
\usepackage{amssymb}
\usepackage{amsfonts}
\usepackage{pstricks}
\usepackage{graphicx}
\usepackage{lscape}
\usepackage{epsfig}
\usepackage{epstopdf}
\theoremstyle{plain}
\newtheorem{thm}{Theorem}[section]
\newtheorem{prop}[thm]{Proposition}
\newtheorem{lem}[thm]{Lemma}
\newtheorem{cor}[thm]{Corollary}

\newtheorem{claim}[thm]{Claim}

\theoremstyle{definition}
\newtheorem{defn}[thm]{Definition}
\newtheorem{rem}[thm]{Remark}
\newcommand{\vol}{\mathrm{vol}}
\newcommand{\R}{\mathbf{R}}
\newcommand{\Rn}{\mathbf{R}^n}
\newcommand{\Rth}{\mathbf{R}^3}
\newcommand{\Rfo}{\mathbf{R}^4}
\newcommand{\Hth}{\mathbf{H}^3}
\newcommand{\Hn}{\mathbf{H}^n}

\newcommand{\Sth}{\mathbf{S}^3}
\newcommand{\Sn}{\mathbf{S}^n}
\newcommand{\Gm}{\mathbf{G}^m}
\newcommand{\sech}{\textrm{sech\,}}
\newcommand{\Prop}{Proposition}
\newcommand{\Lem}{Lemma}
\newcommand{\Rem}{Remark}

\newcommand{\Thm}{Theorem}

\begin{document}

\title[Double Bubbles in $\Sth$ and $\Hth$]{ Double Bubbles in $\Sth$ and $\Hth$}

\date{\today}

\author[J. Corneli et al.]{Joseph Corneli, Neil Hoffman, Paul Holt, George Lee, Nicholas Leger, Stephen Moseley, Eric Schoenfeld}

\address{Mailing address: C/O Frank Morgan\\
          Department of Mathematics and Statistics\\
         \indent Williams College\\
         Williamstown, MA 01267}

\email{Frank.Morgan@williams.edu}

\address{Joseph Corneli, Neil Hoffman, and Nicholas Leger\\
Department of Mathematics\\
University of Texas\\
 Austin, TX 78712}

\email{jcorneli@math.utexas.edu}
\email{nhoffman@math.utexas.edu}
\email{nleger@math.utexas.edu}

\address{Paul Holt and George Lee\\
          Department of Mathematics and Statistics\\
         Williams College\\
         Williamstown, MA 01267}

\email{pholt@wso.williams.edu}
\email{georgelee@post.harvard.edu}

\address{Stephen Moseley\\
                Center for Applied Mathematics\\ 
                657 Frank H.T. Rhodes Hall\\
Cornell University\\
Ithaca, NY 14853}
                
\email{ssm37@cornell.edu}

\address{Eric Schoenfeld\\
                Stanford University \\
                Mathematics, Bldg. 380\\
                450 Serra Mall\\
                Stanford, CA 94305-2125\\}
                         
\email{erics@math.stanford.edu}

\maketitle

%%%%%%%%%%%%%%%%%%%%%%%%%%%%%%%%%%%%%%%%%%%%%%%%%%%%%%%%%%%%%%%%%%%%%%
%%%%%%%%%%%%%%%%%%%%%%%%%%%%%%%%%%%%%%%%%%%%%%%%%%%%%%%%%%%%%%%%%%%%%%
%%%%%%%%%%%%%%%%%%%%%%%%%%%%%%%%%%%%%%%%%%%%%%%%%%%%%%%%%%%%%%%%%%%%%%
\tableofcontents
%%%%%%%%%%%%%%%%%%%%%%%%%%%%%%%%%%%%%%%%%%%%%%%%%%%%%%%%%%%%%%%%%%%%%%
%%%%%%%%%%%%%%%%%%%%%%%%%%%%%%%%%%%%%%%%%%%%%%%%%%%%%%%%%%%%%%%%%%%%%%
%%%%%%%%%%%%%%%%%%%%%%%%%%%%%%%%%%%%%%%%%%%%%%%%%%%%%%%%%%%%%%%%%%%%%%
\section*{Abstract}
We prove the double bubble conjecture in the three-sphere $\Sth$ and hyperbolic three-space $\Hth$ 
in the cases where we can apply Hutchings theory: 
\begin{itemize}
\item{ in $\Sth$, each enclosed volume
and the complement occupy at least 10\% of the volume of $\Sth$;}
\item{in $\Hth$, the smaller volume is at least 85\% that of the larger.}
\end{itemize}
 A balancing argument and asymptotic analysis reduce the problem in $\Sth$ and $\Hth$ to some computer checking. The computer analysis has been designed and fully implemented for both spaces.     
%%%%%%%%%%%%%%%%%%%%%%%%%%%%%%%%%%%%%%%%%%%%%%%%%%%%%%%
\listoffigures
%%%%%%%%%%%%%%%%%%%%%%%%%%%%%%%%%%%%%%%%%%%%%%%%%%%%%%%

\section{Introduction}

\subsection{The double bubble conjecture in $\Sth$, $\Hth$}

In March of 2002,  Hutchings, Morgan, Ritor\'e, and Ros \cite{HMRR} proved 
that the area-minimizing way to enclose
and separate two given volumes in $\Rth$ is by a standard double
bubble, defined as three spherical caps meeting in threes at 120
degree angles.  In 2003, Cotton and Freeman extended these methods
to $\Sth$ and $\Hth$, proving that the standard double bubble is
the most efficient way to enclose and separate two equal volumes
in these historically important non-Euclidean spaces, with the
added condition that in $\Sth$ the exterior of the double bubble
takes up at least 10\% of the volume of $\Sth$.  More recently, Corneli et al. 
\cite{CCHSX} have proved the double bubble conjecture in spheres 
of all dimensions provided the double bubble partitions the sphere 
into the nearly equal volumes.  

Our Main Theorem in $\Sth$ (Theorem \ref{main_sth}) improves upon Cotton 
and Freeman's results when $v,w$ are sufficiently large, 
showing that that the standard double bubble is the least-area way to enclose
and separate two regions of prescribed unequal volumes,
whenever each region and the exterior contains at least ten
percent of the total volume. In $\Hth$, our Main Theorem \ref{main_hth}, 
extends Cotton and Freeman's results, showing double bubbles are standard whenever the smallest
region has volume at least $0.85$ that of the the larger region.
Computer plots \ref{fig_F_Sth_plot} and \ref{fig_F_Hth_plot} show that this is approximately the largest range on which current methods of proof using the positivity of the Hutchings function can work.  Indeed, the main focus of this paper is rigorously proving that the Hutchings function is positive on most of the regions where it appears positive in these plots. The plots suggest that $F(v,w)$ is always positive when $v$ is the volume of the larger region. This fact is shown in Remark \ref{cf_answer}, answering the open question of Cotton and Freeman \cite[Open Question 1.1]{CF}.

%%%%%%%%%%%%%%%%%%%%%%%%%%%%%%%%%%%%%%%%%%%%%%%%%%%%%%%%%%%
\begin{figure}
\includegraphics[width=1.\textwidth]{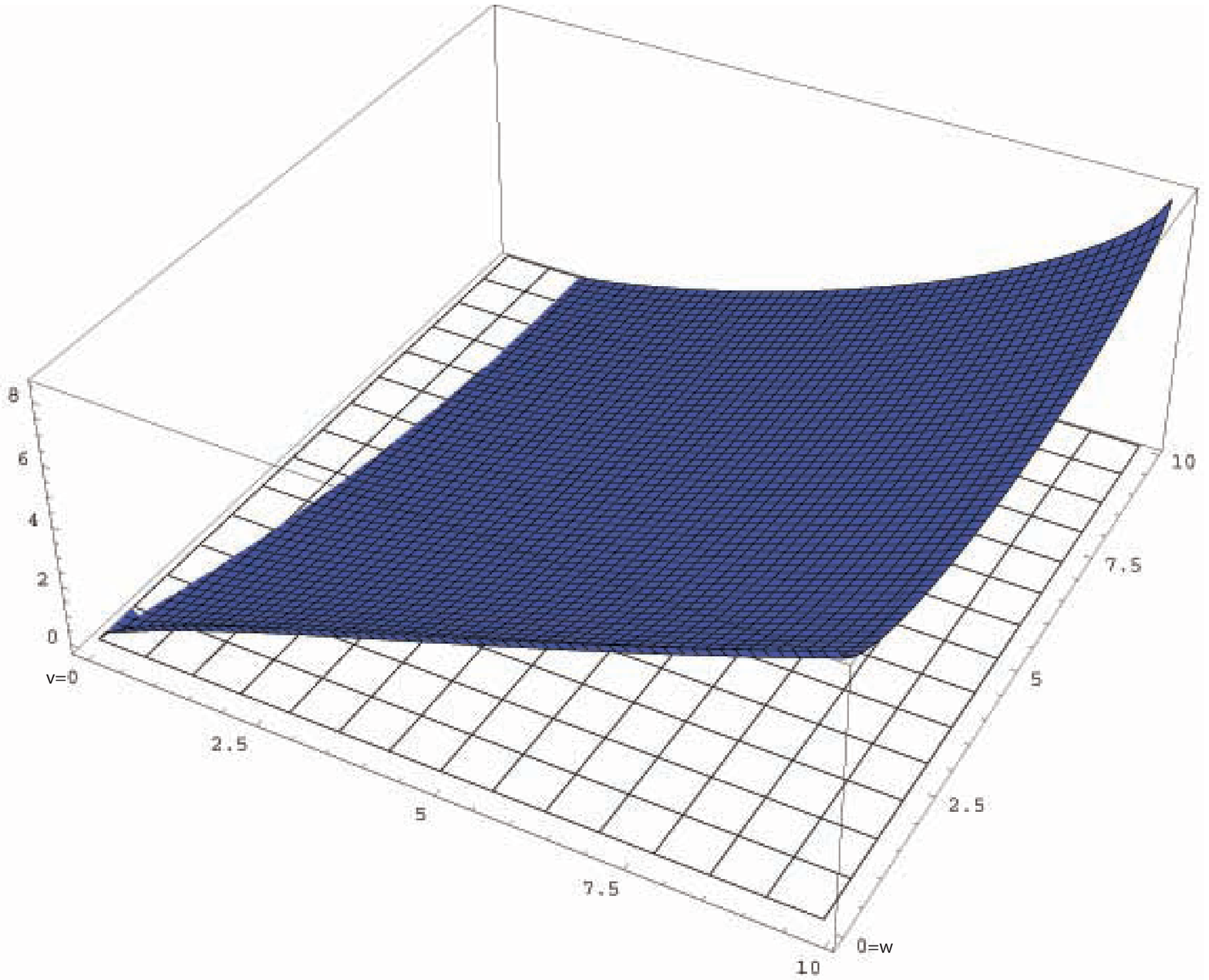}
\begin{center}
\end{center}
\caption{ \label{fig_F_Sth_plot} Plot of the
Hutchings function $F(v,w)$ where it is positive in $\Sth$. }
\end{figure}
%%%%%%%%%%%%%%%%%%%%%%%%%%%%%%%%%%%%%%%%%%%%%%%%%%%%%%%%%%%%%%

%%%%%%%%%%%%%%%%%%%%%%%%%%%%%%%%%%%%%%%%%%%%%%%%%%%%%%%%%%%%%%
\begin{figure}[h]
\includegraphics[width=1.\textwidth]{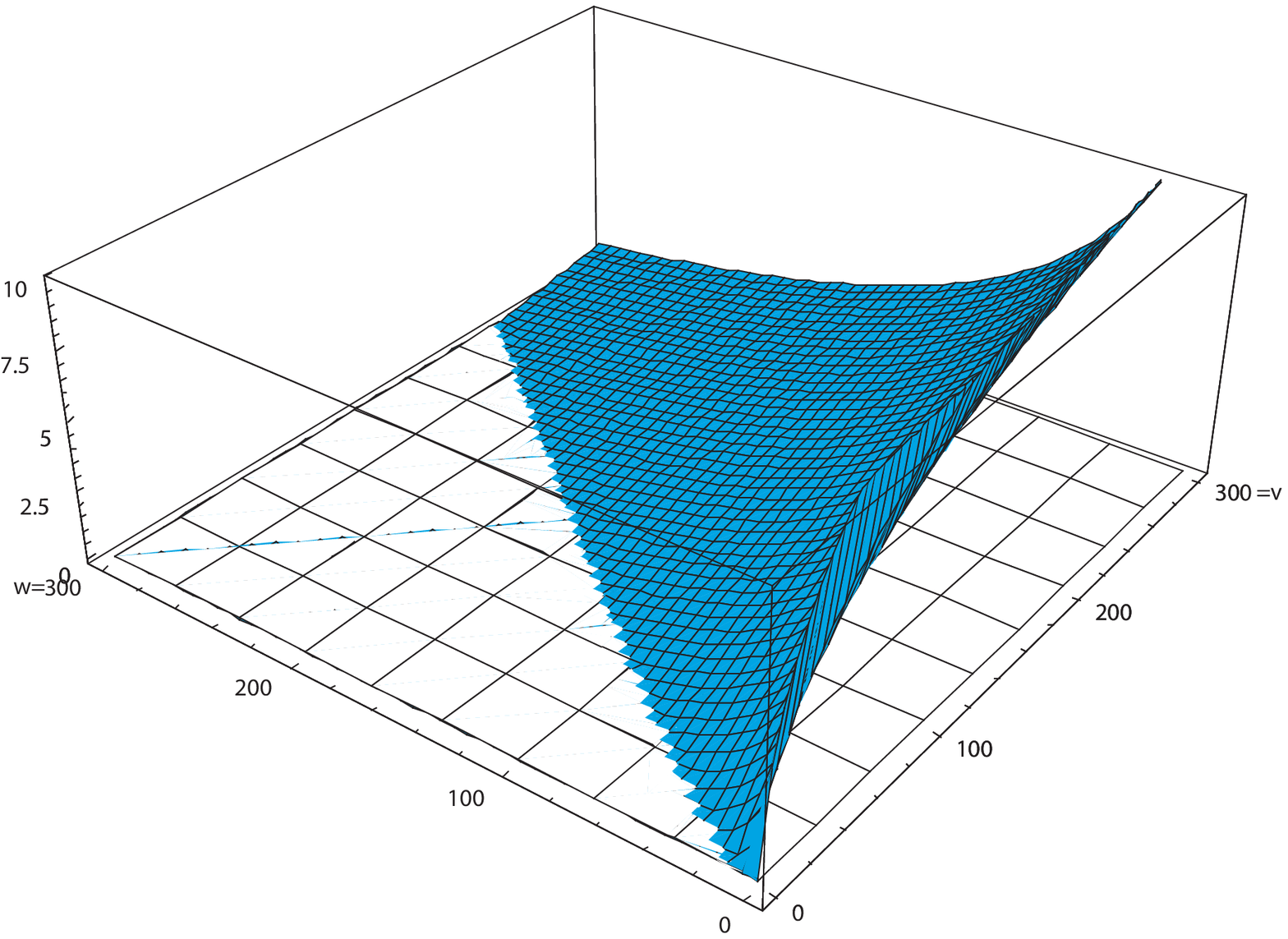}
\begin{center}
\end{center}
\caption{\label{fig_F_Hth_plot} Plot of the Hutchings function $F(v,w)$
where it is positive in $\Hth$.}
\end{figure}

\paragraph{{\bf The proof.}}
The main difficulty is showing that both regions of a
minimizing double bubble and the complement all have at most one component each. Given such connectedness, Cotton
and Freeman \cite{CF} extend the instability argument of
\cite{HMRR} to show that a minimizing double bubble must be
standard. By work of Hutchings \cite{H}, the first region in a
minimizer of volumes $v, w$ is known to be connected if a certain
``Hutchings function'' $F(v,w)$, determined implicitly by integral
formulas for volumes of spheres and standard double bubbles (see
Definition \ref{define_hutch_funct}), is positive. In their paper,
Cotton and Freeman use asymptotic analysis and intensive
computation to verify this result along most of the line $v=w$ in
$\Sth$ and all of the line $v=w$ in $\Hth$, thus proving that the standard double bubble is standard for
equal volumes. For the two-dimensional domain of volumes $(v,w)$
in our unequal volume conjecture, the proof that $F(v,w) > 0$ is
considerably harder.  Sections 3 and 4 contain the requisite
asymptotic analysis for the non-compact space $\Hth$. For small volumes, we use a 
Euclidean approximation (Section 3),
and for large volumes we obtain the interesting result
(Proposition \ref{limit_along_rays_is_positive}) that

$$\lim_{w\rightarrow\infty}  F(\psi w, w) = 2 \pi \ln \frac{4(\psi + 1)}{e^2},$$

which is positive if and only if $\psi > \lambda = e^2/4-1 < .85$. Along the line $v=\lambda w$, F is \emph{decreasing} in $w$, when $w \geq 300$.
Moreover, once $w > 150$, $\partial F / \partial v > 0$.  This asymptotic analysis reduces the problem in $\Hth$ to
a bounded domain. We are able to use a computer to show that the Hutchings function 
is positive on this bounded domain.  In $\Sth$, which is a space of finite volume,
our conjecture requires each region to contain at least 10\% of
the total volume of $\Sth$, and we are again examining a bounded domain.  Next 
we use a new balancing argument (Proposition \ref{balancing}),
which says that if $F(\frac{v+w}{2},\frac{v+w}{2})$ is positive, so is $F(v,w)$
for $v>2w$. This argument, together with symmetry (used only for $\Sth$) and 
a balancing argument due to Hutchings (Lemma \ref{hutch_balancing}),
reduces the size of the domain to check. We then break up the
domain into small rectangles (and triangles). Finally, we
decompose the Hutchings function as the difference of a concave
function and an increasing function, so that it suffices to
consider the values on (or beyond) the corners
of the rectangles (see Lemma \ref{computational_lemma}). We find values of the
parameters to locate such points and verify that $F(v,w) > 0$,
allowing a safe margin of error. If necessary, we
subdivide the rectangles and repeat the process. \\

Though using a computer introduces error, we bound this computation error and account for it.\\

A more succinct version of this paper has been submitted to the Journal of Geometric Analysis \cite{CHHLLMS}.

\bigskip

%%%%%%%%%%%%%%%%%%%%%%%%%%%%%%%%%%%%%%%%%%%%%%%%%%%%%%%%%%%%%%%%%%%%%%%%
\subsection{Acknowledgments}
%%%%%%%%%%%%%%%%%%%%%%%%%%%%%%%%%%%%%%%%%%%%%%%%%%%%%%%%%%%%%%%%%%%%%%%%

The authors would like to thank Colin Adams for posing this problem and for his insightful questions and helpful conversations. 
 We would also like to thank David Futer, Matt Kudzin, and Pat McDonald for helpful 
 conversations. Finally, we would like to thank Frank Morgan for 
 his direction and advice.\\

This research is the joint work of SMALL undergraduate research Geometry Group from '01, '02 and '03, and was completed by
Neil Hoffman (Geometry Group '03) in his undergraduate thesis work \cite{Ho}. We thank the National Science 
Foundation, Williams College, and the organizers of the SMALL REU for helping to make this research possible.

%%%%%%%%%%%%%%%%%%%%%%%%%%%%%%%%%%%%%%%%%%%%%%%%%%%%%%%%%%%%%%%%%%%%%%%%
%%%%%%%%%%%%%%%%%%%%%%%%%%%%%%%%%%%%%%%%%%%%%%%%%%%%%%%%%%%%%%%%%%%%%%%%
%%%%%%%%%%%%%%%%%%%%%%%%%%%%%%%%%%%%%%%%%%%%%%%%%%%%%%%%%%%%%%%%%%%%%%%%

\section{The Hutchings function in $\Sth$  and $\Hth$} \label{hutchings_function}

%%%%%%%%%%%%%%%%%%%%%%%%%%%%%%%%%%%%%%%%%%%%%%%%%%%%%%%%%%%%%%%%%%%%%%%%
%%%%%%%%%%%%%%%%%%%%%%%%%%%%%%%%%%%%%%%%%%%%%%%%%%%%%%%%%%%%%%%%%%%%%%%%
%%%%%%%%%%%%%%%%%%%%%%%%%%%%%%%%%%%%%%%%%%%%%%%%%%%%%%%%%%%%%%%%%%%%%%%%

In $\Sth$ and $\Hth$, given volume $v$, a round ball has least boundary area $A(v)$ (\cite{Schmidt}, \cite[p.
127]{Morgan1}). Similarly, given volumes $v$, $w$, there is an area-minimizing double bubble that encloses and separates these
two volumes; furthermore, this double bubble is comprised of smooth constant-mean-curvature hypersurfaces, except possibly for a set of measure zero
(\cite[Thm 13.4, Remark before Prop 13.8]{Morgan1}, \cite[Propostition 2.3]{CF}), and it is symetric about some geodesic (\cite[Lemma 2.9, Remark 3.8]{H}). Of course, in $\Sth$, we assume that $v,w$ satisfy $v+w<|\Sth|$, where $|\Sth|$ denotes the volume of $\Sth$. Also, we put $\bar{v},\bar{w},\bar{u}$ to be $v/|\Sth|$, $w/|\Sth|$, $u/|\Sth|$, respectively. Though the area of the area-minimizing double bubble is sometimes denoted $A(v,w)$ we reserve this notation for the area of the standard double bubble. The following argument shows $A(v,w)$ is well defined. 

\begin{lem}[Unique standard double bubble] \label{sdb}
Given volumes $v$ and $w$:
\begin{itemize}
\item Up to isometries of the space, there is a unique double bubble
in $\Sth$ consisting of three spherical caps meeting at $120^{\circ}$
that encloses and separates a region of volume $v$ and a region of
volume $w$ (whenever $v+w<1$).

\item Up to isometries of the space, there is a unique double bubble
in $\Hth$ consisting of two spheres (the outer caps) and a sphere,
hyposphere, horosphere or geodesic plane (the inner cap) meeting at
$120^{\circ}$ that encloses and separates a region of volume $v$ and a
region of volume $w$.
\end{itemize}

\end{lem}

\begin{defn}\label{define_hutch_funct}
The \emph{Hutchings function} in $\Sth$ is $F_{\Sth}: \{(v,w): \bar{v}+\bar{w} <
1\} \rightarrow \R$, where
\begin{equation}
F_{\Sth}(v,w)=A_{\Sth}(\frac{v}{2})+A_{\Sth}(w)+A_{\Sth}(v+w)-A_{\Sth}(v,w)\label{hf_in_s3}.
\end{equation}
\end{defn}

The definition of $F_{\Hth}$ is similar.

\subsection{The Hutchings component bound}

In this section, we recall for the reader the basic property of the
Hutchings function, namely that it can be used to limit the number of
components in an area-minimizing double bubble.\\

We suppress the subscripts on $F$ and $A$, since the results stated
here apply to both $\Sth$ and $\Hth$.\\

\begin{prop} \label{hutch_pos_implies_conn}
If $F(v,w)$ is positive, then in an area-minimizing double bubble
enclosing and separating volumes $v$ and $w$, the region
$V$ is connected.
\end{prop}

\begin{proof}
This is a direct consequence of  \cite[Proposition 4.8]{CF} since the area of the standard double bubble has at
least as much area as the area-minimizing double bubble. 
\end{proof}

\subsection{Area and volume formulas}

This section recalls the formulas for
area and volume of area-minimizing single bubbles (Remark
\ref{area_and_vol_for_single_bubbles}). We
have already shown the existence and uniqueness of a standard
double bubble enclosing given volumes (Proposition \ref{sdb}),
followed by Propositions \ref{Area_of_Spherical_Cap_in_Sth},
\ref{Volume_of_Spherical_Cap_in_Sth},
\ref{Volume_of_Spherical_Cap_in_Hth}, and
\ref{Area_of_Spherical_Cap_in_Hth} which set up basic formulas needed
to compute area and volume for standard double bubbles.

\paragraph{Single bubbles} \label{ssect_single_bubble_formulas}

Here we present the standard results on volume and surface area of
spheres in $\Sth$ and $\Hth$.

\begin{rem}\label{area_and_vol_for_single_bubbles}
The formulas for surface area of a sphere of radius $r$ in $\Sth$ and
$\Hth$ are
\begin{eqnarray}
A_{\Sth} &=& 4 \pi \sin^{2} r \\
A_{\Hth} &=& 4 \pi \sinh^{2} r.
\end{eqnarray}
The volume formulas for a ball of radius $r$ in $\Sth$ and $\Hth$ are
\begin{eqnarray}
V_{\Sth} &=& \pi(2r - \sin 2r) \\
V_{\Hth} &=& \pi(\sinh 2r - 2r).
\end{eqnarray}
\end{rem}

\begin{lem}
The mean curvatures, $\frac {dA}{dV}$, of spheres of radius $r$ in
$\Sth$ and $\Hth$ are $2 \cot r$ and $2 \coth r$ respectively.
\end{lem}

\begin{proof}
Volume and area are related by $\frac {dV}{dr} = A$.  Thus, in $\Sth$,
by \Rem\ \ref{area_and_vol_for_single_bubbles},
\begin{eqnarray}
\frac {dA}{dV} &=& \frac{dA/dr}{dV/dr} \\
               &=& A'(r)/A(r) \\
               &=& \frac {8\pi \sin r \cos r}{4\pi \sin ^{2} r} \\
               &=&  2 \cot r.
\end{eqnarray}
The proof in $\Hth$ is analogous.
\end{proof}

Since each standard double bubble in $\Sth$ and $\Hth$ is symmetric
about a geodesic line, it is convenient to compute area and volume of
a standard double bubble by considering the revolution of some
generating curve consisting of three circular arcs meeting at
$120^{\circ}$. Figure \ref{two_dimensional_version} represents an
arbitrary generating curve labeled with relevant measurements.

%%%%%%%%%%%%%%%%%%%%%%%%%%%%%%%%%%%%%%%%%%%%%%%%%%%%%%%%%%%%
\begin{figure}[h]
\begin{center}
\includegraphics[width= 3in]{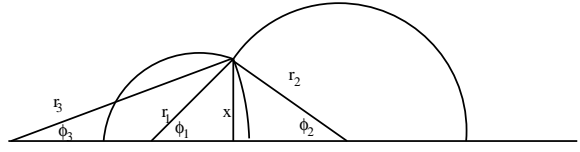}
\end{center}
\caption{\label{two_dimensional_version} Some important geometrical
features of the generating curve for a standard double bubble.}
\end{figure}
%%%%%%%%%%%%%%%%%%%%%%%%%%%%%%%%%%%%%%%%%%%%%%%%%%%%%%%%

\begin{prop}[Area of Spherical Cap in $\Sth$]\label{Area_of_Spherical_Cap_in_Sth}
In $\Sth$ the area of a spherical cap subtended by an angle
$2\phi_0$ is
\begin{equation}
\int_{0}^{2 \pi} \int_{0}^{\phi_0} \sin^2 r \sin \phi \:  d \phi d \theta =2\pi \sin^2 r ( 1 - \cos \phi_0).
\end{equation}
\end{prop}

\begin{prop}[Volume of Spherical Cap in $\Sth$]\label{Volume_of_Spherical_Cap_in_Sth}
    The volume of a spherical cap that is subtended by an angle
$2\phi_0$ is
\begin{equation}
-\pi (\tan^{-1}( \cos \phi_0 \tan r) -r \cos \phi_0 + (-1 + \cos \phi_0) (r - \cos r \sin r)).
\end{equation}
\end{prop}

\paragraph{\textbf{Hyperbolic Space}}

\begin{prop}[Volume of Spherical Cap in $\Hth$]\label{Volume_of_Spherical_Cap_in_Hth}
    The volume of a spherical cap in a sphere of radius $r$ subtended by an angle
$2\phi_0$ is
\begin{equation}
\pi (-r + \tanh^{-1}(\cos \phi_0 \tanh r)- \cos \phi_0 \cosh r \sinh r + \cosh r \sinh r))).
\end{equation}
\end{prop}

\begin{prop}[Area of Spherical Cap in $\Hth$]\label{Area_of_Spherical_Cap_in_Hth}
In $\Hth$ the area of a spherical cap in a sphere of radius $r$
subtended by an angle $2\phi_0$ is
\begin{equation}
\int_{0}^{2 \pi} \int_{0}^{\phi_0} \sinh^2 r \sin \phi  d \phi d
\theta =2\pi \sinh^2 r ( 1 - \cos \phi_0).
\end{equation}
\end{prop}

Now we can compute the volume and areas of standard double bubbles.\\

For each bubble component, we compute the volume and area of the
spherical cap which separates it from the exterior.  To find the area
of the double bubble, we simply add the area of the three spherical caps.\\

To compute the volume of the bubble components we add the volume
contained by the separating cap to that contained by the smaller of
the two outer cap to obtain the volume of the smaller region, and
subtract the volume contained by the separating cap from that
contained by the larger of the two outer cap to obtain the volume of
the larger region. (Of course, when the two outer caps are the same
size, the separating cap contains no volume so there is nothing to add
or subtract.)\\

Examining Figure \ref{two_dimensional_version}, we observe
that the angle $\phi_1$ has a different orientation (relative to $r_2$ and $r_3$) depending on the
relative size of the two bubbles. When $\phi_1$ is $90^{\circ}$, $r_2$ and $r_3$ are equal causing the volumes enclosed by both regions of the double bubble to be equal. Tracing back through the equations, we
notice that this happens when
\begin{equation}
\frac{\tan(r_2 )-\tan(r_1)}{\tan(r_2)+\tan(r_1)} = 1/3
\hspace{8 pt}
(\mathrm{in} \hspace{4 pt} \Sth)
\end{equation}
and
\begin{equation}
\frac{k_1-k_2}{k_1+k_2} = 1/3
\hspace{8 pt}
(\mathrm{in} \hspace{4 pt} \Hth).
\end{equation}
We make adjustments to $\phi_1$ accordingly, and call the adjusted
angle $\hat{\phi_1}$. Using this angle in the formulas for area and
volumes of spherical caps yields the desired result, i.e., if
$\hat{\phi_1} > \frac{\pi}{2}$, the formulas return values
corresponding to a cap smaller than a hemisphere, otherwise the values
returned describe a cap that is larger than a hemisphere.

%%%%%%%%%%%%%%%%%%%%%%%%%%%%%%%%%%%%%%%%%%%%%%%
\subsection{Properties of Hutchings function}
%%%%%%%%%%%%%%%%%%%%%%%%%%%%%%%%%%%%%%%%%%%%%%%
The concavity of $A(v)$ is well-known in both $\Sth$ and $\Hth$. Hutchings shows that the area of the 
area-minimizing double bubble enclosing volumes $v$ and $w$  is a concave function in both variables \cite[Rmk 3.8, Thm 3.9]{H}. For
our computations, the concavity of the standard double bubble is relevant. 

\begin{prop} \label{concavity_prop}
  A(v,w) in $\Sn$ and $\Hn$ $n \geq 2$ are strictly concave functions.
\end{prop}

\begin{proof} 
By uniqueness of standard double bubble (Lem \ref{sdb}), it is the only equilibrium 
and hence the minimizer among spherical surfaces meeting along a sphere. 
Given $v_0$, $w_0$, and the standard double bubble, consider the area $A_1(v,w)$ of
spherical double bubbles for nearby values of $v$ and $w$ obtained by varying the curvatures of the two
bubbles, but leaving the curvature of the interface unchanged and giving up the 120-degree-angle
condition. Along any line in $(v,w)$ space, if $v$ is strictly increasing, then the mean curvature $H_1$ of the first spherical
cap is strictly decreasing, and similarly for $w$ and $H_2$.  Hence $A_1' = H_1 dv/dt + H_2 dw/dt$ is strictly decreasing and 
$A_1$ is strictly concave. Since $A(v,w) \leq A_1 (v,w)$, 
with equality at $(v_0,w_0)$, $A(v,w)$ is strictly concave.
\end{proof}

\begin{prop} \label{double_increases_Sth}
A(v,w) in $\Sth$ is strictly increasing in $v,w$ on the closed triangular domain defined by $(0,0),(\frac{|\Sth|}{3},\frac{|\Sth|}{3}),(0, \frac{|\Sth|}{2})$.
Furthermore, A(v, w) is also increasing on the line segment with end points $(0,\frac{|\Sth|}{2}),(\frac{|\Sth|}{3},\frac{|\Sth|}{3})$.
\end{prop}

\begin{proof}
Since $A(v,w)$ is concave (\Prop\ \ref{concavity_prop})
 and symmetric in both arguments,
$A(v,w)$ in $\Sth$ attains its maximum at $(\frac{|\Sth|}{3},\frac{|\Sth|}{3})$. Since $A(v,w)$ is also concave along lines through $(\frac{|\Sth|}{3},\frac{|\Sth|}{3})$ it is increasing on the triangular domain and the line segment with end points $(0,\frac{|\Sth|}{2}),(\frac{|\Sth|}{3},\frac{|\Sth|}{3})$.
\end{proof}

\begin{prop}\label{Ah_Increasing}
For all $v,w$, $A_{\Hth}(v,w)$ is strictly increasing in each
variable.
\end{prop}

\begin{proof}
$A_{\Hth}(v,w)$ is strictly concave along any line, as shown in
Lemma \ref{concavity_prop}: in particular, along any line $v = v_0$ or
$w = w_0$. Because it is positive for all $v,w > 0$, it follows
that it can never be non-increasing along any of these lines.
\end{proof}

%%%%%%%%%%%%%%%%%%%%%%%%%%%%%%%%%%%%%%%%%%%%%%%%%%%%%%%%%%%%%%%%
\subsection{Balancing}
%%%%%%%%%%%%%%%%%%%%%%%%%%%%%%%%%%%%%%%%%%%%%%%%%%%%%%%%%%%%%%%%
We call the following useful proposition ``S-balancing'' because it relies on the concavity of area of the standard double bubble,  in
contrast to Hutchings balancing (Lem \ref{hutch_balancing}) which relies on the concavity of area of the minimizing double 
bubble.

\begin{prop}[S-balancing] \label{balancing}
In $\Rn$, $\Hn$, or $\Sn$, for $n \geq 3$, if $w<v\leq2w$ and
$F(\frac{v+w}{2}, \frac{v+w}{2})>0$, then $F(v,w)
>0$.
\end{prop}

\begin{proof}
Consider an area pair $v > w$ and suppose that
$F(\frac{v+w}{2},\frac{v+w}{2})$ is positive, i.e., that
%%%%%%%%%%%%%%%%%%%%%%%%%%%%%%%%%%%%%%%%%%%%%%%%%%%%%%%%%%%%%%%%%%%%%%%%%%%%%%%%
\begin{equation}
2A\left(\frac{v+w}{4}\right) + A\left(\frac{v+w}{2}\right) + A(v+w) -
2A\left(\frac{v+w}{2},\frac{v+w}{2}\right) > 0.
\end{equation}
%%%%%%%%%%%%%%%%%%%%%%%%%%%%%%%%%%%%%%%%%%%%%%%%%%%%%%%%%%%%%%%%%%%%%%%%%%%%%%%%
Concavity and the symmetry of $A$ implies that for any $v$ not equal to $w$,
%%%%%%%%%%%%%%%%%%%%%%%%%%%%%%%%%%%%%%%%%%%%%%%%%%%%%%%%%%%%%%%%%%%%%%%%%%%%%%%%
$$A\left(\frac{v+w}{2},\frac{v+w}{2}\right) > A(v,w).$$
%%%%%%%%%%%%%%%%%%%%%%%%%%%%%%%%%%%%%%%%%%%%%%%%%%%%%%%%%%%%%%%%%%%%%%%%%%%%%%%%
But then we have
%%%%%%%%%%%%%%%%%%%%%%%%%%%%%%%%%%%%%%%%%%%%%%%%%%%%%%%%%%%%%%%%%%%%%%%%%%%%%%%%
$$F(v,w) > 2A\left(\frac{v}{2}\right) + A(w) + A(v+w) - 2A\left(\frac{v+w}{2},\frac{v+w}{2}\right).$$
%%%%%%%%%%%%%%%%%%%%%%%%%%%%%%%%%%%%%%%%%%%%%%%%%%%%%%%%%%%%%%%%%%%%%%%%%%%%%%%%
So it suffices to show that
\begin{eqnarray*}
&& 2A\left(\frac{v}{2}\right) + A(w) + A(v+w) - 2A\left(\frac{v+w}{2},\frac{v+w}{2}\right) \\
&>& 2A\left(\frac{v+w}{4}\right) + A\left(\frac{v+w}{2}\right) + A(v+w) -
2A\left(\frac{v+w}{2},\frac{v+w}{2}\right)
\end{eqnarray*}
i.e. that
%%%%%%%%%%%%%%%%%%%%%%%%%%%%%%%%%%%%%%%%%%%%%%%%%%%%%%%%%%%%%%%%%%%%%%%%%%%%%%%%
\begin{equation} \label{suf_condition_in_balancing}
2A\left(\frac{v}{2}\right) + A(w) > 2A\left(\frac{v+w}{4}\right) + A\left(\frac{v+w}{2}\right).
\end{equation}
%%%%%%%%%%%%%%%%%%%%%%%%%%%%%%%%%%%%%%%%%%%%%%%%%%%%%%%%%%%%%%%%%%%%%%%%%%%%%%%%

Consider three spheres, two enclosing volume $x/2 < y$ and one
enclosing volume $y$. Initially, increasing $x$ and decreasing $y$
while keeping $x+y$ constant will increase the total surface area
(we know this because $2A(x/2)+A(y(x))$ is
a concave function). Indeed, area will increase until $x/2 = y$,
that is, until all three spheres have equal volume and any further
transition would be \emph{un}balancing.  In particular, if $w < v
< 2w$ and we put $x=y=(v+w)/2$, then increasing $x$ until it
equals $v$ while decreasing $y$ until it equals $w$ will increase
the total surface area of our three spheres. This shows
(Equation \ref{suf_condition_in_balancing}).\\

Hence $F(v,w) > 0$ for $w < v \leq 2w$.  \\

\end{proof}

\begin{lem}[Hutchings Balancing]\label{hutch_balancing}
In $\Rn$, $\Hn$, or $\Sn$, for $n \geq 3$, if $v>2w$ (or in $\Sth$ if $v>2u$), then the region of volume $v$ is connected.
\end{lem}

\begin{proof}
See \cite[Remark 3.8, Theorem 3.5, Corollary 3.9]{H}.
\end{proof}

%%%%%%%%%%%%%%%%%%%%%%%%%%%%%%%%%%%%%%%%%%%%%%%
\subsection{Permutation}
%%%%%%%%%%%%%%%%%%%%%%%%%%%%%%%%%%%%%%%%%%%%%%%

Permutation utilizes symmetries that exist with the Hutchings function
when applied in compact spaces. Though this argument is applied in
$\Sth$, it can be used in any space with finite volume that has a
Hutchings function.

\begin{prop}[Permutation] \label {permutation} $F(v,w)=F(v,u)$ where $u$ is the volume of the complement.
\end{prop}

\begin{proof}
First notice that $A(v,w) = A(v,u)$ as it is only a matter of labeling
$U$ or $W$ the exterior.  Now, $A(v + w) = A(u)$ and $A(v + u) =
A(w)$, since $U= (V \cup W)^{\complement}$, and similarly, $W= (V \cup
U)^{\complement}$.  Thus, $F(v, w) = 2A(\frac{v}{2}) + A(w) + A(v + w)
- 2A(v, w) = 2A(\frac{v}{2}) + A(w) + A(u) - 2A(v, w) = F(v, u)$, as
desired.
\end{proof}

\begin{rem}\label{cf_answer}
 S-balancing (Prop \ref{balancing}), Hutchings balancing (Lem \ref{hutch_balancing}), and permutation (Prop \ref{permutation}) yield a complete answer to a question of Cotton and Freeman \cite[Open Question 1]{CF}. 
 In particular, the larger region of an area-minimizing double bubble in $\Hth$ is always connected. If $v>2w$, this follows from Hutchings balancing (Lem \ref{hutch_balancing}). Otherwise, connectedness follows from the positivity of the Hutchings
function $F(v,w)$ (Prop \ref{hutch_pos_implies_conn}).  This in turn follows from S-balancing (Prop \ref{balancing}) and the positivity of the Hutchings function $F(v,v)$  for equal volumes  \cite[Prop 5.11, 5.14,5.19] {CF}. In $\Sth$, Hutchings balancing (Lem  \ref{hutch_balancing}) shows connectedness when $v>2w$ or $v>2u$. By Proposition \ref{hutch_pos_implies_conn}, it suffices to show that $F(v,w)>0$ if $w<v\leq2w$ or $u<v\leq2u$. If  $\bar{v} \leq .4$, $F(v,v)>0$ (\cite[Prop 5.5, 5.8]{CF}). S-balancing (Prop \ref{balancing}) shows that $F(v,w)>0$ if $w < v \leq 2w$. Finally, permutation (Prop \ref{permutation}) shows $F(v,w)>0$ if $u< v \leq 2u$.  
\end{rem}

%%%%%%%%%%%%%%%%%%%%%%%%%%%%%%%%%%%%%%%%%%%%%%%%%%%%%%%%%%%%%%%
%%%%%%%%%%%%%%%%%%%%%%%%%%%%%%%%%%%%%%%%%%%%%%%%%%%%%%%%%%%%%%%
%%%%%%%%%%%%%%%%%%%%%%%%%%%%%%%%%%%%%%%%%%%%%%%%%%%%%%%%%%%%%%%
\section{Positivity of the Hutchings function for small volumes in
$\Hth$} \label{small_volumes}
%%%%%%%%%%%%%%%%%%%%%%%%%%%%%%%%%%%%%%%%%%%%%%%%%%%%%%%%%%%%%%%
%%%%%%%%%%%%%%%%%%%%%%%%%%%%%%%%%%%%%%%%%%%%%%%%%%%%%%%%%%%%%%%
%%%%%%%%%%%%%%%%%%%%%%%%%%%%%%%%%%%%%%%%%%%%%%%%%%%%%%%%%%%%%%%

\subsection{Properties of the Hutchings function in $\Hth$}

\begin{thm}[Double Bubble Theorem in $\Rth$ \cite{HMRR}]\label{DB_in_Rth}
The unique least-area enclosure of prescribed volumes $v,w$ in
$\Rth$ is a standard double bubble.
\end{thm}

\begin{rem}
For small prescribed volumes in $\Hth$, the space around any
standard double bubble is nearly flat and looks like a portion of
$\Rth$. By estimating the amount of distortion carefully, we can
use information about the Hutchings function in $\Rth$ to prove
that the Hutchings function in $\Hth$ is positive for certain
prescribed volumes near $0$. In the following proof, $\lambda$ is
a constant which captures the amount of distortion, and $r$
represents the radius of a small ball in $\Hth$ which is only slightly affected
by the curvature of $\Hth$.
\end{rem}

\begin{lem}\label{small_volumes_can_fit_in_sdb_in_small_sphere}
For any pair of prescribed volumes less than $(\frac{2\pi}{3} +\frac{\pi\sqrt{3}}{2})(\frac{2}{2+\sqrt{3}} r)^3$,
some minimizing double bubble in $\Rth$ enclosing those volumes
fits inside a ball of radius $r$.
\end{lem}

\begin{proof}

Without loss of generality assume that $v \leq w$ and fix $w$. The exterior spherical cap enclosing the region of volume $w$ has a radius, $r_w$. Consider standard double bubbles where the radius of this spherical cap
is fixed at $r_w$. In this case, the equal volumes double bubble has the greatest diameter. However, the region enclosed by the spherical cap with radius $r_w$ has volume less than
$w$. By increasing the radius until this region has volume $w$, the diameter increases. Thus for any non-equal volume double bubble enclosing volumes $v$ and $w$,
where $v \leq w$, there will be an equal volume double bubble enclosing $v'$ and $w$ with a bigger diameter. \\

 If $r_w$ is the radius of the spherical cap that
forms the exterior of $W$, then the diameter of the equal volume double bubble is $2 r_w + {r_w} \sqrt{3}$. So, this double bubble will fit inside a ball of radius $r > r_w \frac{( 2 + \sqrt{3} )}{2}$.\\

We also know that in $\Rth$ a spherical cap of radius $r_0$ subtended by an angle $2 \phi$ has volume:    \\

%%%%%%%%%%%%%%%%%%%%%%%%%%%%%%%%%%%%%%%%%%%%%%%%%%%%%%%%%%%%%%%%%%%
\begin{equation}
V(r) =  {\frac {2 \pi } {3} } {(r_0)}^3 + \pi {r_0}^3 cos \phi.
\end{equation}
%%%%%%%%%%%%%%%%%%%%%%%%%%%%%%%%%%%%%%%%%%%%%%%%%%%%%%%%%%%%%%%%%%%
Since $ \phi \in [0, \frac{\pi}{6}] $, \\
%%%%%%%%%%%%%%%%%%%%%%%%%%%%%%%%%%%%%%%%%%%%%%%%%%%%%%%%%%%%%%%%%%%
\begin{equation}
w \leq {(r_w)}^3 (\frac{2 \pi}{3} + \frac{\pi \sqrt{3}}{2}).
\end{equation}
%%%%%%%%%%%%%%%%%%%%%%%%%%%%%%%%%%%%%%%%%%%%%%%%%%%%%%%%%%%%%%%%%%%

Substituting $\frac{2}{2+\sqrt{3}} r$ for $r_w$ in the above equation gives the desired result.

\end{proof}

\begin{lem} \label{HMRRStrong} If $.84w \le v \le w$, then
\[
2A_{\Rth}(\frac{v}{2}) + A_{\Rth}(w) + A_{\Rth}(v+w) >
(2.02676)A_{\Rth}(v,w)
\]
holds.
\end{lem}

\begin{proof}
(Modification of  \cite[Proposition 6.2]{HMRR}) By rescaling, we may assume
that $v = 1 - w$ and that $w \in [\frac{1}{2}, \frac{1}{1.84}]$.
Because $A_{\Rth}(v,w)$ is concave \cite[Theorem  3.2]{H}, we
have
\[
\textstyle A_{\Rth}(1-w,w) \le A_{\Rth}(\frac{1}{2},\frac{1}{2})
 = (2^{-\frac{4}{3}} \cdot
3)A_{\Rth}(1).
\]

Hence, it suffices to prove that
\[
\textstyle 2A_{\Rth}(\frac{1-w}{2}) + A_{\Rth}(w) + A_{\Rth}
\ge 2.02676 (2^{-\frac{4}{3}} \cdot 3)A_{\Rth}(1),
\]
or equivalently (dividing by $A_{\Rth}(1)$ on both sides) that
\[
\textstyle 2^{\frac{1}{3}} (1-w)^{\frac{2}{3}} + w^{\frac{2}{3}} +
1 \ge 2.02676 (2^{-\frac{4}{3}} \cdot 3).
\]
At $w = \frac{1}{2}$ and $w = \frac{1}{1.84}$, the left hand side
of the above inequality is bigger than $2.4236$ and
$2.412965$, respectively. The right hand side of the above
inequality is less than $2.412966$. Hence, the inequality
holds for $w = \frac{1}{2}$ and $w = \frac{1}{1.84}$; and because
the left hand side is concave in $w$, the inequality holds for all
$w \in [\frac{1}{2}, \frac{1}{1.84}]$, as desired.

\end{proof}

\begin{prop}\label{hth_small_volumes}
 If $0 < v,w < .002743 $   and $.84w \leq v \leq w$, then
 
%%%%%%%%%%%%%%%%%%%%%%%%%%%%%%%%%%%%%%%%%%%%%%%%%%%%%%%%%%%%%%%%%%%
\begin{equation}
2A_{\Hth}(\frac{v}{2}) + A_{\Hth}(w) + A_{\Hth}(v + w) - 2A_{\Hth}(v,w) > 0.
\end{equation}
%%%%%%%%%%%%%%%%%%%%%%%%%%%%%%%%%%%%%%%%%%%%%%%%%%%%%%%%%%%%%%%%%%%%

\end{prop}

\begin{proof}

By Lemma \ref{small_volumes_can_fit_in_sdb_in_small_sphere}, we can fit a double bubble in $\Rth$ 
with prescribed volumes $v, w < . 002743 $ in a ball $B$ of radius $r < .1547$.\\

For ease of notation let $\lambda = \frac{\sinh r}{r} < 1.003994$.  \\

 Notice that if this ball
can enclose a standard double bubble of volumes $(\lambda^{-2} v, \lambda^{-2} w)$
 then it can enclose spheres of volumes:\\
$\{\lambda^{-2} \frac{v}{2}, \lambda^{-2} w,\lambda^{-2} (v+ w) \}$. 
 We may estimate the surface area of any sphere in
$\Rth$ which fits inside $B$. Suppose that a sphere enclosing
volume $\lambda^{-2} \alpha\ $ fits inside $B$; in particular, this holds
for  $ \alpha \in \{ \frac{v}{2}, w, (v+ w)\}$.  The single bubble in $\Hth$
encloses a volume $\alpha' \leq \alpha$ and has
surface area at least $A_{\Rth} (\lambda^{-2} \alpha)$. Thus,
%%%%%%%%%%%%%%%%%%%%%%%%%%%%%%%%%%%%%%%%%%%%%%%%%%%%%%%%%%%%%%%%%%%
\begin{equation}\label{ieq1}
A_{\Hth}(\alpha) \geq A_{\Rth}(\lambda^{-2} \alpha).
\end{equation}
%%%%%%%%%%%%%%%%%%%%%%%%%%%%%%%%%%%%%%%%%%%%%%%%%%%%%%%%%%%%%%%%%%%
Therefore,
%%%%%%%%%%%%%%%%%%%%%%%%%%%%%%%%%%%%%%%%%%%%%%%%%%%%%%%%%%%%%%%%%%%
\begin{eqnarray*}\label{ieq2}
&&2A_{\Hth}(\frac{v}{2}) + A_{\Hth}(w) + A_{\Hth}(v + w)\\
&& \geq 2A_{\Rth}(\lambda^{-2} v
 /2) + A_{\Rth}( \lambda^{-2} w) + A_{\Rth}( \lambda^{-2} (v+ w ) ).
\end{eqnarray*}
%%%%%%%%%%%%%%%%%%%%%%%%%%%%%%%%%%%%%%%%%%%%%%%%%%%%%%%%%%%%%%%%%%%
By Lemma \ref{HMRRStrong},  
%%%%%%%%%%%%%%%%%%%%%%%%%%%%%%%%%%%%%%%%%%%%%%%%%%%%%%%%%%%%%%%%%%%
\begin{eqnarray*}\label{ieq3}
&&2A_{\Rth}( \frac{\lambda^{-2} v }{2}) + A_{\Rth}( \lambda^{-2} w) + A_{\Rth}( \lambda^{-2} (v
+ w))\\
&& \geq (2.02676)A_{\Rth}( \lambda^{-2} v, \lambda^{-2} w).
\end{eqnarray*}
%%%%%%%%%%%%%%%%%%%%%%%%%%%%%%%%%%%%%%%%%%%%%%%%%%%%%%%%%%%%%%%%%%
Scaling of double bubbles in $\Rth$ tells us
%%%%%%%%%%%%%%%%%%%%%%%%%%%%%%%%%%%%%%%%%%%%%%%%%%%%%%%%%%%%%%%%%%%
\begin{equation}\label{ieq4}
A_{\Rth}( \lambda^{-2} v, \lambda^{-2} w) = (\lambda)^{-4/3}A_{\Rth}(v,w).
\end{equation}
%%%%%%%%%%%%%%%%%%%%%%%%%%%%%%%%%%%%%%%%%%%%%%%%%%%%%%%%%%%%%%%%%%%
Because $B$ contains some minimizing double bubble
$\Sigma_1$ in $\Rth$ enclosing volumes $v,w$, we can estimate the surface area of the 
image of $\Sigma_1$ in $\Hth$.  This double bubble $\Sigma_2$ encloses volumes $v_2
\geq v$ and $w_2 \geq w$, with surface area less than or equal to
$A_{\Rth}(v_2,w_2)\lambda^{2}$. Therefore,
%%%%%%%%%%%%%%%%%%%%%%%%%%%%%%%%%%%%%%%%%%%%%%%%%%%%%%%%%%%%%%%%%%%
\begin{equation}\label{ieq5}
A_{\Rth}(v,w)\geq \lambda^{-2} A_{\Hth}(v_2,w_2).
\end{equation}
%%%%%%%%%%%%%%%%%%%%%%%%%%%%%%%%%%%%%%%%%%%%%%%%%%%%%%%%%%%%%%%%%%%
By \Prop\ \ref{Ah_Increasing}, $A_{\Hth}(x,y)$ is increasing in
each variable, implying that
%%%%%%%%%%%%%%%%%%%%%%%%%%%%%%%%%%%%%%%%%%%%%%%%%%%%%%%%%%%%%%%%%%%
\begin{equation}\label{ieq6}
A_{\Hth}(v_2,w_2)\geq A_{\Hth}(v,w).
\end{equation}
%%%%%%%%%%%%%%%%%%%%%%%%%%%%%%%%%%%%%%%%%%%%%%%%%%%%%%%%%%%%%%%%%%%
Combining inequalities (\ref{ieq2}), (\ref{ieq3}), 
(\ref{ieq5}), and (\ref{ieq6}) with equality (\ref{ieq4}) gives
%%%%%%%%%%%%%%%%%%%%%%%%%%%%%%%%%%%%%%%%%%%%%%%%%%%%%%%%%%%%%%%%%%%
\begin{equation}
2A_{\Hth}(\frac{v}{2}) + A_{\Hth}(w) + A_{\Hth}(v + w) \geq (2.0676)(\lambda)^{-10/3} A_{\Hth}(v,w).
\end{equation}
%%%%%%%%%%%%%%%%%%%%%%%%%%%%%%%%%%%%%%%%%%%%%%%%%%%%%%%%%%%%%%%%%%%
 Note that $(2.02676)\lambda^{-10/3}>2$, so
%%%%%%%%%%%%%%%%%%%%%%%%%%%%%%%%%%%%%%%%%%%%%%%%%%%%%%%%%%%%%%%%%%%
\begin{equation}
2A_{\Hth}(\frac{v}{2}) + A_{\Hth}(w) + A_{\Hth}(v + w) > 2A_{\Hth}(v,w)
\end{equation}
%%%%%%%%%%%%%%%%%%%%%%%%%%%%%%%%%%%%%%%%%%%%%%%%%%%%%%%%%%%%%%%%%%%
as desired.

\end{proof}

\begin{rem}[Notes on $\Sth$]
The above argument works for $\Sth$ as well, although we do not need it
 because we are just concerned with volumes greater than 10\% of the
total volume of $\Sth$. The $\Sth$ proof uses a different distortion factor,
$\lambda = \frac{\sin r}{r}$. Otherwise the proof follows with very little adjustment. Using Lemma \ref{HMRRStrong}, the 
Hutchings Function is positive for small volumes in $\Sth$ namely $v, w < 0.002738$ and $.84w \leq v \leq w$.

\end{rem}
%%%%%%%%%%%%%%%%%%%%%%%%%%%%%%%%%%%%%%%%%%%%%%%%%%%%%%%%%%%%%%%
%%%%%%%%%%%%%%%%%%%%%%%%%%%%%%%%%%%%%%%%%%%%%%%%%%%%%%%%%%%%%%%
%%%%%%%%%%%%%%%%%%%%%%%%%%%%%%%%%%%%%%%%%%%%%%%%%%%%%%%%%%%%%%%
\section{Positivity of the Hutchings function for large volumes
in $\Hth$} \label{large_volumes}
%%%%%%%%%%%%%%%%%%%%%%%%%%%%%%%%%%%%%%%%%%%%%%%%%%%%%%%%%%%%%%%
%%%%%%%%%%%%%%%%%%%%%%%%%%%%%%%%%%%%%%%%%%%%%%%%%%%%%%%%%%%%%%%
%%%%%%%%%%%%%%%%%%%%%%%%%%%%%%%%%%%%%%%%%%%%%%%%%%%%%%%%%%%%%%%
\subsection{Introduction}

Theorem \ref{hf_positive_for_large_volumes} shows that for sufficiently large $v$ and $w$, the Hutchings function in $\Hth$ is positive whenever $\psi = v/w \geq \lambda $ and $w \geq 300$. Throughout 
this section we set $\lambda = e^2/4-1$. The proof of Theorem \ref{hf_positive_for_large_volumes} uses three ancillary results:
%%%%%%%%%%%%%%%%%%%%%%%%%%%%%%%%%%%%%%%%%%%%%%%%%%%%%%%%%%%%%%%%%%%
\begin{list}{}{}
%------------------------------------------------------------------*
\item[Proposition \ref{limit_along_rays_is_positive}]. \emph{For a fixed ratio $v/w$, the limit of the
 Hutchings function as $w\rightarrow\infty$ is nonnegative
if and only if $\psi \geq \lambda < .85$.}
%------------------------------------------------------------------*
\item[Proposition \ref{hf_decr_along_crit_line_when_w_large}]. \emph{The Hutchings function strictly decreases if one
travels outward along the line $v = \lambda w$ as long as $w >
300$.}
%------------------------------------------------------------------*
\item[Lemma \ref{derivative_with_respect_to_v_is_positive}]. \emph{The partial derivative with respect to $v$ of the Hutchings function is positive if $w \geq 150,$ $v \geq \lambda w$.}
%------------------------------------------------------------------*
\end{list}
%%%%%%%%%%%%%%%%%%%%%%%%%%%%%%%%%%%%%%%%%%%%%%%%%%%%%%%%%%%%%%%%%%%
Together, these results show that the Hutchings function is positive for $\psi \geq
\lambda $ and $w > 300$. In order to prove Proposition \ref{limit_along_rays_is_positive}, we must describe the limits of various
quantities for standard double bubbles enclosing volumes $v,w$ as
$v,w$ grow large. Proving Proposition \ref{hf_decr_along_crit_line_when_w_large} and Lemma \ref{derivative_with_respect_to_v_is_positive} require a closer examination of how fast these quantities converge to their limits.  So, the results about
the limits appear alongside various inexact numerical estimates.

\subsection{Preliminary material: spheres in $\Hth$}

If $r = r(v)$ is the radius of a sphere with volume $v$,  then
%%%%%%%%%%%%%%%%%%%%%%%%%%%%%%%%%%%%%%%%%%%%%%%%%%%%%%%%%%%%%%%%%%%
\begin{equation}
2v/\pi = e^{2r} - e^{-2r} - 4r < e^{2r}
\end{equation}
%%%%%%%%%%%%%%%%%%%%%%%%%%%%%%%%%%%%%%%%%%%%%%%%%%%%%%%%%%%%%%%%%%%
so that
%%%%%%%%%%%%%%%%%%%%%%%%%%%%%%%%%%%%%%%%%%%%%%%%%%%%%%%%%%%%%%%%%%%
\begin{equation}
r >\frac{1}{2} \ln \frac{2v}{\pi}  \label{lower_bound_on_radius}.
\end{equation}
%%%%%%%%%%%%%%%%%%%%%%%%%%%%%%%%%%%%%%%%%%%%%%%%%%%%%%%%%%%%%%%%%%%

\begin{lem} \label{limit_of_radius_in_hth} 
If $r = r(v)$ is the radius of a sphere with volume $v$, then
%%%%%%%%%%%%%%%%%%%%%%%%%%%%%%%%%%%%%%%%%%%%%%%%%%%%%%%%%%%%%%%%%%%
\begin{equation}
\lim_{v\rightarrow\infty} \left( r -  \frac{1}{2} \ln \frac{2v}{\pi}\right) =
0.
\end{equation}
%%%%%%%%%%%%%%%%%%%%%%%%%%%%%%%%%%%%%%%%%%%%%%%%%%%%%%%%%%%%%%%%%%%
\end{lem}

\begin{proof}
%%%%%%%%%%%%%%%%%%%%%%%%%%%%%%%%%%%%%%%%%%%%%%%%%%%%%%%%%%%%%%%%%%%
By manipulation of the volume formula,
%%%%%%%%%%%%%%%%%%%%%%%%%%%%%%%%%%%%%%%%%%%%%%%%%%%%%%%%%%%%%%%%%%%
\begin{equation}
\frac{2v}{\pi} (e^{2(r-\frac{1}{2}\ln \frac{2v}{\pi})} - 1 ) =
e^{-2r} + 4r.
\end{equation}
%%%%%%%%%%%%%%%%%%%%%%%%%%%%%%%%%%%%%%%%%%%%%%%%%%%%%%%%%%%%%%%%%%%

Notice for any $\epsilon > 0$ we can find a sufficiently large $r$, such that, 
%%%%%%%%%%%%%%%%%%%%%%%%%%%%%%%%%%%%%%%%%%%%%%%%%%%%%%%%%%%%%%%%%%%
\begin{equation}
e^{-2r} + 4r < 2(\sinh 2r - 2r)(e^{2\epsilon} - 1)  \label{lve1}.
\end{equation}
%%%%%%%%%%%%%%%%%%%%%%%%%%%%%%%%%%%%%%%%%%%%%%%%%%%%%%%%%%%%%%%%%%%

Hence, $\frac{2v}{\pi} (e^{2(r-\frac{1}{2}\ln \frac{2v}{\pi})} - 1)$  is less than
 $ \frac{2v}{\pi}( e^{2\epsilon} - 1)$ and
$2(r-\frac{1}{2}\ln \frac{2v}{\pi}) < 2\epsilon,$ as needed.
\end{proof}

\begin{rem} \label{upper_bound_on_r}
When $v > 150 \lambda > \pi(\sinh (4.4) - 4.4)$, we have $r >
2.2$ and Inequality (\ref{lve1}) holds for $\epsilon= 0.06$. Thus,
$r < \frac{1}{2} \ln \frac{2v}{\pi}+.06$ for such $v$.  We obtain
the following numerical estimate: if $v > 150\lambda < 127.09$,
then $r < \frac{1}{2} \ln \frac{2v}{\pi}+.06$.
\end{rem}

\begin{lem} \label{limiting_value_of_surface_of_sphere} %2
As $v$ approaches infinity, $A(v)$ tends to $2v + 2\pi \ln v - 2 \pi (1-
\ln(\pi/2))$.
\end{lem}

\begin{proof}
Let $r = r(v)$ denote the radius of a sphere with volume $v$. Since
%%%%%%%%%%%%%%%%%%%%%%%%%%%%%%%%%%%%%%%%%%%%%
\begin{equation}
r = \frac{1}{2} \ln \left( \frac{2v}{\pi} + e^{-2r} + 4r \right),
\end{equation}
%%%%%%%%%%%%%%%%%%%%%%%%%%%%%%%%%%%%%%%%%%%
we see that,
%%%%%%%%%%%%%%%%%%%%%%%%%%%%%%%%%%%%%%%%%%%%%%%%%%%%%%%%%%%%%%%%%%%
\begin{equation}
A(v) = 2v - 2\pi + 2\pi e^{-2r} + 4\pi r.
\end{equation}
%%%%%%%%%%%%%%%%%%%%%%%%%%%%%%%%%%%%%%%%%%%%%%%%%%%%%%%%%%%%%%%%%%%
By Lemma \ref{limit_of_radius_in_hth},
%%%%%%%%%%%%%%%%%%%%%%%%%%%%%%%%%%%%%%%%%%%%%%%%%%%%%%%%%%%%%%%%%%%
\begin{equation}
\lim_{v \rightarrow \infty} r - \frac{1}{2}\ln \frac{2v}{\pi} =
\lim_{v \rightarrow \infty} r - \frac{1}{2}\ln v + \frac{1}{2}\ln \frac{\pi}{2} = 0.
\end{equation}

%%%%%%%%%%%%%%%%%%%%%%%%%%%%%%%%%%%%%%%%%%%%%%%%%%%%%%%%%%%%%%%%%%%
Also, $\lim_{v\rightarrow \infty}e^{-2r}=0$. The desired result
follows easily.
\end{proof}

\begin{lem} \label{curvature_of_a_sphere} %8
The curvature of a sphere with volume $v$ and radius $r$
is
%%%%%%%%%%%%%%%%%%%%%%%%%%%%%%%%%%%%%%%%%%%%%%%%%%%%%%%%%%%%%%%%%%%
\begin{equation}
2+\frac{2\pi}{v+2\pi r +\frac{\pi}{2}e^{-2r}-\pi/2}.
\end{equation}
%%%%%%%%%%%%%%%%%%%%%%%%%%%%%%%%%%%%%%%%%%%%%%%%%%%%%%%%%%%%%%%%%%%
\end{lem}

\begin{proof}
The curvature equals
%%%%%%%%%%%%%%%%%%%%%%%%%%%%%%%%%%%%%%%%%%%%%%%%%%%%%%%%%%%%%%%%%%%
\begin{equation}
2\coth r = 2 \frac{e^r+e^{-r}}{e^r-e^{-r}} = 2 +
\frac{4e^{-r}}{e^r-e^{-r}}= 2 + \frac{4}{e^{2r}-1}.
\end{equation}
%%%%%%%%%%%%%%%%%%%%%%%%%%%%%%%%%%%%%%%%%%%%%%%%%%%%%%%%%%%%%%%%%%%
Also, $2v/\pi = e^{2r} - e^{-2r} - 4r$. The desired result follows
easily.
\end{proof}

\begin{lem} \label{upper_bound_on_change_in_area_for_spheres}  %12
$A'(x) < 2 + \frac{2\pi}{x+\pi \ln x -3}$.
\end{lem}

\begin{proof}

From (\ref{lower_bound_on_radius}) we have the following bound on the
radius $r$ of a sphere containing volume $x$:
%%%%%%%%%%%%%%%%%%%%%%%%%%%%%%%%%%%%%%%%%%%%%%%%%%%%%%%%%%%%%%%%%%%
\begin{eqnarray}
r &>& \frac{1}{2} \ln (2x/\pi) \\
  &=& \frac{1}{2} \ln x + (1/2)\ln (2/\pi) \\
  &>& \frac{1}{2} \ln x - 3/2\pi+1/4.
\end{eqnarray}
%%%%%%%%%%%%%%%%%%%%%%%%%%%%%%%%%%%%%%%%%%%%%%%%%%%%%%%%%%%%%%%%%%%
By Lemma \ref{curvature_of_a_sphere},
%%%%%%%%%%%%%%%%%%%%%%%%%%%%%%%%%%%%%%%%%%%%%%%%%%%%%%%%%%%%%%%%%%%
\begin{equation}
A'(x) = 2 + \frac{2\pi}{x+2\pi r + (\pi/2)e^{-2r}-\pi/2}.
\end{equation}
%%%%%%%%%%%%%%%%%%%%%%%%%%%%%%%%%%%%%%%%%%%%%%%%%%%%%%%%%%%%%%%%%%%
Also, $e^{-2r} > 0$. Applying these inequalities in the above
expression for $A'(x)$ yields the desired result.
\end{proof}

\begin{lem} \label{lower_bound_on_change_in_area_for_large_spheres_w_150} %9
If $x > 150\lambda $ then $A'(x ) > 2 + \frac{2\pi}{x + \pi
\ln x -1.041}$.
\end{lem}

\begin{proof}
By Lemma \ref{curvature_of_a_sphere},
 \begin{equation}
 A'(v) = 2 + \frac{2\pi}{v + 2\pi r + \frac{\pi}{2}e^{-2r} -
 \frac{\pi}{2}}.
 \end{equation}

 The denominator is a sum of three terms:
\[
 v, 2 \pi r, \frac{\pi}{2}e^{-2r} - \frac{\pi}{2}.
 \]

Because $e^{-2r} < 1$, the third term is less than 0. We now
approximate the second term. By Remark \ref{upper_bound_on_r}, we
have $r < \frac{1}{2} \ln \frac{2v}{\pi} + 0.06$ because $v
> 150 \lambda$. Hence, the second term satisfies
 \begin{equation}
 2 \pi r < \pi \ln \frac{2v}{\pi} + 0.12\pi = \pi \ln v + \pi \ln
 \frac{2}{\pi} + 0.12\pi < \pi \ln v - 1.041.
 \end{equation}

 With these approximations, we find
 \begin{equation}
 A'(v) > 2 + \frac{2\pi}{v + \pi \ln v - 1.041}.
 \end{equation}

  \end{proof}

\begin{cor} \label{lower_bound_on_change_in_area_for_large_spheres} 
If $x > 300\lambda$ then $A'(x + 3) > 2 + \frac{2\pi}{x + \pi
\ln x + 2}$.
\end{cor}

\begin{proof}
By Lemma \ref{lower_bound_on_change_in_area_for_large_spheres_w_150},  
$A'(x ) > 2 + \frac{2\pi}{x + \pi
\ln x -1.041}$, if $x > 150\lambda $.\\

Because $e^{-2r} < 1$, the third term is less than 0. We now
approximate the second term. By Remark \ref{upper_bound_on_r}, we
have $r < \frac{1}{2} \ln \frac{2v}{\pi} + 0.06$ because $v
> 150\lambda $. Hence, the second term satisfies
 \begin{equation}
 2 \pi r < \pi \ln \frac{2v}{\pi} + 0.12\pi = \pi \ln v + \pi \ln
 \frac{2}{\pi} + 0.12\pi < \pi \ln v - 1.041.
 \end{equation}

 With these approximations, we find
 \begin{equation}
 A'(v) > 2 + \frac{2\pi}{v + \pi \ln v - 1.041}.
 \end{equation}

 For $x > 300\lambda $, we have
 \begin{equation}
 \pi\ln\frac{x+3}{x} - 1.041 < -1.004.
 \end{equation}
 
 Hence,
 \begin{equation}
 A'(x+3) > 2 + \frac{2\pi}{x + 3 + \pi\ln x - 1.004} > 2 +
 \frac{2\pi}{x + \pi \ln x + 2}.
 \end{equation}

 \end{proof}

\subsection{The Hutchings function has a positive limit}

\begin{lem} \label{upper_bound_on_rad_circ_interface} 
Given a standard double bubble, the radius of the circular
interface where the three caps meet is less than $\cosh^{-1}(2)$.
\end{lem}

\begin{proof}
Consider the disc whose boundary is the circular interface. Let $y$ be the radius of this disk. It
meets the bigger of the outer caps of the double bubble at an angle
$\alpha\geq 2\pi/3$. The radii of this outer cap and  the disk come together at an angle of $\beta \geq \pi /6$ along the circular interface. Let $2 \phi_o$ be the angle that subtends this cap.  By the formula for the area of a spherical cap in $\Hth$ (Proposition \ref{Area_of_Spherical_Cap_in_Hth}) and the hyperbolic laws of cosines and sines \cite{Thurston1} and \cite{Stahl},

%%%%%%%%%%%%%%%%%%%%%%%%%%%%%%%%%%%%%%%%%%%%%%%%%%%%%%%%%%%%%%%%%%%
\begin{equation}
\cos \phi_o = -\cosh y \sin \beta
\end{equation}
%%%%%%%%%%%%%%%%%%%%%%%%%%%%%%%%%%%%%%%%%%%%%%%%%%%%%%%%%%%%%%%%%%%

%%%%%%%%%%%%%%%%%%%%%%%%%%%%%%%%%%%%%%%%%%%%%%%%%%%%%%%%%%%%%%%%%%%
\begin{equation}
\sinh r = \frac{\sinh y}{\sin \phi_o} 
\end{equation}
%%%%%%%%%%%%%%%%%%%%%%%%%%%%%%%%%%%%%%%%%%%%%%%%%%%%%%%%%%%%%%%%%%%

the surface area of the outer cap is 

%%%%%%%%%%%%%%%%%%%%%%%%%%%%%%%%%%%%%%%%%%%%%%%%%%%%%%%%%%%%%%%%%%%
\begin{equation}
\frac{2 \pi \sinh^2 y}{1 - \cosh y \sin \beta}.
\end{equation}
%%%%%%%%%%%%%%%%%%%%%%%%%%%%%%%%%%%%%%%%%%%%%%%%%%%%%%%%%%%%%%%%%%%

The surface are of this cap is positive and finite, implying that $\cosh y < \csc
\beta \leq \csc(\pi/6) = 2$.  Because $t \mapsto \cosh t$ is
increasing for $t > 0$, $y < \cosh^{-1}(2)$.

\end{proof}

\begin{lem}\label{alt_volume_formula}
Given a spherical cap subtended by an angle of $2\phi$ on a sphere of radius $r$, let $\theta$ be the angle the sphere makes
with the disk of radius $y$ bounding the cap from below. The volume of this cap can be expressed as,

\begin{equation}
\pi \left(\frac{\sinh y \sin \theta}{ \sech y + \cos \theta}
 - \tanh^{-1}\left(\frac{\sinh y \sin \theta}{\cosh y + \cos \theta}\right)\right).
\end{equation}

\end{lem}

\begin{proof}
This follows from using the identities:

\begin{center}
$\tanh(a+b) = \frac{\tanh a + \tanh b}{1 + \tanh a \tanh b}$ where $(a+b) = (-r + \tanh^{-1}(\tanh r \cos \phi))$,\\

$1- \tanh^2 x = \sech^2 x$ and $\cosh b = \frac{\sin \phi}{\cos \gamma}$, \\

$\cosh c = \cosh y \cosh b$,\\

$\cos \gamma = \sin \phi \cosh b$,\\ 

$\tanh r \cos \gamma = \tanh y$, $\cosh y \sin \gamma= \cos \phi$,\\  

$\sinh y = \sin \phi \sinh r$ and $\cos \gamma = \frac{\tanh y}{\tanh r}$, and\\

$\sin \phi = \frac{\sinh y}{\sinh r}$.

\end{center}

\end{proof}

\begin{lem} \label{lim_of_angle_and_extrinsic_radius_of_sep_cap} %C4L9
Given a standard double bubble enclosing volumes $v,w$, suppose that
$\theta$ is the angle between the separating cap and the disc with the
same boundary, and let $y$ be the radius of the disc. As $v,w$ grow
large, $\theta$ approaches $0$ and $y$ approaches $\cosh^{-1}(2)$.
\end{lem}

\begin{proof}
By Lemma \ref{upper_bound_on_rad_circ_interface}, $y <
\cosh^{-1}(2)$. The two caps of the region containing the disk meet the
given disc of radius $y$ at angles  $\theta \in [0, \pi /3)$ and $\alpha=
2\pi/3-\theta$. By volume formulas, the volume $x$ of this region is
the sum of
%%%%%%%%%%%%%%%%%%%%%%%%%%%%%%%%%%%%%%%%%%%%%%%%%%%%%%%%%%%%%%%%%%%
\begin{equation}
\pi \left(\frac{\sinh y \sin \theta}{\sech y + \cos \theta}
 - \tanh^{-1}\left(\frac{\sinh y \sin \theta}{\cosh y + \cos \theta}\right)\right)
\end{equation}
%%%%%%%%%%%%%%%%%%%%%%%%%%%%%%%%%%%%%%%%%%%%%%%%%%%%%%%%%%%%%%%%%%%
and
%%%%%%%%%%%%%%%%%%%%%%%%%%%%%%%%%%%%%%%%%%%%%%%%%%%%%%%%%%%%%%%%%%%
\begin{equation}
\pi \left(\frac{\sinh y \sin \alpha}{\sech y + \cos \alpha}
- \tanh^{-1}\left(\frac{\sinh y \sin \alpha}{\cosh y + \cos \alpha}\right)\right).
\end{equation}
%%%%%%%%%%%%%%%%%%%%%%%%%%%%%%%%%%%%%%%%%%%%%%%%%%%%%%%%%%%%%%%%%%%
Notice that $\cosh y + \cos \theta > 1 + -1 = 0$, and similarly
$\cosh y + \cos \alpha > 0$. Hence the ``$\tanh^{-1}$'' terms are
positive, implying that
%%%%%%%%%%%%%%%%%%%%%%%%%%%%%%%%%%%%%%%%%%%%%%%%%%%%%%%%%%%%%%%%%%%
\begin{equation}
x < \pi \left(\frac{\sinh y \sin \alpha}{\sech y + \cos \alpha} +
\frac{\sinh y \sin \theta}{\sech y + \cos \theta}\right).
\end{equation}
%%%%%%%%%%%%%%%%%%%%%%%%%%%%%%%%%%%%%%%%%%%%%%%%%%%%%%%%%%%%%%%%%%%
Because $\sech y > 1/2$ and $\cos\theta, \cos\alpha > -1/2$, the
denominators of the above expression are positive. Hence,
%%%%%%%%%%%%%%%%%%%%%%%%%%%%%%%%%%%%%%%%%%%%%%%%%%%%%%%%%%%%%%%%%%%
\begin{eqnarray*}
x &<& \pi\left( \frac{\sinh (\cosh^{-1}(2))\sin(\pi/2)}{\sech y + \cos
\theta} +  \frac{\sinh (\cosh^{-1}(2))\sin(\pi/2)}{\sech y + \cos
\alpha}\right) \\
   &=& \pi\left( \frac{\sqrt{3}}{\sech y + \cos \theta}
           +\frac{\sqrt{3}}{\sech y + \cos \alpha}\right)\\
   &<&  \frac{2\pi \sqrt{3}}{\sech y + \cos (2\pi/3 -\theta)}.
\end{eqnarray*}

As $v,w$ approach infinity, $x \geq \min\{v,w\}$ must approach
infinity as well, implying that the positive quantity $\sech y +
\cos(2\pi/3 - \theta)$ must approach $0$ -- that is, $y$ must approach
$\cosh^{-1}(2)$, and $\cos(2\pi/3 - \theta)$ must approach $-\frac{1}{2}$.
Because $2\pi/3 -\theta \in (\pi/3,2\pi/3]$, $\theta$ must approach
$0$.  In particular, when $v,w > x_0 = 300 \lambda$, we have
%%%%%%%%%%%%%%%%%%%%%%%%%%%%%%%%%%%%%%%%%%%%%%%%%%%%%%%%%%%%%%%%%%%
\begin{equation}
x_0 <\frac{2\pi\sqrt{3}}{\sech y + \cos(2\pi/3 - \theta)} < 
\frac{2\pi\sqrt{3}}{1/2 + \cos(2\pi/3 - \theta)}
\end{equation}
%%%%%%%%%%%%%%%%%%%%%%%%%%%%%%%%%%%%%%%%%%%%%%%%%%%%%%%%%%%%%%%%%%%
or

%%%%%%%%%%%%%%%%%%%%%%%%%%%%%%%%%%%%%%%%%%%%%%%%%%%%%%%%%%%%%%%%%%%
\begin{equation}
\cos(2\pi/3 - \theta)< \frac{2 \pi\sqrt{3}}{x_0}-1/2 < -.457.
\end{equation}
%%%%%%%%%%%%%%%%%%%%%%%%%%%%%%%%%%%%%%%%%%%%%%%%%%%%%%%%%%%%%%%%%%%
For $\theta\in[0,\pi/3)$, this is true only if $\theta< 1/20$.
\end{proof}

We have the following numerical estimate: if $v,w > 300\lambda \approx 254.18$, then $\theta < 1/20$.

\begin{lem} \label{limiting_surface_area_of_sdb} 
As $v,w$ grow large, the surface area of a standard double bubble
enclosing volumes $v,w$ approaches
%%%%%%%%%%%%%%%%%%%%%%%%%%%%%%%%%%%%%%%%%%%%%%%%%%%%%%%%%%%%%%%%%%%
\begin{equation}
A_{\Hth}(v + v_{\infty}) + A_{\Hth}(w + v_{\infty}) - 2a_{\infty} + c_{\infty}
\end{equation}
%%%%%%%%%%%%%%%%%%%%%%%%%%%%%%%%%%%%%%%%%%%%%%%%%%%%%%%%%%%%%%%%%%%
where $v_{\infty} = \pi (3/2 - \ln 2)$, $a_{\infty} = 3\pi$, and
$c_{\infty} = 2\pi.$
\end{lem}

\begin{proof}
Given a standard double bubble enclosing volumes $v,w$, consider
the region enclosing volume $v$. Its outer cap lies on a sphere;
let the volume of this sphere be $v + v^{(1)}_{v,w}$ and let the
surface area of the sphere exceed that of the outer cap by
$a^{(1)}_{v,w}$. Define $v^{(2)}_{v,w}$ and $a^{(2)}_{v,w}$
similarly with respect to the region enclosing volume $w$.
Finally, let $c_{v,w}$ denote the surface area of the separating
cap. Then the surface area of the standard double bubble is
%%%%%%%%%%%%%%%%%%%%%%%%%%%%%%%%%%%%%%%%%%%%%%%%%%%%%%%%%%%%%%%%%%%
\begin{equation}
A_{\Hth}(v + v^{(1)}_{v,w}) + A_{\Hth}(w + v^{(2)}_{v,w}) - a^{(1)}_{v,w}
-a^{(2)}_{v,w}+ c_{v,w}. \label{surface_area_of_sdb_equation}
\end{equation}
%%%%%%%%%%%%%%%%%%%%%%%%%%%%%%%%%%%%%%%%%%%%%%%%%%%%%%%%%%%%%%%%%%%
Suppose that $\theta$ is the angle between the separating cap and the
disc with the same boundary, and let $y$ be the radius of the disc. By
Lemma \ref{lim_of_angle_and_extrinsic_radius_of_sep_cap}, as $v$ and
$w$ grow large, $\theta\rightarrow 0$ and $y \rightarrow y_0 =
\cosh^{-1}(2)$. Using formulas for volume and surface area, one can
easily check that
%------------------------------------------------------------------*
\begin{eqnarray*}
\lim_{v,w\rightarrow \infty} v^{(1)}_{v,w}
&=&\lim_{v,w\rightarrow \infty} v^{(2)}_{v,w} \\
&=&\pi\left( \frac{\sinh y_0 \sin(\pi/3)}{\sech y + \cos (\pi/3)}
       -\tanh^{-1} \left(\frac{\sinh y_0 \sin(\pi/3)}{\cos y +\cos (\pi/3)}\right)\right)\\
&=&v_{\infty}
\end{eqnarray*}
%------------------------------------------------------------------*
\begin{eqnarray*}
\lim_{v,w\rightarrow \infty} a^{(1)}_{v,w}
&=& \lim_{v,w\rightarrow \infty} a^{(2)}_{v,w} \\
&=& \frac{2\pi\sinh^2 y_0}{1 + \cos (\pi/3)\cosh y_0} \\
&=& a_{\infty}
\end{eqnarray*}
%------------------------------------------------------------------*
and
%------------------------------------------------------------------*
\begin{eqnarray*}
\lim_{v,w\rightarrow \infty} c_{v,w} &=& \frac{2\pi\sinh^2 y_0}{1 + \cos 0 \cosh y_0}\\
&=& c_{\infty}.
\end{eqnarray*}
%------------------------------------------------------------------*
Hence,
%%%%%%%%%%%%%%%%%%%%%%%%%%%%%%%%%%%%%%%%%%%%%%%%%%%%%%%%%%%%%%%%%%%
$$|A(v + v^{(1)}_{v,w}) - A(v + v_{\infty})| <
A'(v)|v^{(1)}_{v,w} - v_{\infty}|$$
%%%%%%%%%%%%%%%%%%%%%%%%%%%%%%%%%%%%%%%%%%%%%%%%%%%%%%%%%%%%%%%%%%%
approaches $0$ as $v,w$ grow
large. So too does $|A(w + v^{(2)}_{v,w}) - A(w + v_{\infty})|$. From
these computed limits, it follows that the surface area of the
standard double bubble enclosing volumes $v,w$ -- as calculated in
Equation \ref{surface_area_of_sdb_equation} -- approaches $A(v + v_{\infty}) + A(w +
v_{\infty}) - 2a_{\infty} + c_{\infty}$ as $v,w$ get large.
\end{proof}

\begin{prop} \label{limit_along_rays_is_positive} 
For each fixed $\psi > 0$,
%%%%%%%%%%%%%%%%%%%%%%%%%%%%%%%%%%%%%%%%%%%%%%%%%%%%%%%%%%%%%%%%%%
$$\lim_{w\rightarrow\infty}  F(\psi w, w) = 2 \pi \ln \frac{4(\psi + 1)}{e^2}.$$
%%%%%%%%%%%%%%%%%%%%%%%%%%%%%%%%%%%%%%%%%%%%%%%%%%%%%%%%%%%%%%%%%%
(Note that this limit is nonnegative if and only if $\psi \geq \lambda$ which is less than $0.85$.)
\end{prop}

\begin{proof}
We write $v = v(w) = \lambda w$. Let $v_{\infty} = \pi (3/2 - \ln 2)$,
$a_{\infty} = 3\pi$, and $c_{\infty} = 2\pi$. By Lemma
\ref{limiting_surface_area_of_sdb}, the given limit equals
%%%%%%%%%%%%%%%%%%%%%%%%%%%%%%%%%%%%%%%%%%%%%%%%%%%%%%%%%%%%%%%%%%%
\begin{equation}
\lim_{w\rightarrow \infty} 2A(\frac{v}{2}) + A(w) + A(v + w) - 2(A(v +
v_{\infty}) + A(w + v_{\infty}) - 2a_{\infty} + c_{\infty})
\label{form_of_L2}
\end{equation}
%%%%%%%%%%%%%%%%%%%%%%%%%%%%%%%%%%%%%%%%%%%%%%%%%%%%%%%%%%%%%%%%%%%
provided that the latter limit exists. By Lemma \ref{limiting_value_of_surface_of_sphere},
\begin{equation}
\lim_{x\rightarrow \infty} A(x) = 2x + 2 \pi \ln x - 2 \pi( 1 - \ln (\pi/2))
\end{equation}
for $x \in \{\frac{v}{2}, w, v+w, 2v+v_{\infty}, w+v_{\infty}\}$. Substituting these expressions into \ref{form_of_L2} and simplifying
gives
%%%%%%%%%%%%%%%%%%%%%%%%%%%%%%%%%%%%%%%%%%%%%%%%%%%%%%%%%%%%%%%%%%%
\begin{eqnarray*}
&& \lim_{w\rightarrow \infty}
(2 \pi (2\ln(\frac{v}{2}) + \ln w + \ln(v + w) - 2\ln(v + v_{\infty}) -
2\ln(w + v_{\infty})) \\
&& \qquad \qquad - 8v_{\infty} + 4 a_{\infty} - 2c_{\infty})
\end{eqnarray*}
%%%%%%%%%%%%%%%%%%%%%%%%%%%%%%%%%%%%%%%%%%%%%%%%%%%%%%%%%%%%%%%%%%%
provided that this new limit exists. Indeed, this new limit does
exist and equals
\begin{eqnarray*}
 && \lim_{w\rightarrow \infty} ( 2 \pi \ln
\left(\frac{v^2w(v+w)}{4(v+v_{\infty})^2(w+v_{\infty})^2}\right) -
8v_{\infty} + 4a_{\infty} - 2c_{\infty}) \\
%------------------------------------------------------------------*
&=& \lim_{w\rightarrow \infty}
( 2 \pi \ln \left(\frac{v+w}{4w}\right) -
8v_{\infty} + 4a_{\infty} - 2c_{\infty})\\
%------------------------------------------------------------------*
&=&  2 \pi \ln \left(\frac{\psi+1}{4}\right) - 8 \cdot \pi (3/2 - \ln 2) +
4\cdot 3 \pi-2 \cdot 2 \pi \\
&=& 2 \pi \ln \left(\frac{4(\psi +1 )}{e^2}\right).
\end{eqnarray*}
%------------------------------------------------------------------*
Hence the limit in (\ref{form_of_L2}) evaluates to $2 \pi \ln
(\frac{4(\psi+1)}{e^2})$, which is only positive when $\psi \geq \lambda$, finishing the proof.
\end{proof}

\subsection{The Hutchings function is decreasing along a line for large volumes}

\begin{lem} \label{bound_on_extra_area} 
Given a standard double bubble enclosing volumes $v,w$, the outer cap
of the region of volume $v$ lies on a sphere, say with volume $v +
v_1$ (resp.  the outer cap of the region of volume $w$ lies on a
sphere with volume $w + w_1$). If $v,w > 300\lambda$ then $v_1,
w_1 < 3$.
\end{lem}

\begin{proof}
Let $\vol(y,\alpha)$ denote the volume of a spherical cap with
associated disc of radius $y$, where the cap and disc meet at an angle
$[0,\pi]$. This volume is increasing in both $y$ and $\alpha$. Without
loss of generality, assume that $v \leq w$. The three caps of the
standard double bubble meet along a circular interface with radius
$y_0$; let the disc with this circle as its boundary meet the
separating cap of the double bubble at angle $\theta$. By Lemma
\ref{upper_bound_on_rad_circ_interface}, $y_0 < \cosh ^{-1} (2)$. By
Lemma \ref{lim_of_angle_and_extrinsic_radius_of_sep_cap}, $ \theta <
1/20$. Both $v_1 = \vol(y_0,\pi /3 + \theta) - \vol(y_0,\theta)$ and $w_1 =
\vol(y_0,\pi/3 - \theta ) + \vol(y_0,\theta)$ are less than
$\vol(y_0,\pi /3 + \theta) + \vol(y_0, \theta)$, which in turn is less
than $\vol(\cosh^{-1}(2),\pi/3 +1/20) + \vol(\cosh^{-1}(2), 1/20) <
3$.
\end{proof}

We now prove some algebraic lemmas that will be used in the proof of
\Prop\ \ref{hf_decr_along_crit_line_when_w_large} and later in \Prop\ \ref{algebraic_lemma}.

\begin{lem} \label{fraction_arithmetic} 
If $\pi \ln x + a>0$, then
%%%%%%%%%%%%%%%%%%%%%%%%%%%%%%%%%%%%%%%%%%%%%%%%%%%%%%%%%%%%%%%%%%%
$\frac{1}{x + \pi \ln x + a} > \frac{1}{x} - \frac{ \pi \ln x + a}{x^2}.$
%%%%%%%%%%%%%%%%%%%%%%%%%%%%%%%%%%%%%%%%%%%%%%%%%%%%%%%%%%%%%%%%%%%
\end{lem}

\begin{proof}
The left hand side equals
%%%%%%%%%%%%%%%%%%%%%%%%%%%%%%%%%%%%%%%%%%%%%%%%%%%%%%%%%%%%%%%%%%%
\begin{equation}
\frac{1}{x} - \frac{\pi \ln x+ a}{x(x+\pi \ln x+ a)},
\end{equation}
%%%%%%%%%%%%%%%%%%%%%%%%%%%%%%%%%%%%%%%%%%%%%%%%%%%%%%%%%%%%%%%%%%%
which is greater than $\frac{1}{x} -\frac{\pi \ln x + a}{x^2}$.
\end{proof}

\begin{cor} \label{positive_numbers_with_product_bigger_than_e}
For positive numbers $\mu, w$ with $\ln \mu + a>0$
%%%%%%%%%%%%%%%%%%%%%%%%%%%%%%%%%%%%%%%%%%%%%%%%%%%%%%%%%%%%%%%%%%%
\begin{equation}
\frac{\mu}{\mu w+\pi\ln(\mu w)+a}
>
\frac{1}{w}-\frac{1}{w^2}
\left(\frac{\pi\ln w}{\mu}
+\frac{\pi\ln \mu}{\mu}
-\frac{a}{\mu}\right).
\end{equation}
%%%%%%%%%%%%%%%%%%%%%%%%%%%%%%%%%%%%%%%%%%%%%%%%%%%%%%%%%%%%%%%%%%%
\end{cor}

\begin{proof}
Simply apply Lemma \ref{fraction_arithmetic} with $x = \mu w$.
\end{proof}

\begin{lem} \label{fraction_arithmetic_with_big_numbers}  
If $x >  150\lambda$, then
$$\frac{1}{x + \pi \ln x -3} < \frac{1}{x} - \frac{ \pi \ln x -3}{1.1 x^2}.$$
\end{lem}

\begin{proof}
For such $x$, we have $1.1x > x + \pi \ln x - 3 > 0$. Hence,
%%%%%%%%%%%%%%%%%%%%%%%%%%%%%%%%%%%%%%%%%%%%%%%%%%%%%%%%%%%%%%%%%%%
\begin{equation}
\frac{1}{x+\pi\ln x-3}=\frac{1}{x}-\frac{\pi\ln x-3}{x(x+\pi\ln
x-3)}< \frac{1}{x}-\frac{\pi\ln x-3}{1.1x^2} \end{equation}
%%%%%%%%%%%%%%%%%%%%%%%%%%%%%%%%%%%%%%%%%%%%%%%%%%%%%%%%%%%%%%%%%%%
as desired.
\end{proof}

\begin{cor} \label{w_big_and_mu_half_way_to_critical_slope}
For positive $\mu$, $w$ , with $\mu w > 150$ 
%%%%%%%%%%%%%%%%%%%%%%%%%%%%%%%%%%%%%%%%%%%%%%%%%%%%%%%%%%%%%%%%%%%
\begin{equation}
\frac{\mu}{\mu w+\pi \ln(\mu w)-3}
<
\frac{1}{w}
- \frac{1}{w^2}
\left(
        \frac{\pi \ln w}{1.1\mu}
    +\frac{\pi\ln \mu}{1.1\mu}
    -\frac{3}{1.1\mu}
\right).
\end{equation}
%%%%%%%%%%%%%%%%%%%%%%%%%%%%%%%%%%%%%%%%%%%%%%%%%%%%%%%%%%%%%%%%%%%
\end{cor}

\begin{proof}
Apply Lemma \ref{fraction_arithmetic_with_big_numbers} with $x = \mu w$.
\end{proof}

\begin{lem}\label{simple_ineq}
 For all $w \geq 300$,
%%%%%%%%%%%%%%%%%%%%%%%%%%%%%%%%%%%%%%%%%%%%%%%%%%%%%%%%%%%%%%%%%%%
\begin{eqnarray*}
&&2\left( \frac{\pi \ln w}{\lambda}+ \frac{\pi \ln \lambda}{\lambda}- \frac{2}{\lambda} + \pi \ln w -2 \right) < \\
&&
\frac{2\pi \ln w}{1.1 \frac {\lambda}{2}} +\frac{2\pi \ln (\frac {\lambda}{2})}{1.1 \frac {\lambda}{2}}-\frac {6}{1.1 \frac {\lambda}{2}}+
\frac{\pi \ln w}{1.1} -\frac {3}{1.1}+\\
&&
\frac{\pi \ln w}{1.1 (\lambda+1)} +\frac{\pi \ln (\lambda+1)}{1.1 (\lambda+1)}-\frac {3}{1.1 (\lambda+1)}.
\end{eqnarray*}
%%%%%%%%%%%%%%%%%%%%%%%%%%%%%%%%%%%%%%%%%%%%%%%%%%%%%%%%%%%%%%%%%%%
\end{lem}

\begin{proof}
For all $w \geq 300$,

%%%%%%%%%%%%%%%%%%%%%%%%%%%%%%%%%%%%%%%%%%%%%%%%%%%%%%%%%%%%%%%%%%%
\begin{eqnarray*}
&& \pi \left(2\frac{\ln \lambda}{\lambda}
-2\frac{\ln \lambda/2}{1.1\lambda/2}
-\frac{\ln (\lambda + 1)}{1.1 (\lambda + 1)}\right)\\
&&
+\left(2\frac{3}{1.1\lambda/2}
+\frac{3}{1.1}
+\frac{3}{1.1(\lambda + 1)}
-\frac{4}{\lambda}
-4\right) <\\
&& 
\left(\frac{2}{1.1\lambda/2} +\frac{1}{1.1} +\frac{1}{1.1(\lambda + 1)}
-\frac{2}{\lambda} -2\right)\pi\ln w. 
\end{eqnarray*}
%%%%%%%%%%%%%%%%%%%%%%%%%%%%%%%%%%%%%%%%%%%%%%%%%%%%%%%%%%%%%%%%%%%

Simple algebraic manipulation gives the desired result.

\end{proof}

\begin{lem}\label{ineq_for_crit_line}
If $w \geq 300$, then
%%%%%%%%%%%%%%%%%%%%%%%%%%%%%%%%%%%%%%%%%%%%%%%%%%%%%%%%%%%%%%%%%%%
\begin{eqnarray*}
&&4+4\lambda +2\pi \left(\frac{2\lambda}{\lambda w + \pi \ln(\lambda w) +2}+\frac{2}{w +
\pi \ln(w) +2}\right) \\
&&
>4+4\lambda +2\pi \left(\frac{\lambda}{\lambda \frac{w}{2} + \pi \ln(\lambda \frac{w}{2})- 3}\right. +\frac{1}{w + \pi \ln(w) - 3}  \\
&&
\left. +\frac{\lambda + 1}{(\lambda + 1)w + \pi \ln((\lambda + 1)w)
- 3 }\right).
\end{eqnarray*}
%%%%%%%%%%%%%%%%%%%%%%%%%%%%%%%%%%%%%%%%%%%%%%%%%%%%%%%%%%%%%%%%%%%
\end{lem}

\begin{proof}
By Corollary \ref{positive_numbers_with_product_bigger_than_e},

%%%%%%%%%%%%%%%%%%%%%%%%%%%%%%%%%%%%%%%%%%%%%%%%%%%%%%%%%%%%%%%%%%%
\begin{eqnarray*}
&&\frac{2\lambda}{\lambda w + \pi \ln(\lambda w) +2}+\frac{2}{w +
\pi \ln(w) +2}\\
&&
 > 2 \left( \frac{2}{w}-\frac{1}{w^2} \left(\frac{\pi \ln w}{\lambda}+\frac{\pi \ln (\lambda)}{\lambda}-\frac{2}{\lambda}\right) - \frac{1}{w^2} (\pi \ln w - 2)   \right)
\end{eqnarray*}
%%%%%%%%%%%%%%%%%%%%%%%%%%%%%%%%%%%%%%%%%%%%%%%%%%%%%%%%%%%%%%%%%%%

\begin{eqnarray*}
&&> \frac{4}{w}-\frac{1}{w^2}\left(\frac{2\pi \ln w}{1.1 \frac {\lambda}{2}} +\frac{2\pi \ln (\frac {\lambda}{2})}{1.1 \frac {\lambda}{2}}-\frac {6}{1.1 \frac {\lambda}{2}}+
\frac{\pi \ln w}{1.1} -\frac {3}{1.1}\right.\\
&&+\frac{\pi \ln w}{1.1 (\lambda+1)} 
 \left. +\frac{\pi \ln (\lambda+1)}{1.1 (\lambda+1)}-\frac {3}{1.1 (\lambda+1)}
 \right)
\end{eqnarray*}
(by Lemma \ref{simple_ineq})

\begin{eqnarray*}
> \frac{\lambda}{\frac{\lambda w}{2}+ \pi \ln \frac{\lambda w}{2} - 3} +\frac{1}{w+ \pi \ln w - 3}+\frac{(1+\lambda)}{(\lambda+1)w+ \pi \ln ((\lambda+1)w) - 3}
\end{eqnarray*}
(by Corollary \ref{w_big_and_mu_half_way_to_critical_slope}). Multiplying both sides of the inequality by $2\pi$ and adding $4(1+\lambda)$ yields the desired result.
\end{proof}

\begin{prop} \label{hf_decr_along_crit_line_when_w_large} 
If $w \geq 300$, then
\begin{eqnarray*}
 \frac{d }{d w}F(\lambda w,w) < 0. 
\end{eqnarray*}
\end{prop}

\begin{proof}
Consider a standard double bubble enclosing volumes $\lambda w$ and $w$. Let $v_1, w_1$, be the volumes 
needed to complete the two outer caps of the double bubble which enclose volumes $\lambda w$ and $w$. Then,
%%%%%%%%%%%%%%%%%%%%%%%%%%%%%%%%%%%%%%%%%%%%%%%%%%%%%%%%%%%%%%%%%%%
\begin{equation}
\frac{d A}{dw}(\lambda w,w) = \lambda A'(\lambda w + v_1) + A'(w +w_1).  \label{derivative_with_respect_to_w}
\end{equation}
%%%%%%%%%%%%%%%%%%%%%%%%%%%%%%%%%%%%%%%%%%%%%%%%%%%%%%%%%%%%%%%%%%%

 The bound on $A'(x)$ given in Lemma \ref{upper_bound_on_change_in_area_for_spheres} tells us that
%%%%%%%%%%%%%%%%%%%%%%%%%%%%%%%%%%%%%%%%%%%%%%%%%%%%%%%%%%%%%%%%%%%
\begin{equation}
\left(
       2 + \frac{2\pi}{\mu w+\pi\ln(\mu w)-3}
\right) >  A'(\mu w).
\end{equation}
%%%%%%%%%%%%%%%%%%%%%%%%%%%%%%%%%%%%%%%%%%%%%%%%%%%%%%%%%%%%%%%%%%%

Hence, 

%%%%%%%%%%%%%%%%%%%%%%%%%%%%%%%%%%%%
\begin{equation}\label{first_part_of_HF_less_than_log}
 2\lambda+2\pi \left(\frac{\lambda}{\lambda \frac{w}{2}+\pi \ln(\lambda \frac{w}{2}) - 3}\right)> \lambda A' \left(\frac{\lambda w}{2}\right),
\end{equation}

\begin{equation}\label{second_part_of_HF_less_than_log}
 2+2\pi \left(\frac{1}{w+\pi \ln(w) - 3} \right)>  A'(w),
\end{equation}

\begin{equation}\label{third_part_of_HF_less_than_log}
 2+2\pi \left(\frac{\lambda+1}{(\lambda +1)w +\pi \ln((\lambda +1)w) - 3}\right)>  A'((\lambda +1)w).
\end{equation}

By Corollary
\ref{lower_bound_on_change_in_area_for_large_spheres}, 
%%%%%%%%%%%%%%%%%%%%%%%%%%%%%%%%%%%%%%%%%%%%%%%%%%%%%%%% 
\begin{multline}
\lambda A'(\lambda w + 3) + A'(w + 3) > 
\lambda \left(2+\frac{2\pi}{\lambda w+\pi\ln (\lambda w)+2}\right)\\
+\left(2+\frac{2\pi}{w+\pi\ln w+2}\right). 
\end{multline}
%%%%%%%%%%%%%%%%%%%%%%%%%%%%%%%%%%%%%%%%%%%%%%%%%%%%%%%%

Since $\lambda w$ and $w$ are greater than $300 \lambda$, we can apply Lemma
\ref{bound_on_extra_area} to see that $v_1,w_1 < 3$.  Hence,

\begin{equation}
 \lambda A'(\lambda w+ v_1) + A'(w+w_1) > \lambda
A'(\lambda w + 3) + A'(w + 3).
\end{equation}

Therefore,
\begin{multline}\label{curvatures_bigger_than_logs}
\lambda A'(\lambda w + v_1) + A'(w + w_1) >\\
\lambda \left(2+\frac{2\pi}{\lambda w+\pi \ln \lambda w+2}\right)+
\left(2+\frac{2\pi}{w+\pi\ln w+2}\right).
\end{multline}

Combining inequalities \ref{curvatures_bigger_than_logs}, \ref{first_part_of_HF_less_than_log}, \ref{second_part_of_HF_less_than_log}, 
\ref{third_part_of_HF_less_than_log} and Proposition \ref{ineq_for_crit_line},
we can see for all $\lambda w,w > 300 \lambda$,

\begin{equation}
2\lambda A'(\lambda w + v_1) +2A'(w+w_1)> \lambda A' \left(\frac{\lambda w}{2}\right)+ A'(w)+A'(\lambda w+w).
\end{equation}

Thus,
\begin{equation}
2\frac{d}{dw}A(\lambda w, w) > \lambda A'\left(\frac{\lambda w}{2}\right)+ A'(w)+A'(\lambda w+w)
\end{equation}
as desired.
\end{proof}

\subsection{The Hutchings function is increasing in v for large volumes}

\begin{prop}\label{algebraic_lemma}
For $v\geq \lambda w$, $w \geq 150,$ 

\begin{equation}
\frac{1}{\frac{v}{2} + \pi \ln \frac{v}{2} - 1.041}+\frac{1}{v + w + \pi \ln (v+w) - 1.041}>2\frac{1}{v \pi \ln v - 3}
\end{equation}
\end{prop}

\begin{proof}

Since $v\geq \lambda w$, $w \geq 150,$

\begin{equation}
\frac{v^2}{2} > 4 (v+w )(\pi \ln \frac{v}{2}-1.041)
\end{equation}
and
\begin{equation}
\frac{v+w}{2} > \pi \ln (v+w).
\end{equation}

Hence,
\begin{multline}
 v^2(v+w) > (v+w)^2 \frac{\pi \ln \frac{v}{2} -1.041}{\frac{1}{4}} +v^2 \pi \ln (v+w)\\
-1.041v^2 -(v+w)^2 \frac{\pi \ln v - 3}{1.1}. 
\end{multline}

So,
\begin{equation}\label{fraction_inequality}
\frac{2}{v} - \frac{\pi \ln \frac{v}{2} - 1.041}{(\frac{v}{2})^2}+\frac{1}{v+w} - \frac{\pi \ln (v+w) - 1.041}{(v+w)^2}> 2(\frac{1}{v}- \frac{\pi \ln v -3}{1.1v^2}).
\end{equation}

Applying Lemmas \ref{fraction_arithmetic} and \ref{fraction_arithmetic_with_big_numbers} to the above equation yields,
\begin{equation}
\frac{1}{\frac{v}{2} + \pi \ln \frac{v}{2} - 1.041}+\frac{1}{v + w + \pi \ln (v+w) - 1.041}>2\frac{1}{v + \pi \ln v - 3}
\end{equation}
as desired.

\end{proof}

\begin{lem}\label{derivative_with_respect_to_v_is_positive}

For any fixed $w \geq 150$ and $v \geq \lambda w$, 
%%%%%%%%%%%%%%%%%%%%%%%%%%%%%%%%%%%%
\begin{equation}
\frac{\partial F}{\partial v} >0, 
\end{equation}
%%%%%%%%%%%%%%%%%%%%%%%%%%%%%%%%%%%%%
\end{lem}

\begin{proof}

Differentiating the Hutchings function with respect to $v$ yields

\begin{equation}\label{hutchings_derivative_for_v}
\frac{\partial F}{\partial v} = A'(\frac{v}{2}) + A'(v+w) - 2 \frac{\partial}{\partial v} A(v,w)> A'(\frac{v}{2}) + A'(v+w) - 2 A'(v).
\end{equation}

By Lemma \ref{lower_bound_on_change_in_area_for_large_spheres_w_150} and Lemma \ref{upper_bound_on_change_in_area_for_spheres},
\begin{multline}
\frac{\partial F}{\partial v} >  \frac{1}{\frac{v}{2} + \pi \ln \frac{v}{2}- 1.041}\\
 +\frac{1}{v + w + \pi \ln (v+w) - 1.041}-2\frac{1}{v + \pi \ln v - 3}.
\end{multline}

By \Prop\ \ref{algebraic_lemma},
\begin{equation}
\frac{\partial F}{\partial v} >0, 
\end{equation}
as desired.

\end{proof}

\subsection{Conclusion}

\begin{thm}[Hutchings function is positive for large volumes] \label{hf_positive_for_large_volumes}
For all $v \leq w$ such that $w > 300$ and $v > .85w$ (or indeed
$v\geq \lambda w \approx .841w$), the Hutchings function $F(v,w)=2A(\frac{v}{2})
+ A(w) + A(v + w) - 2A(v,w)$ is positive.
\end{thm}

\begin{proof}
Consider $(v,w) = (v_0,w_0)$, where $\lambda = v_0/w_0 >.85> \lambda $ and $w_0 > 300$.  By \Prop\ \ref{derivative_with_respect_to_v_is_positive} the
derivative of $F$ with respect to its first argument is positive
for the volume pairs under consideration, so $F(v_0,w_0) >
F(\lambda w_0,w_0)$.

By Proposition \ref{hf_decr_along_crit_line_when_w_large} $F(\lambda w, w)$ decreases as $w$ increases and by \Prop\
\ref{limit_along_rays_is_positive},
%%%%%%%%%%%%%%%%%%%%%%%%%%%%%%%%%%%%%%%%%%%%%%%%%%%%%%%%%%%%%%%
$$ F(\lambda w_0,w_0)  > \lim_{w\rightarrow \infty} F(\lambda w, w) =0.$$
%%%%%%%%%%%%%%%%%%%%%%%%%%%%%%%%%%%%%%%%%%%%%%%%%%%%%%%%%%%%%%%
Hence, $F(v_0,w_0) > 0$, as claimed.
\end{proof}

\begin{rem}
The ratio $\lambda = e^2/4-1$ is sharp in the following sense: for $\psi<
\lambda $, there exist arbitrarily large $v,w$ with $v/w = \psi$
such that the Hutchings function is negative. This is a result of Proposition \ref{limit_along_rays_is_positive}. 
\end{rem}

%%%%%%%%%%%%%%%%%%%%%%%%%%%%%%%%%%%%%%%%%%%%%%%%%%%%%%%%%%%%%%%
%%%%%%%%%%%%%%%%%%%%%%%%%%%%%%%%%%%%%%%%%%%%%%%%%%%%%%%%%%%%%%%
%%%%%%%%%%%%%%%%%%%%%%%%%%%%%%%%%%%%%%%%%%%%%%%%%%%%%%%%%%%%%%%
\section{The computer proof of the positivity of the Hutchings function in $\Sth$ and $\Hth$.} \label{computer_proof}
%%%%%%%%%%%%%%%%%%%%%%%%%%%%%%%%%%%%%%%%%%%%%%%%%%%%%%%%%%%%%%%
%%%%%%%%%%%%%%%%%%%%%%%%%%%%%%%%%%%%%%%%%%%%%%%%%%%%%%%%%%%%%%%
%%%%%%%%%%%%%%%%%%%%%%%%%%%%%%%%%%%%%%%%%%%%%%%%%%%%%%%%%%%%%%%
\subsection{Introduction}

Rigorously showing positivity of the Hutchings function is equivalent to showing that $g(v,w) = 2A(\frac{v}{2}) + A(w) + A(v+w)$ is strictly greater than
$h(v, w) = 2A(v, w)$. The functions $g$ and $h$ have very similar properties in $\Sth$ and $\Hth$, but the proofs vary slightly because the area of single bubbles and standard double bubbles are not increasing for all volumes in $\Sth$ and the $\Hth$ code has to efficiently calculate the curvatures of double bubbles.\\

The computer proof relies on the facts that $g$ is concave and $h$ is increasing on relevant domains (Propositions \ref{double_increases_Sth} and \ref{Ah_Increasing}) in both spaces and that $g$ is increasing on all of $\Hth$. If $g$ and $h$ were each concave increasing functions of one variable, we could show that $g$ is bigger than $h$ on the interval $[v,v+b]$, by showing $g(v) > h(v+a)$ and $g(v+a)> h(v+b)$ (See Figure \ref{1dcase}). 

 The computational lemma (Lemma \ref{computational_lemma}) extends this idea into two dimensions showing that $g > h $ on a polygonal  domain in the $vw$-plane if the minimum value of $g$ on the vertex set is bigger than the maximum value of $h$.\\
  
Using a computer does introduce error, which must be accounted for. A computer uses approximations for $g$ and $h$, which we will call $g_{comp}$ and $h_{comp}$. However if $g$ is underestimated and $h$ is overestimated then the computer test will show $g_{comp} > h_{comp}$ only if $g > h$.  

The following lemma shows that checking a finite number of points can give a global inequality.

 \begin{lem}\label{computational_lemma}
 Given a polygonal domain $D$ with vertices identified by ordered pairs $(v_i, w_i)$,
  a concave function $g(v,w)$, and another function $h(v,w)$
 that obtains its maximum at $(v_0, w_0)$, if $min(g(v_i,  w_i)) > h(v_0, w_0)$
then $g > h$ on $D$.
 \end{lem}

 \begin{proof}
Since $g$ is concave it has a minimum on $D$ at a vertex of the polygon. Thus, $min(g(v_i,  w_i))$ is a lower bound for $g$ on $D$ and  $h(v_0, w_0)$ is an upper bound for $h$ on $D$.
Hence, $g > h$ follows directly.
 \end{proof}

%%%%%%%%%%%%%%%%%%%%%%%%%%%%%%%%%%%%%%%%%%%%%%%%%%%%%%%%%%%%%%%
%%%%%%%%%%%%%%%%%%%%%%%%%%%%%%%%%%%%%%%%%%%%%%%%%%%%%%%%%%%%%%%
%%%%%%%%%%%%%%%%%%%%%%%%%%%%%%%%%%%%%%%%%%%%%%%%%%%%%%%%%%%%%%%
\begin{figure}[h]
\begin{center}
\includegraphics[width=4in]{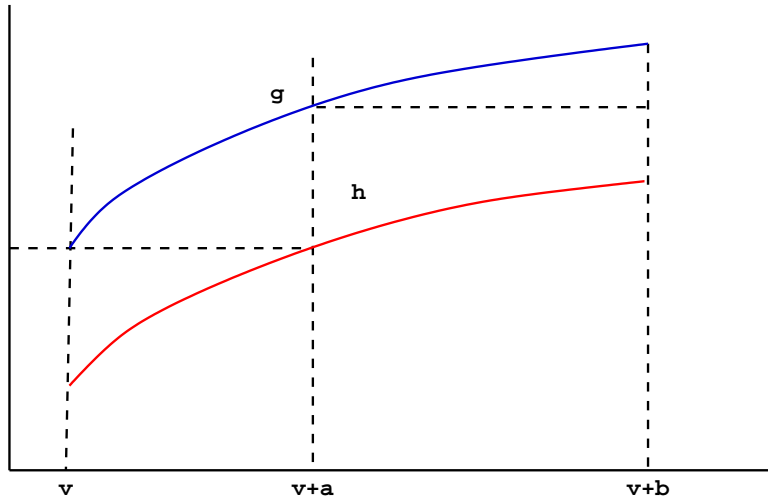}
\caption{\label{1dcase} The steps to show that $g>h$ if both are concave increasing functions of one variable.}
\end{center}
\end{figure}
%%%%%%%%%%%%%%%%%%%%%%%%%%%%%%%%%%%%%%%%%%%%%%%%%%%%%%%%%%%%%%%
%%%%%%%%%%%%%%%%%%%%%%%%%%%%%%%%%%%%%%%%%%%%%%%%%%%%%%%%%%%%%%%
%%%%%%%%%%%%%%%%%%%%%%%%%%%%%%%%%%%%%%%%%%%%%%%%%%%%%%%%%%%%%%%

\paragraph{\textbf{The computer proof in $\Sth$}}
In $\Sth$, rectangular regions and triangular regions are considered. There is additional computational complexity in $\Sth$ because $A(v)$ and $A(v,w)$ are not increasing for all volume pairs $(v,w)$ in $\Sth$. For rectangles where $h(v,w)$ is strictly increasing in a neighborhood of the rectangle, the area of the standard double bubble is over-approximated by using radii of the outer spherical caps corresponding to volumes $v_4$ and $w_4$ bigger than $v$ and $w$ respectfully (See Figure \ref{proofrectangle}). 

 For the area of the single bubble each part of $g(v,w)$ is under-approximated.  $A(\frac{v}{2})$ and $A(w)$ are always under-approximated by taking radii that correspond to spheres enclosing volumes less than $\frac{v}{2}$ and $w$. Evaluating $g_1$ at $(v_1,w_1)$ in Figures \ref{proofrectangle} and \ref{prooftriangle} gives an lower bound on $g_1(v,w)$ on the whole polygon.  $A(v + w)$ is under-approximated by considering the minimum of the value of $A$ at a point to the lower left and a point to the upper right of the rectangle. This is the minimum of $g_2(v_2, w_2)$ and $g_2 (v_3, w_3)$ in Figures \ref{proofrectangle} and \ref{prooftriangle}.  \\

If $g(v,w)\leq h(v+a_2,w +b_2)$ (where $v+a_2$ and $w+b_2$ correspond to Figure \ref{proofrectangle}), then the inequality is checked on the four rectangles with lower left corners of $(v,w), (v+a_1,w), (v,w+b_1)$ and $(v+a_1,w+b_1)$. \\ 

\begin{figure}
\begin{center}
\includegraphics[width=4in]{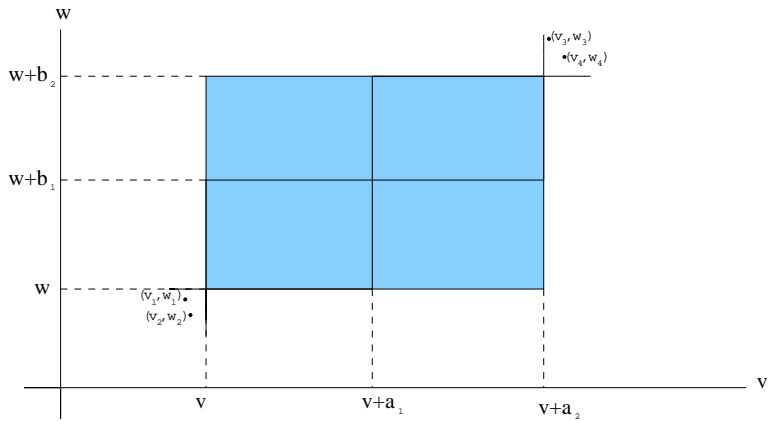}
\caption{\label{proofrectangle} The steps needed to show that $g>h$ on a rectangular region.}
\end{center}
\end{figure}

Next, we consider the entire region where numerical analysis is necessary (Figure \ref{s3domain}).  
The triangular domain is first broken into parts: the right triangle $(\bar{v}, \bar{w}) = (.1,.1), (.1, \frac{1}{3})$, $(\frac{1}{3},\frac{1}{3})$ and the triangle $(\bar{v}, \bar{w}) = (.1,\frac{1}{3}), (\frac{1}{3},\frac{1}{3})$
 and $(.1, .45).$  The latter, upper triangle is difficult to deal with because $h$ is increasing up to the hypotenuse and decreasing above the hypotenuse.  We can assure that we land exactly on this line by setting the two parameters $r_1$ and $r_3$ to be equal. Getting an over-approximation for area on this line just involves overestimating $r_1$. In addition, $A(\frac{1}{3}, \frac{1}{3})$ must be hard coded into the program. If we want to over-approximate $h(v,w)$ we move the point $(v,w)$ southeast along the line where $w = u$ to $(v_4,w_4)$ and use $h(v_4,w_4)$ as an over-approximation (See Figure \ref{prooftriangle}). For $g$, we still just take the appropriate corners of the rectangle containing the triangle. If not, then we break the triangle into a rectangle and two smaller right triangles by connecting a point on the hypotenuse to the two legs of the triangle via perpendicular lines. The desired inequality is then shown to be true on the rectangle and the two smaller triangles. Similarly for the rectangle if the desired inequality can not be shown on the large rectangle, the rectangle is split into four smaller rectangles and the inequality is shown to hold true on them (See Figure \ref{proofrectangle}). Both subdivisions repeat until the desired inequality can be shown on each subdivided rectangle or triangle. In the case where the upper right corner a rectangle is on the line where $w=u$, we over-approximate $v,w$ using the same methods as the triangle.\\

\begin{figure}
\begin{center}
\includegraphics[width=4in]{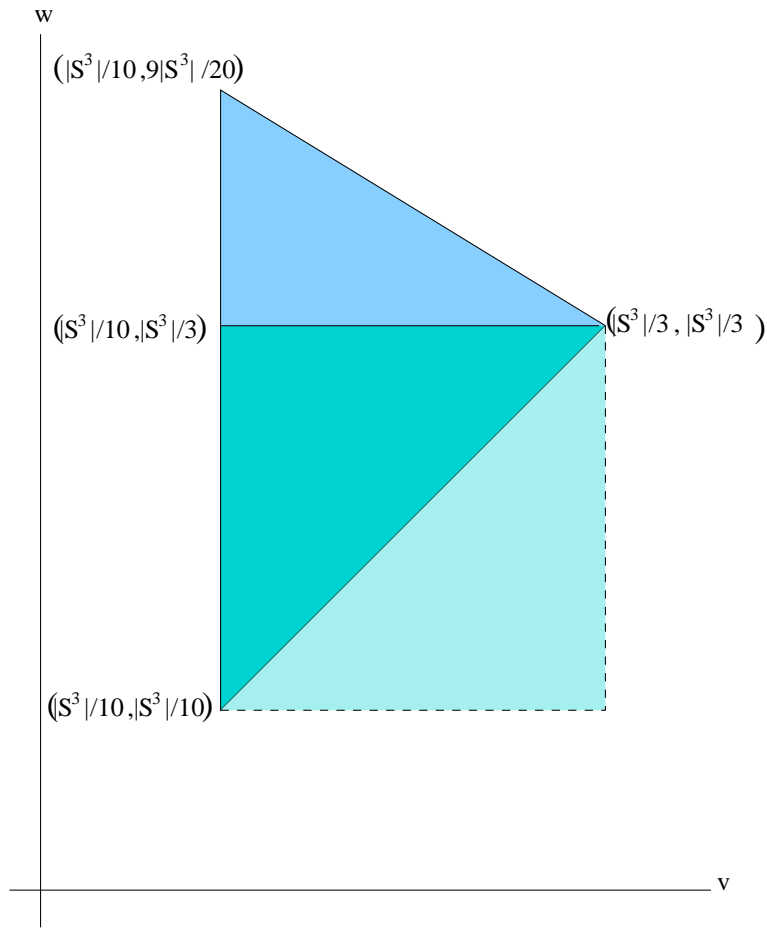}
\caption{\label{s3domain} The triangular domain of the $(v,w)$ where numerical analysis is implemented.}
\end{center}
\end{figure}

\begin{figure}
\begin{center}
\includegraphics[width=4in]{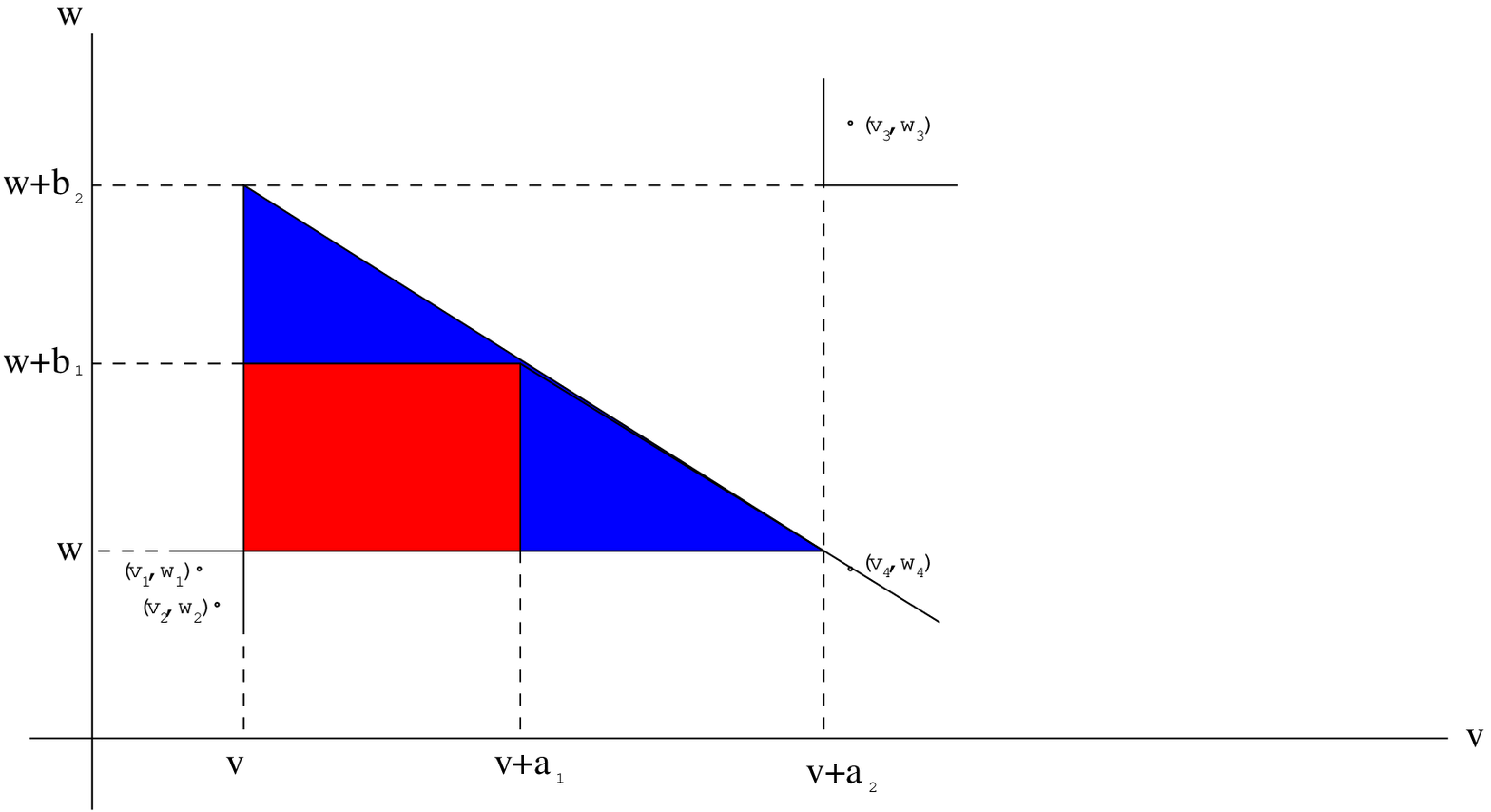}
\caption{\label{prooftriangle} The steps needed to show that $g>h$ on a triangular domain with hypotenuse along the line $w=u.$}
\end{center}
\end{figure}

 The lower triangle can be replaced by the rectangle containing it, which can then be subdivided into smaller rectangles by the argument above.\\
 
 Mathematica uses computer algebra (with infinite precision variables). When checking inequalities, it reports true, false, or null if it cannot tell for sure \cite{W}.\\

\paragraph{\textbf{The computer proof in $\Hth$}}
In $\Hth$, rectangular regions that tile the area between two lines through the origin are considered (See Figures \ref{h3domain} and \ref{h3_proof_outline}). The function $g$ is strictly increasing, so $g(v_0, w_0)$ is less than any value of $g(v, w)$ on a rectangular region with $v_0 \leq v$ and $w_0 \leq w$. We can assure that these two conditions hold if we take a guess at the curvature that we know to be too low (corresponding to a volume that is too high). Then, we increase the curvature (decreasing the volume) of the sphere by multiplying it by some constant. This process is done first with a relatively large constant until the curvature corresponds to volume that is too low;  then the curvature is divided by the constant and the curvature is increased by a smaller constant until it corresponds to a a volume that is less then the volume being approximated. Using this curvature for our surface area calculations under-approximates the surface area of single bubbles.\\

%%%%%%%%%%%%%%%%%%%%%%%%%%%%%%%%%%%%%%%%%%%%%%%%%%%%%%%%%%%%%%%
%%%%%%%%%%%%%%%%%%%%%%%%%%%%%%%%%%%%%%%%%%%%%%%%%%%%%%%%%%%%%%%
%%%%%%%%%%%%%%%%%%%%%%%%%%%%%%%%%%%%%%%%%%%%%%%%%%%%%%%%%%%%%%%
\begin{figure}
\begin{center}
\includegraphics[width=4in]{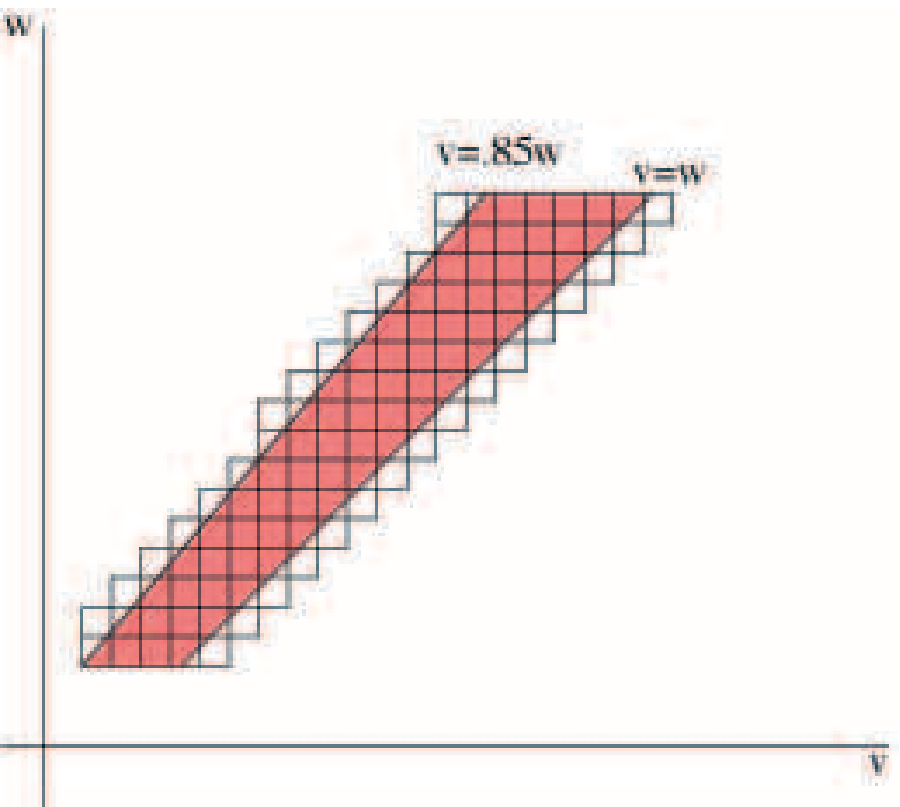}
\caption{\label{h3domain} Showing that the Hutchings function is positive on the rectangles covering the shaded region shows that the function is positive on the shaded region.}
\end{center}
\end{figure}
%%%%%%%%%%%%%%%%%%%%%%%%%%%%%%%%%%%%%%%%%%%%%%%%%%%%%%%%%%%%%%%
%%%%%%%%%%%%%%%%%%%%%%%%%%%%%%%%%%%%%%%%%%%%%%%%%%%%%%%%%%%%%%%
%%%%%%%%%%%%%%%%%%%%%%%%%%%%%%%%%%%%%%%%%%%%%%%%%%%%%%%%%%%%%%%
 
  For $A(v, w)$,   the double bubble is completely determined by the mean curvatures of the outer caps.  Thus, under-approximating the curvatures of these caps (by an analogous method to the single bubble case) corresponds to a double bubble with larger volumes and hence more surface area. \\

We use Mathematica's built-in $SetAccuracy$ function, which assures accuracy on 25 binary digits right of the decimal point in our important computations. When we calculate
 curvature, we check that the calculated associated volume $VolSphere[k]$ is less than $v-2^{-24},$ which guarantees that the associated volume is less that $v$. Hence the associated area will
 be less than A. Finally, we subtract $2^{-24}$ from our calculated area to be sure that we have a lower bound.\\
 
 We check that the calculated associated volumes $VolBubV[k1, k2]$ and\\
  $VolBubW[k1, k2]$  exceed v and w by $2^{-23}=2\cdot 2^{-24}$, because sometimes there are two numbers summed 
 in the calculation. Hence the associated area will be greater than A. We add $3\cdot 2^{-24}$ to the calculated area to be sure.\\

For Claim \ref{ray_claim} the proof-function is restricted to showing that the Hutchings function is positive only on rectangles that include the line $v=.85w$ (See Figure \ref{h3domain}).

\begin{rem}
In $\Sth$, the proof-function determines the size of each polygon via a recursive algorithm. Both proof-functions (Claims \ref{rectangle_claim} and \ref{triangle_claim}) complete in well under six hours each depending on computer speed. In $\Hth$, the proof-functions (Claims \ref{smallest_claim}-\ref{largest_claim}) take well over 150 hours to complete. The function is broken into pieces in order to maximize efficiency by choosing the rectangle size to be relatively large on a given region. Finally, the claim dealing with the ray $v=.85w$ (Claim \ref{ray_claim}) completes in under three hours. 
\end{rem}
%%%%%%%%%%%%%%%%%%%%%%%%%%%%%%%%%%%%%%%%%%%%%%%%%%%%%%%%%%
\subsection{Program implementation}

This section discusses the code in the appendix. The computer proof in $\Sth$ uses two proof functions. The first is called ProofFunctionTriangle. The other is called ProofFunctionRectangle. The following claims can be proved by examining and running the code in the Appendix (Section \ref{appendix_with_code}). The proof-functions returns a 1 if the Hutchings function is positive on the rectangle or triangle it is called upon.

\paragraph{Implementation in $\Sth$}

\begin{claim}
The computed function A(v\_, error\_) returns a lower bound on the area of a sphere in $\Sth$ with given volume v.
\end{claim}

\begin{claim}
The function A(v\_, w\_, VError\_, WError\_, changeToV\_, changeToW\_) gives an upper bound on A(v,w). 
\end{claim}

\begin{claim}\label{rectangle_claim}
If ProofFunctionRectangle returns a 1, then the Hutchings Function is positive on this region. 
\end{claim} 

\begin{claim}\label{triangle_claim}
If ProofFunctionTriangle returns a 1, then the Hutchings Function is positive on this region. 
\end{claim} 

\begin{claim}\label{rectangle_comp_proof}
ProofFunctionRectangle[lhsF2, rhsF2, 
\{VolOfS3/10, VolOfS3/10\}, \{VolOfS3/3,\
   VolOfS3/3\}, lhsF2[VolOfS3/10,\\
    VolOfS3/10], rhsF2[VolOfS3/10,\ VolOfS3/10], 
    lhsF2[VolOfS3/3, VolOfS3/3], rhsF2[VolOfS3/3, VolOfS3/3], temp] returns a 1 where 
   
\begin{eqnarray*}
lhsF2 &=  2*AreaSphereGivenVolume[v/2, VError/2] +\\
 & AreaSphereGivenVolume[w,WError] + \\
& AreaSphereGivenVolume[v + w, VError]
\end{eqnarray*}

and

\begin{eqnarray*}
    rhsF2 =  2*A[v, w, VError, WError, ChangeInV, ChangeInW ].
\end{eqnarray*}

\end{claim}

\begin{claim}\label{triangle_comp_proof}
ProofFunctionTriangle[lhsF2, rhsF2, VolOfS3/10, \\
VolOfS3/3, VolOfS3/3, 9*VolOfS3/20, lhsF2[VolOfS3/10,\\ 
VolOfS3/3], rhsF2[VolOfS3/10,
 VolOfS3/3], lhsF2[VolOfS3/10,\\
  9*VolOfS3/20], rhsF2[VolOfS3/10, 
   9*VolOfS3/20], \\
   lhsF2[VolOfS3/3,VolOfS3/3], rhsF2[VolOfS3/3,VolOfS3/3],\\
    temp] returns a 1 where 
    
\begin{eqnarray*}
lhsF2 &=  2*AreaSphereGivenVolume[v/2, VError/2] +\\
 & AreaSphereGivenVolume[w,WError] + \\
& AreaSphereGivenVolume[v + w, VError]
\end{eqnarray*}

and

\begin{eqnarray*}
    rhsF2 =  2*A[v, w, VError, WError, ChangeInV, ChangeInW ].
\end{eqnarray*}

\end{claim}

\paragraph{Implementation in $\Hth$}
The following claims can be proved by examining and running the code in the appendix. 
The program stops if it finds a rectangle where $g \leq h$. Therefore,  if the
 proof-function completes, the Hutchings function is positive on the region it was
  called on (Claim \ref{basic_claim}).

\begin{claim}\label{basic_claim}
If ArrayFillingProof completes then the Hutchings function is positive on the region it was called on. 
\end{claim}

\begin{claim}\label{smallest_claim}
ArrayFillingProof[.002329, .00274, .01, .00001,\\
 .00001,11.46, 10.95, .9999, .9995] completes.    
\end{claim}

\begin{claim}
ArrayFillingProof[.0085, .01, .1, .00005, .00005, 7.475, 7.15, .9999, .9995]  completes.    
\end{claim}

\begin{claim}
ArrayFillingProof[.085, .1, 1., .0005, .0005, 3.56, 3.415, .9999, .9995] completes.    
\end{claim}

\begin{claim}
ArrayFillingProof[.85, 1.,15., .005, .005, 1.849, 1.787,  .9999, .9995] completes.    
\end{claim}

\begin{claim}
ArrayFillingProof[12.75, 15.,25., .01, .01,1.15204, 1.1352,  .9999, .9995] completes.
\end{claim}

\begin{claim}
ArrayFillingProof[21.25, 25.,45., .02, .02,  1.1027, 1.09054,  .9999, .9995] completes.
\end{claim}

\begin{claim}
ArrayFillingProof[38.25, 45.,65., .02, .02, 1.0637, 1.05562,  .9999, .9995] completes.
\end{claim}

\begin{claim}
ArrayFillingProof[55.25, 65., 85., .02, .02, 1.04658, 1.040469,  .9999, .9995] completes.
\end{claim}

\begin{claim}
ArrayFillingProof[72.25, 85., 110., .015, .015, \\
1.03684, 1.031905,  .9999, .9995] completes.
\end{claim}

\begin{claim}
ArrayFillingProof[93.5, 110., 130., .015, .015, \\
1.02927, 1.025276,  .9999, .9995] completes.
\end{claim}

\begin{claim}\label{largest_claim}
ArrayFillingProof[110.5, 130., 150., .015, .015,\\
 1.025161, 1.021693,  .9999, .9995] completes.
\end{claim}

Finally if we adjust the ArrayFillingProof, so that it only shows that the Hutchings function is positive
along the line $v=.85w$,
\begin{claim}\label{ray_claim}
ArrayFillingProof[127.5, 150., 300., .01, .01, 1.022077, 1.019009, .9999, .9995]  completes.
\end{claim}

\subsection{Main propositions}

The following propositions consolidate the claims that the computer code runs showing the the Hutchings function
is positive on the desired domains. \Prop\ \ref{hf_positive_on_ray_w_150} uses a modified version of the code that checks only that
the Hutchings function is positive on a small domain that includes the line $v=.85w$.

\begin{prop} \label {sth_computer_proof}
If an area-minimizing double bubble in $\Sth$ encloses volumes
$.1 \leq \bar{v} \leq min\{\bar{w}, 1-2\bar{w}\}$, then the Hutchings
function $F_{\Sth} (v, w)
> 0$.
\end{prop}

\begin{proof}
This follows from implementing the code in the appendix (Claims \ref{rectangle_comp_proof}, \ref{triangle_comp_proof}), and Claims \ref{rectangle_claim} and \ref{triangle_claim}.
\end{proof}

\begin{prop} \label{smaller_more_than_.85_larger_implies_connected}
If an area-minimizing double bubble in $\Hth$ encloses volumes $v, w$
such that $.002743 \leq v \leq w \leq 150$ and $v \geq .85 w$ then the
Hutchings function $F_{\Hth}(v,w)>0$.
\end{prop}

\begin{proof}
This follows from Claim \ref{basic_claim} and Claims \ref{smallest_claim}-\ref{largest_claim}.
\end{proof}

\begin{prop}\label{hf_positive_on_ray_w_150}
The Hutchings function $F_{\Hth}(v,w)>0$, for volume pairs $(v,w)$, where $.85v=w$ and $w \in [150, 300]$.
\end{prop}

\begin{proof}
This follows from Claim \ref{basic_claim} and Claim \ref{ray_claim}.
\end{proof}

%%%%%%%%%%%%%%%%%%%%%%%%%%%%%%%%%%%%%%%%%%%%%%%%%%%%%%%%%%%%%%%%%%%%%%%%%%%%%%%%%
%%%%%%%%%%%%%%%%%%%%%%%%%%%%%%%%%%%%%%%%%%%%%%%%%%%%%%%%%%%%%%%%%%%%%%%%%%%%%%%%%
%%%%%%%%%%%%%%%%%%%%%%%%%%%%%%%%%%%%%%%%%%%%%%%%%%%%%%%%%%%%%%%%%%%%%%%%%%%%%%%%%
\section{Double bubbles in $\Sth$ and $\Hth$} \label{sth_and_hth}
%%%%%%%%%%%%%%%%%%%%%%%%%%%%%%%%%%%%%%%%%%%%%%%%%%%%%%%%%%%%%%%%%%%%%%%%%%%%%%%%%
%%%%%%%%%%%%%%%%%%%%%%%%%%%%%%%%%%%%%%%%%%%%%%%%%%%%%%%%%%%%%%%%%%%%%%%%%%%%%%%%%
%%%%%%%%%%%%%%%%%%%%%%%%%%%%%%%%%%%%%%%%%%%%%%%%%%%%%%%%%%%%%%%%%%%%%%%%%%%%%%%%%

Theorems \ref{main_sth} and \ref{main_hth} provide our main results on double bubbles in $\Sth$ and $\Hth$.  They depend on Propositions \ref{conn_implies_sdb} and \ref{conn_implies_sdb_in_hth} of Cotton and Freeman \cite{CF}, after Hutchings, Morgan, Ritor\'e, and Ros \cite{HMRR}, which reduce the proofs to our connectivity results of sections 3-5 via the Hutchings Function.

\begin{prop} \label{conn_implies_sdb}
An area-minimizing double bubble in $\Sth$ for which both enclosed regions and
the exterior are connected must be standard.
\end{prop}

\begin{proof}
See Cotton and Freeman \cite[Proposition 7.3]{CF}.
\end{proof}

\begin{prop} \label{conn_implies_sdb_in_hth}
An area-minimizing double bubble in $\Hth$ for which both enclosed regions are
connected must be standard.
\end{prop}

\begin{proof}
See Cotton and Freeman \cite[Proposition 7.7]{CF}.
\end{proof}

\begin{figure}
\includegraphics[width=1.\textwidth]{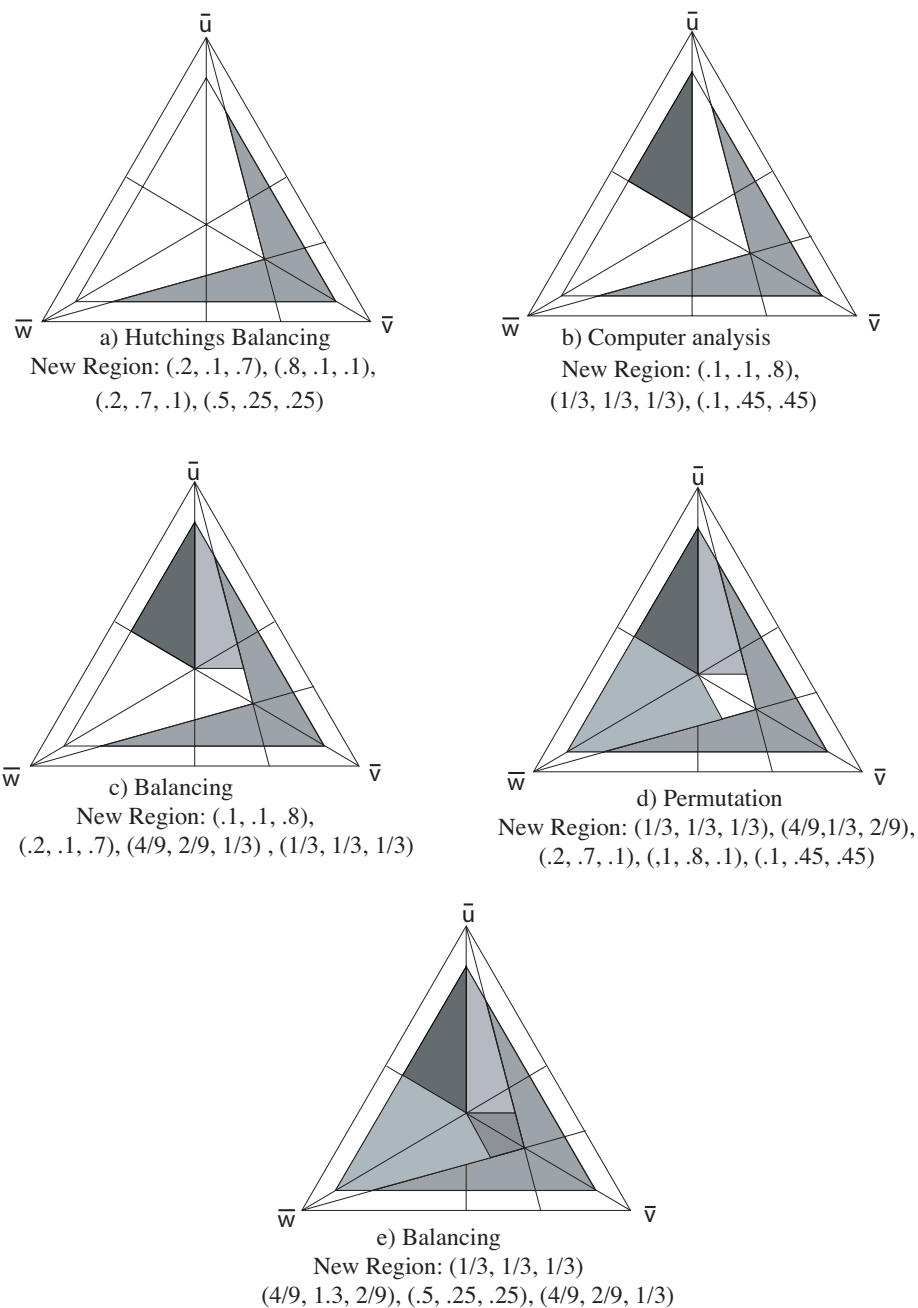}
\caption{ \label{Sth_Computer_Region} An illustration of the steps
taken to cover the whole domain of volumes on $\Sth$.  }
\end{figure}

\begin{figure}
\begin {center}
\includegraphics{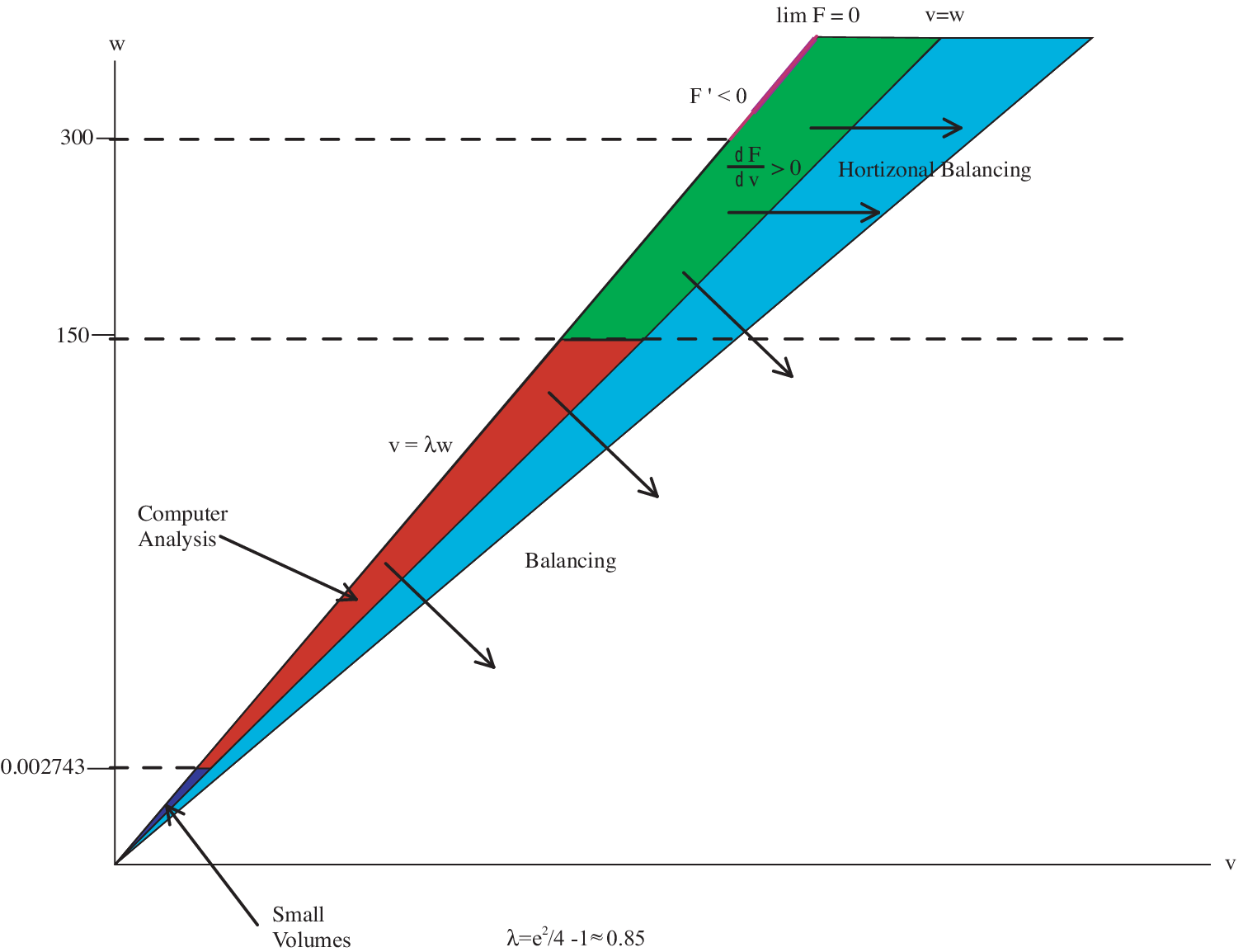}
\end{center}
\caption{\label{h3_proof_outline} The basic steps of the proof that F is positive for $v \geq .85w$ and $w \geq .85v$ in $\Hth$.}
\end{figure}

\begin{thm}[Area-minimizing double bubbles in $\Sth$] \label{main_sth}
An area-minimizing double bubble in $\Sth$ is standard if the volume
of each of the three regions is larger than 10\% of the total volume
of $\Sth$.
\end{thm}

\begin{proof}
By Proposition \ref{conn_implies_sdb}, it suffices to show that
all regions are connected. By Lemma \ref{hutch_balancing}, we know 
that the region of volume $v$ is connected whenever $v>2w$ or $v>2u$ (Figure \ref{Sth_Computer_Region} a).  This
 reduces the problem to showing that the region of volume $v$ is connected
on the pentagonal domain of Figure \ref{Sth_Computer_Region} e  defined by the points:
 $(\bar{v},\bar{w},\bar{u})= (.1,.1,.8), (.1,.8,.1),  (.2,.7,.1), (.5, .25, .25),$ and $(.2,.1,.7).$ 
 By \Prop\ \ref{hutch_pos_implies_conn}, it suffices to show that
the Hutchings Function $F(v,w)$ is positive on this domain.
When $.1 \leq \bar{v} \leq \bar{w} $ and $\bar{v} \leq 1 - 2\bar{w}$, computer analysis
(Proposition \ref{sth_computer_proof}, see Figure \ref{Sth_Computer_Region}b) tells us that the Hutchings
Function $F_{\Sth}(v,w)$ is positive. In particular, for $\bar{v} \leq
\frac{1}{3}$, $F_{\Sth} (v,v) $ is positive. When $\frac{1}{3} \geq \bar{v} > \bar{w}$ and $v \leq 2w$, by S-balancing (\Prop\
\ref{balancing}, see Figure \ref{Sth_Computer_Region}c)  $F_{\Sth}(v,w)$ is positive on the quadrilateral domain
 with endpoints: $(\bar{v},\bar{w},\bar{u})= (.1, .1, .8),
(\frac{1}{3}, \frac{1}{3}, \frac{1}{3})$, $( \frac{4}{9},\frac{2}{9}, \frac{1}{3})$, and
$(.2, .1, .7)$. Thus $F({v},{w}) > 0$ on the pentagonal domain gained by S-balancing in addition to the
 region where the computer analysis was done. Hence by permutation (\Prop\ \ref{permutation}, see
  Figure \ref{Sth_Computer_Region}d),  $F(v,w) > 0$ on the pentagonal domain with endpoints: $(\bar{v},\bar{w},\bar{u})=
(.2, .7, .1), (\frac{4}{9},\frac{1}{3}, \frac{2}{9}), (\frac{1}{3}, \frac{1}{3}, \frac{1}{3}), (.1, .45, .45)$ and $(.1, .8, .1)$. In particular,
 if $\bar{v} \leq .45$, $F({v,v})$ is positive.  By S-balancing (\Prop\ \ref{balancing}, see Figure \ref{Sth_Computer_Region}e)  $F(v,w)$ 
is positive whenever $\bar{w} \leq \bar{v} \leq \bar{2w}$. Thus, $F({v},{w})$ is positive, for on the pentagonal domain 
$(\bar{v},\bar{w},\bar{u})= (.1,.1,.8), (.1,.8,.1),  (.2,.7,.1), (.5, .25, .25),$ and $(.2,.1,.7).$
\end{proof}

\begin{thm}[Area-minimizing double bubbles in $\Hth$] \label{main_hth}
An area-minimizing double bubble in $\Hth$ is standard if the
volume of the smaller region is at least $85\%$ of the volume of
the larger region.
\end{thm}

\begin{proof}
By Proposition \ref{conn_implies_sdb_in_hth} it suffices to show
that the two enclosed regions are connected. By \Prop\ \ref{hutch_pos_implies_conn}, it suffices to show that
the Hutchings Function $F(v,w)$ is positive for $v \geq .85w$ (see Figure \ref{h3_proof_outline}). Computer analysis (Propositions
\ref{smaller_more_than_.85_larger_implies_connected}, \ref{hf_positive_on_ray_w_150}) shows that
when $.002743 \leq v \leq w \leq 150$ and $v \geq .85w$, the
Hutchings Function $F_{\Hth} (v,w)$ is positive. It also shows that the Hutchings function is positive along the line
segment $v =.85w, w \in [150, 300]$.  Since the partial derivative of the Hutchings function with respect to v is positive for 
$w \geq 150$ (Lemma \ref{derivative_with_respect_to_v_is_positive}), $F(v,w)>0$ for $w \in [150, 300]$. Asymptotic
analysis (Theorems \ref{hth_small_volumes} and
\ref{hf_positive_for_large_volumes}) shows that $F_{\Hth} (v, w)$
is positive when $v \leq w < .002743$ or $v/.85\geq w \geq 300$. In
particular, $F_{\Hth} (v, v)$ is positive. Hence when $v \geq w$,
\Prop\ \ref{balancing} shows that $F_{\Hth}(v,w)$ is positive.
\end{proof}

\begin{rem}
In hyperbolic quotient spaces, Theorem \ref{main_hth} and \cite[Remark 4.4]{CCWB}  
prove that the standard double bubble is area minimizing 
for some small volume pairs where the smaller volume is at least 85 percent that of the larger.
\end{rem}

\paragraph{Future work.}

It follows from \Thm\ \ref{main_hth} and balancing (\Lem\ \ref{hutch_balancing}, \Prop\ \ref{balancing})
 that the largest region in $\Hth$
is always connected.  Connectivity of the largest region is the key
hypothesis for the even more general instability argument used in
$\Rfo$  \cite[Section 8]{Rei}; future work might produce an analog of this
result for hyperbolic space.\\

%%%%%%%%%%%%%%%%%%%%%%%%%%%%%%%%%%%%%%%%%%%%%%%%%%%%%%%%%%%%%%%%%%%%%%
%%%%%%%%%%%%%%%%%%%%%%%%%%%%%%%%%%%%%%%%%%%%%%%%%%%%%%%%%%%%%%%%%%%%%%
%%%%%%%%%%%%%%%%%%%%%%%%%%%%%%%%%%%%%%%%%%%%%%%%%%%%%%%%%%%%%%%%%%%%%%
\section{Appendix: program code} \label{appendix_with_code}
%%%%%%%%%%%%%%%%%%%%%%%%%%%%%%%%%%%%%%%%%%%%%%%%%%%%%%%%%%%%%%%%%%%%%%
%%%%%%%%%%%%%%%%%%%%%%%%%%%%%%%%%%%%%%%%%%%%%%%%%%%%%%%%%%%%%%%%%%%%%%
%%%%%%%%%%%%%%%%%%%%%%%%%%%%%%%%%%%%%%%%%%%%%%%%%%%%%%%%%%%%%%%%%%%%%%

\subsection{$\Sth$ code}
{\tiny

\begin{verbatim}


(*The Volume of the 3-sphere with constant curvature 1*)
VolOfS3=Pi*(2*Pi-Sin[2*Pi]);


(*These functions give the relationships of illustrated in Figure 2.1*)
ThetaFN[r1_,r2_]:=ArcTan[Sqrt[3]*(Tan[r2]-Tan[r1])/(Tan[r1]+Tan[r2])];

xFN[r1_,r2_]:=ArcTan[Tan[r1]*(Sqrt[3]/2*Cos[ThetaFN[r1,r2]]+1/2*

Sin[ThetaFN[r1,r2]])];

r3FN[r1_,r2_]:=ArcCot[Cot[r1]-Cot[r2]];

Phi1FN[r1_,r2_]:=ArcSin[Sin[xFN[r1,r2]]/Sin[r1]];

Psi1FN[r1_,r2_]:=Module[{Phi1},Phi1=Phi1FN[r1,r2];If[(Tan[r2]-Tan[r1])/(
      Tan[r1]+Tan[r2])-1/3<0,Pi-Phi1,Phi1]]

Phi2FN[r1_,r2_]:=ArcSin[Sin[xFN[r1,r2]]/Sin[r2]];

Phi3FN[r1_,r2_]:=ArcSin[Sin[xFN[r1,r2]]/Sin[r3FN[r1,r2]]];


(*Volume of a spherical cap in S3 of radius r and subtended by an angle
of 2 Phi *)
VolCap[r_,Phi_]:=If[r\[Equal]Pi/2,(*
      then*)0,(*else*)Pi(r-ArcTan[Cos[Phi]*Tan[r]]-Cos[r]*Sin[r]+Cos[
Phi]*Cos[r]*Sin[r])];

(*Area of a spherical cap in S3 of radius r and subtended by an angle
of 2 Phi *)
AreaCap[r_,Phi_]:=2Pi*(Sin[r]^2)*(1+Cos[Phi]);

(*Volume of a sphere of radius r in S3*)
VolSphere[r_]:=Pi*(2*r-Sin[2*r]);

(*Area of a sphere of radius r in S3*)
AreaSphereGivenRadius[r_]:=4*Pi*(Sin[r])^2

(*Standard double bubbles are compossed of two regions one enclosing the a
smaller volume and a larger volume. The exterior cap of the
 smaller-volume-region has a radius of r1. The
 exterior cap of the larger-volume-region has radius r2.*)
(*This function gives the volume of the smaller region of a standard double bubble 
determined by radii r1, r2*)
VolOneSDB[r1_,r2_]:=Module[{Psi1,r3,Phi3},Psi1=Psi1FN[r1,r2];
    r3=r3FN[r1,r2];
    Phi3=Phi3FN[r1,r2];
    (*return*)VolCap[r1,Psi1]+VolCap[r3,Phi3]]

(*This function gives the volume of the larger region of a standard double bubble 
determined by radii r1, r2*)
VolTwoSDB[r1_,r2_]:=Module[{Phi2,r3,Phi3},Phi2=Phi2FN[r1,r2];
    r3=r3FN[r1,r2];
    Phi3=Phi3FN[r1,r2];
    (*return*)VolCap[r2,Pi-Phi2]-VolCap[r3,Phi3]]

(*When both regions of the double bubble are the same it becomes advantageous 
to consider the one parameter function, that give the volume of one region given a 
radius r.*)
EqualVolGivenRadius[r_]:=Pi/2*(
    2r-Sin[2r])*(1+(Sqrt[2]*Cos[r])/(Sqrt[7+Cos[2r]]))+Pi(ArcTan[Sqrt[2]Sin[r]\
/(Sqrt[7+Cos[2r]])]- (Sqrt[2]*r*Cos[r])/(Sqrt[7+Cos[2r]]));



(*Returns the sum of the areas of the three spherical caps that make up the standard
 double bubble corresponding to the two given radii*)
AreaSDBGivenRadii[r1_,r2_]:=Module[{Psi1,Phi2,Phi3,r3},Psi1=Psi1FN[r1,r2];
    Phi2=Phi2FN[r1,r2];
    Phi3=Phi3FN[r1,r2];
    r3=r3FN[r1,r2];
    (*return*)AreaCap[r1,Pi-Psi1]+AreaCap[r2,Phi2]+AreaCap[r3,Pi-Phi3]]

(*Returns a radius of a sphere in S3 has less area than the sphere of volume v. 
If v <= VolOfS3/2, then this function returns a radius of a sphere with volume in
 the interval [v-error, v]. If v > VolOfS3/2, then
this function returns a radius of a sphere with volume in the interval [v, v+error] *)
RadiusSphere[v_, error_]:=Module[{minRadius, maxRadius, currentRadius, 
                    counter},
      
      counter =0;
      minRadius=0;
      maxRadius = Pi;
      currentRadius = Pi/2;
      
      If[v ² VolOfS3/2,
        
        (*This while loop checks to see if the volume is in the interval [v-error, v] if not
        then it checks to see if it the volume is two high or low. The While loop
         improves the estimates minRadius and maxRadius. While the loop is running 
         minRadius corresponds to sphere of volume less than v-error and maxRadius 
         corresponds to a sphere of volume bigger than v.
         Finally, the sphere of radius currentRadius corresponds to a sphere of volume 
         in [v-error, v] and the loop stops.*)
        
        While[VolSphere[currentRadius] > v || (VolSphere[currentRadius] < v \
-error),
          
          
             If[VolSphere[currentRadius] > v,
            
            maxRadius= currentRadius;
            currentRadius = (currentRadius + minRadius)/2;
            
            ];
          
          
          If[VolSphere[currentRadius] < v-error,
            
            minRadius= currentRadius;
            currentRadius = (currentRadius + maxRadius)/2;
            
            ];
          
          counter++;
          
          ],
        
        (*else if v > volofs3/2*)
        (*This while loop checks to see if the volume is in the interval [v, v+error] if not
        then it checks to see if it the volume is two high or low. The While loop 
        improves the estimate minRadius and maxRadius. While the loop is running 
        minRadius corresponds to sphere of volume less than v and maxRadius 
        corresponds to a sphere of volume bigger than v+error. Finally, the sphere of 
        radius currentRadius corresponds to a sphere of volume in [v, v+error] 
        and the loop stops.*)
        
        
        While[VolSphere[currentRadius] < v || (VolSphere[currentRadius] > v \
+error),
            
            If[VolSphere[currentRadius] < v,
              
              minRadius= currentRadius;
              currentRadius = (currentRadius + maxRadius)/2;
              
              ];
            
            If[VolSphere[currentRadius] > v + error,
              
              maxRadius= currentRadius;
              currentRadius = (currentRadius + minRadius)/2;
              
              ];
            
            counter++;
            
            ];
        ];
      
      (*Print["It took ", counter, "steps."];
        Print["Radius is ", N[currentRadius]];*)
      currentRadius
       ];


(*Gives an overestimate for the sphere of volume v*)
AreaSphereGivenVolume[v_, error_]:=AreaSphereGivenRadius[RadiusSphere[v, \
error]];


(* This function returns a pair of radii corresponding to a standard double bubble 
with volumes bigger than v,w by making very small adjustments to the radii of the
 exterior caps.  This function will only work for v<= w < u. *)

RadiiSDB[v_, w_, vError_, wError_, vChangingFactor_,
      wChangingFactor_]:=Module[{VRadius, WRadius, madeChangeToV, \
madeChangeToW, adjustedVError,
      adjustedWError,  vAdjustment, wAdjustment,volOne, volTwo, counter},
      
      VRadius = Pi/4;
      WRadius = Pi/4;
      counter=0;
      
      
      vAdjustment = vChangingFactor;
      wAdjustment =wChangingFactor;
      
      (*These variables make sure the v, w we approximate are in the region 
      we where A(v,w) increases in v and w*)
      adjustedVError = vError;
      adjustedWError = wError;
      
      (*assurance that we are below line w =u and above line v=w*)
      If[v+2w >= VolOfS3 || v >= w,
        
        Print["ERROR RadiiSDB: v and w are not in increasingRegion!"];
        Print["v+ 2w = " ,N[v+2w]];
        Print["Vol of S3 = ", N[VolOfS3]];
        Print["v =" ,v];
        Print["w =",w];
        ,
        
        While[v+adjustedVError + 2(w+adjustedWError) > VolOfS3 || \
v+adjustedVError > w+adjustedWError,
          
          If[v+adjustedVError > VolOfS3 - 2*w,
            
            adjustedVError= adjstedVError/2;
            
            ];
          
          If[2(w+adjustedWError) > VolOfS3 - (v + adjustedVError),
            
            adjustedWError= adjustedWError/2;
            adjustedVError= adjustedVError/Sqrt[2];
            
            ];
          
          If[v+adjustedVError > w+adjustedWError,
            
            adjustedVError = adjustedVError/2;
            
            ];
          
          ];
        
        While[(VolOneSDB[VRadius, WRadius] < v || (VolOneSDB[VRadius, 
          WRadius]> v+
            adjustedVError)||VolTwoSDB[VRadius, 
              WRadius] < w|| (VolTwoSDB[VRadius, WRadius]> \
w+adjustedWError)),
          
          
          (*These two variable prevent infinite while loop from occuring*)
          madeChangeToV = False;
          madeChangeToW = False;
          
          If[VolOneSDB[VRadius, WRadius]< v,
            
            VRadius = VRadius*vAdjustment;
            madeChangeToV= !madeChangeToV;
            
            ];
          
          (*This set of if statements adjusts one or both of the radii. If it can
           not change one of the radii then it refines the amount that the radii
           are adjusted by.*)
          If[VolOneSDB[VRadius, WRadius]> v+adjustedVError,
            
            
            VRadius = VRadius/vAdjustment;
            madeChangeToV= !madeChangeToV;
            
            ];
          
          
          If[VolTwoSDB[VRadius, WRadius]< w,
            
            
            WRadius = WRadius*wAdjustment;
            madeChangeToW= !madeChangeToW;
            
            ];
          
          If[VolTwoSDB[VRadius, WRadius]> w+adjustedWError,
            
            WRadius = WRadius/wAdjustment;
            madeChangeToW= !madeChangeToW;
            
            
            ];
          
          If[!madeChangeToV,
            
            vAdjustment=(vAdjustment+1)/2;
            
            ];
          
          If[!madeChangeToW,
            
            wAdjustment=(wAdjustment+1)/2;
            
            ];
          
          (*Print["Volume V at ", N[VolOneSDB[VRadius, WRadius]]];
            Print["Volume W at ", N[VolTwoSDB[VRadius, WRadius]]];
            
            Print["While Loop test 
              is ", (VolOneSDB[VRadius, WRadius] < v || (VolOneSDB[VRadius, \
WRadius]> v+adjustedVError)||VolTwoSDB[VRadius, WRadius] < 
                        w || (VolTwoSDB[VRadius, WRadius]>
                   w+adjustedWError))];*)
          
          
          ];
        
        ];
      
      (*Print["It took this many steps: ", counter];*)
      
      {VRadius, WRadius}
      
      ];


(*By symmetry the standard double bubble enclosing volumes (v,w,u) has the same 
area of the standard double bubble enclosing volumes (u,w,v). So, when w=u we can
 get an overestimate for the area of the standard double bubble enclosing volumes 
 (v,w,u) by geting an overestime for the area of the standard double bubble
  enclosing volumes (u,w,u). This function returns a pair of radii that correspond to a 
  double bubble where the larger volume is equal to the exterior and
 both are bigger than or equal to w*)
RadiiWhenWEqualsU[w_,
 error_]:=Module[{radius, minRadius, maxRadius, counter, adjustedError},
      minRadius =0;
      maxRadius=3Pi/2;
      radius = (minRadius+maxRadius/2);
      adjustedError=error;
      counter=0;
      
      
      If[3w< VolOfS3,
        
        Print["W is too small."],
        
        While[3(w-adjustedError)< VolOfS3,
          
          adjustedError = adjustedError/2;
          
          ];
        
        (*Print["Target Range = (", N[w-error],",", N[w],")" ];*)
        
        While[EqualVolGivenRadius[radius]>w|| 
        EqualVolGivenRadius[radius]<w-adjustedError,
          (*counter++;
            Print["counter = ", counter];*)
          
          
          If[EqualVolGivenRadius[radius]>w,
            
            maxRadius =radius;
            radius = (minRadius+radius)/2;
            
            ];
          
          If[EqualVolGivenRadius[radius]<w-adjustedError,
            
            minRadius =radius;
            radius = (maxRadius+radius)/2;
            
            ];
          
          (*Print["Vol = ",N[EqualVolGivenRadius[radius]]];*)
          
          ];
        (*Print["It took this many steps: ", counter];*)
        
        
        ];
      {radius, radius}
      ];


(*This piece of code should be used when v=w*)
RadiiWhenVEqualsW[v_, error_]:=Module[{radius, 
minRadius, maxRadius, counter, adjustedError},
      minRadius =0;
      maxRadius=Pi;
      radius = (minRadius+maxRadius/2);
      adjustedError=error;
      counter=0;
      
      (*Print["radius= ", radius];*)
      
      If[3v>= VolOfS3,
        
        Print["V is too big."],
        
        
        While[3(v+adjustedError)> VolOfS3,
          
          adjustedError = adjustedError/2;
          
          ];
        
        (*Print["Target Range = (", N[v-error],",", N[v],")" ];*)
        
        While[EqualVolGivenRadius[
      radius]<v|| EqualVolGivenRadius[radius]>v+adjustedError,
          
          
          
          (*counter++;
            Print["counter = ", counter];*)
          
          
          If[EqualVolGivenRadius[radius]>v+adjustedError,
            
            maxRadius =radius;
            radius = (minRadius+radius)/2;
            
            ];
          
          If[EqualVolGivenRadius[radius]<v,
            
            minRadius =radius;
            radius = (maxRadius+radius)/2;
            
            ];
          
          (*Print["Vol = ",N[EqualVolGivenRadius[radius]]];*)
          
          ];
        
        (*Print["It took this many steps: ", counter];*)
                
        ];
      
      {radius, radius}
      
      
      ];


(*This function returns an overestimate for the area of a 
standard double bubble of volumes v, w.*)

A[v_,w_, vError_, 
      wError_, vChangingFactor_, \
wChangingFactor_]:=Module[{radii,u,r1,r2,area},
      
      u=VolOfS3-v-w;
      
      (*when all three volumes are equal*)
      If[v==w && w==u,
        
        area = 6*Pi;
        
        ,
        (*one the line v=w*)
        If[v==w,
            
            radii = RadiiWhenVEqualsW[v, vError];
            area = AreaSDBGivenRadii[radii[[1]], radii[[2]]];
            
            ,
            (*on the line w=u*)
            If[w==u,
                
                radii = RadiiWhenWEqualsU[w, wError];
                area = AreaSDBGivenRadii[radii[[1]], radii[[2]]];
                
                
                ,
                (*If all volumes are distinct*)
           
                   radii = RadiiSDB[v, w, vError, wError, vChangingFactor,
                 wChangingFactor];
                area = AreaSDBGivenRadii[radii[[1]], radii[[2]]];
                
                ];
            
            ];
        ];
      
      area
      ];
      
      
(*This function is described in Section 4 of the paper. It shows that either 
the Hutchings function is on a rectangular domain or it breaks up the 
rectangle into 4 smaller rectangles and checks the code again. The function
completes and returns a 1 if the function is positive on the rectangular domain
and 0 if the approximated function is negative on any part of the domain*)
ProofFunctionRectangle[lhs_,rhs_,p1_,p3_,left11_,right11_,left33_,
      right33_,depth_]:=Module[{decision,x1,y1,x2,y2,x3,
      y3,left12,right12,left13,right13,
      left21,right21,left22,right22,left23,right23,left31,
        right31,left32,right32,d1,d2,d3,d4},depth=1;
      
      
      x1=p1[[1]];
      y1=p1[[2]];
      
      x3=p3[[1]];
      y3=p3[[2]];
      
      x2=(x1+x3)/2;
      y2=(y1+y3)/2;
      
      left12=lhs[x1,y2];right12=rhs[x1,y2];
      left13=lhs[x1,y3];right13=rhs[x1,y3];
      left21=lhs[x2,y1];right21=rhs[x2,y1];
      left22=lhs[x2,y2];right22=rhs[x2,y2];
      left23=lhs[x2,y3];right23=rhs[x2,y3];
      left31=lhs[x3,y1];right31=rhs[x3,y1];
      left32=lhs[x3,y2];right32=rhs[x3,y2];
      
      (*if x1 > y3 then balancing 
      covers the region and we don't need to check it.*)
      If[x1 > y3,
        
        decision =1,
        
        (*otherwise we do*)
        If[left11>right11&&left33>right33,
          
          
          (*then*)
          
          If[Min[left11,left13,left31,left33]>right33,(*then*)Print["Points \
",p1," and ",p3," -- Direct hit!"];
            decision=1,
            
            (*else*)Print["Points ",p1," and ",
            p3," -- Splitting into four."];
            
            If[ProofFunctionRectangle[lhs,rhs,{x1,y1},{x2,y2},left11,right11,
              left22,right22,d1]\[Equal]1&&
              ProofFunctionRectangle[lhs,rhs,{x2,y1},{x3,y2},left21,right21,\
left32,right32,d2]\[Equal]1&&
            ProofFunctionRectangle[
              lhs,rhs,{x1,y2},{x2,y3},left12,right12,left23,right23,
                    d3]\[Equal]1&&ProofFunctionRectangle[lhs,rhs,{x2,y2},{
                    x3,y3},left22,right22,left33,right33,d4]\[Equal]1,
              
              (*then*)
              Print["Points ",p1," and 
                    ",p3," -- Hit after checking four! Depth: ",Max[depth,
                    d1+1,d2+1,d3+1,d4+1]];
              decision=1,
              
              
              (*else*)Print["Proof function failed"];
              
              decision=0;
              (*endif*)];
            depth=Max[depth,d1+1,d2+1,d3+1,d4+1];
            (*endif*)],Print["Oh no for ",p1," and ",p3,"!"];
          Print[N[left11]];
          Print[N[right11]];
          Print[N[left33]];
          Print[N[right33]];
          decision=0;
          ];
        
        (*endif*)
        ];
      (*return*)decision
      ];

(*This function is described in Section 4 of the paper. It shows that either 
the Hutchings function is on a triangular domain or it breaks up the 
triangle into 2 smaller triangles and a rectangle and checks the code again. The function
completes when it has broken the domain into small enough domains to see that the function
is postive and returns a 1 or if the approximated function is negative on any part of the 
domain the function returns a 0.*)

ProofFunctionTriangle[lhs_,rhs_,x1_,y1_,x3_,y3_,left11_,right11_,left13_,\
right13_,left31_,right31_,depth_]:=Module[{decision,x2,y2,left12,right12,\
left21,right21,left22,right22,d1,d2,d3},depth=1;
    
    Print["starting up the process"];
    
    x2=(x1+x3)/2;
    y2=(y1+y3)/2;
    Print["whoa now!"];
    If[left11>right11&&left13>right13&&left31>right31,(*then*)
      
      If[Min[left11,
    left13,left31]>right31,(*then*)Print["Points ",{x1,y1},", ",{x1,
      y3},", and ",{x3,y1}," -- Direct triangle hit!"];
        decision=1,
        
        (*else*)Print["Points ",{x1,y1},", ",{x1,y3},", and ",{x3,y1}," -- 
          Splitting into three."];
        
        left12=lhs[x1,y2];right12=rhs[x1,y2];
        left21=lhs[x2,y1];right21=rhs[x2,y1];
        left22=lhs[x2,y2];right22=rhs[x2,y2];
        
        If[ProofFunctionRectangle[lhs,rhs,{x1,y1},{x2,y2},left11,right11,\
left22,right22,d1]\[Equal]1&&ProofFunctionTriangle[lhs,rhs,x1,y2,x2,y3,left12,
              right12,left13,right13,left22,right22,d2]\[Equal]1&&\
ProofFunctionTriangle[lhs,rhs,x2,y1,x3,y2,left21,
                right21,left22,right22,left31,right31,d3]\[Equal]1,(*then*)
                Print["Points ",{x1,y1},", ",{x1,y3},", and ",{x3,y1}," -- 
            Hit after checking three! Depth: ",Max[depth,d1+1,d2+1,d3+1]];
          decision=1,(*else*)Print[
                "Points ",{x1,y1},", ",{x1,y3},", and ",{x3,y1}," -- Failed 
                after checking three!"];
          decision=0;
          (*endif*)];
        depth=Max[depth,d1+1,d2+1,d3+1];
        (*endif*)],(*else*)Print["Oh no for " {x1,y1},",
             ",{x1,y3},", and ",{x3,y1},"!"];
      Print[N[left11]];
      Print[N[right11]];
      Print[N[left13]];
      Print[N[right13]];
      Print[N[left31]];
      Print[N[right31]];
      decision=0;
      (*endif*)];
    (*return*)decision]


\end{verbatim}

H3 Code
\begin{verbatim}

(*error control in computation the default value of maching*)
 WorkingPrecision->MachinePrecision;
  COMPUTERERROR= 2^-24;
  ACCURACY = 25;


LargestSingCurvature = 16.8;


AreaSphere[k_Real]:=SetAccuracy[(4*Pi)/(-1+k^2)-COMPUTERERROR, ACCURACY];

VolSphere[k_Real]:=SetAccuracy[Pi*(-2*
    ArcCoth[k]+Sinh[2*ArcCoth[k]]),ACCURACY];

ASC1close[k1_Real,k2_Real]:=SetAccuracy[(2*Pi*(1+Sqrt[-((k1^2*(-1+Cos[(1/6)*(\
Pi-6*ArcTan[(Sqrt[3]*(k1-k2))/(
                  k1+k2)])]^2))/(
                            k1^2-Cos[(1/6)*(
                                    Pi-6*ArcTan[(
                                      Sqrt[3]*(k1-
                                        k2))/(k1+k2)])]^2))]))/(-1+k1^2), \
ACCURACY+1];

ASC1far[k1_Real,
      k2_Real]:=SetAccuracy[-((2*Pi*(-1+Sqrt[-((k1^2*(-1+Cos[(1/6)*(Pi-6*\
ArcTan[(Sqrt[3]*(k1-k2))/(
                      k1+k2)])]^2))/(k1^2-Cos[(
                                    1/6)*(Pi-6*ArcTan[(
                                        Sqrt[3]*(
                                        k1-k2))/(k1+k2)])]^2))]))/(-1+
                                        k1^2)), ACCURACY+1];

ASC1[k1_Real,k2_Real]:=SetAccuracy[If[(k1-k2)/(k1+k2)-1/3<0,ASC1close[
      k1,k2],ASC1far[k1,k2]], ACCURACY+1];

ASC2[k1_Real,k2_Real]:=SetAccuracy[(2*Pi*(1+Sqrt[(k1^2-k2^2*Cos[(1/6)*(Pi-6*
            ArcTan[(Sqrt[3]*(k1-k2))/(k1+
                          k2)])]^2)/(
                            k1^2-Cos[(1/
                              6)*(Pi-6*ArcTan[(Sqrt[3]*(k1-k2))/(k1+
                                        k2)])]^2)]))/(-1+k2^2), ACCURACY+1];

ASC3[k1_Real,k2_Real]:=SetAccuracy[(2*Pi*(1-Sqrt[(k1^2-(k1-k2)^2*Cos[(
      1/6)*(Pi-6*
            ArcTan[(Sqrt[3]*(
                  k1-k2))/(k1+
                        k2)])]^2)/(k1^2-
                            Cos[(1/6)*(
                              Pi-6*ArcTan[(
                                Sqrt[3]*(k1-
                                      k2))/(k1+k2)])]^2)]))/(-1+(k1-k2)^2), \
ACCURACY+1];

AreaDblBubble[k1_Real,k2_Real]:=SetAccuracy[ASC1[k1,k2]+ASC2[k1,
    k2]+ASC3[k1,k2]+3*COMPUTERERROR, ACCURACY];

VolCap1close[k1_Real,k2_Real]:=-((1/(-1+k1^2))*(Pi*((-1+k1^2)*
    ArcCoth[k1]+(-1+k1^2)*ArcTanh[Sqrt[-((k1^2*(-1+Cos[(1/6)*(Pi-6*ArcTan[(\
Sqrt[3]*(k1-k2))/(k1+k2)])]^2))/(k1^2-Cos[(1/6)*(Pi-6*ArcTan[(Sqrt[3]*(k1-k2))\
/(k1+k2)])]^2))]/k1]-k1*(1+Sqrt[-((
                                        k1^2*(-1+Cos[(1/6)*(Pi-6*ArcTan[(Sqrt[\
3]*(k1-k2))/(k1+k2)])]^2))/(k1^2-Cos[(1/6)*(Pi-6*ArcTan[(
                                    Sqrt[3]*(k1-k2))/(k1+k2)])]^2))]))));

VolCap1far[k1_Real,k2_Real]:=Pi*(-ArcCoth[k1]+(1/(-1+k1^2))*(k1+(-1+k1^2)*\
ArcTanh[Sqrt[-((k1^2*(-1+Cos[(1/6)*(Pi-6*ArcTan[(Sqrt[
                3]*(k1-k2))/(
                        k1+k2)])]^2))/(k1^2-Cos[(1/6)*(Pi-6*ArcTan[(Sqrt[
                                  3]*(k1-
                                      k2))/(k1+
                                        k2)])]^2))]/
                                        k1]-k1*
                                        Sqrt[-((k1^2*(-1+Cos[(1/6)*(Pi-6*\
ArcTan[(Sqrt[3]*(k1-k2))/(k1+k2)])]^2))/(k1^2-Cos[(
                              1/6)*(Pi-6*ArcTan[(Sqrt[3]*(k1-k2))/(k1+
                                        k2)])]^2))]));

VolCap1[k1_Real,k2_Real]:=If[(k1-k2)/(k1+k2)-1/3<0,
    VolCap1close[k1,k2],VolCap1far[k1,k2]];

VolCap2[k1_Real,k2_Real]:=Pi*(-(k2/(1-k2^2))-ArcCoth[k2]-
      ArcTanh[Sqrt[(k1^2-k2^2*Cos[(1/6)*(Pi-6*ArcTan[(Sqrt[3]*(
          k1-k2))/(k1+k2)])]^2)/(k1^2-Cos[(
                        1/6)*(Pi-6*
                          ArcTan[(Sqrt[3]*(
                                  k1-k2))/(k1+k2)])]^2)]/k2]-(k2*Sqrt[(k1^2-\
k2^2*Cos[(1/6)*(Pi-6*ArcTan[(Sqrt[3]*(k1-k2))/(k1+k2)])]^2)/(k1^2-
                      Cos[(1/6)*(Pi-6*ArcTan[(Sqrt[
                                3]*(k1-k2))/(k1+k2)])]^2)])/(1-k2^2));

VolCap3[k1_Real,
      k2_Real]:=Re[Pi*(ArcTanh[Sqrt[(k1^2-(k1-k2)^2*Cos[(1/6)*(Pi-6*ArcTan[(\
Sqrt[3]*(k1-k2))/(k1+k2)])]^2)/(k1^2-
                          Cos[(1/6)*(Pi-6*ArcTan[(Sqrt[
                                    3]*(k1-k2))/(k1+k2)])]^2)]/(k1-k2)]+(1/(-\
1+k1^2-2*k1*k2+k2^2))*(k1-k2-(-1+k1^2-2*
                    k1*k2+k2^2)*
                        ArcCoth[
                            k1-k2]+(-k1+
                              k2)*Sqrt[(k1^2-(k1-k2)^2*Cos[(1/6)*(Pi-6*ArcTan[\
(Sqrt[3]*(k1-k2))/(k1+k2)])]^2)/(k1^2-Cos[(1/6)*(Pi-6*ArcTan[(
            Sqrt[3]*(k1-k2))/(k1+k2)])]^2)]))];

VolBub1[k1_Real,k2_Real]:=If[k1 < 1 || k2 < 1,
        Print["VolBub1 Error : k1 or k2 too small"];
      Print["k1 = ", k1, " k2 = ", k2 ];
      ,If[k1 \[Equal] k2,VolCap1[k1,k2]
        ,
        VolCap1[k1,k2]+VolCap3[k1,k2]
        ]
      ];


VolBub2[k1_Real,k2_Real]:=If[k1\[Equal]k2, VolCap2[k1,k2]
      ,
      VolCap2[k1,k2]-VolCap3[k1,k2]
      ];

VolBub1A[k1_Real, k2_Real]:=If[k1>= k2,
      
       VolBub1[k1, k2]
      ,
      (*Print["VolBub1A curvatures switched"];*)
      
      VolBub2[k2, k1]
      ];

VolBub2A[k1_Real, k2_Real]:=If[k1>= k2,
      
       VolBub2[k1, k2]
      ,
      (*Print["VolBub2A curvatures switched to ", VolBub1[k1, k2]];*)
      VolBub1[k2, k1]
      ];

VolBubV[k1_Real, k2_Real] := SetAccuracy[VolBub1A[k1, k2], ACCURACY];


VolBubW[k1_Real, k2_Real] := SetAccuracy[VolBub2A[k1, k2], 15];

(*gives an underappoxiamation for the curvature of a sphere in H3*)
CurvatureEstimatorSingle[v_Real] := If[v > .0001,
      If[ v < .008,
        8.109,
        
        If[ v<.02,
          
          6.,
          
          If[
            (*.02² v < 1*)
            v < 1,
            
            1.84871,
            If[
              (*1² v < 10*)
              v<10,
              1.19394,
              
              If[
                (*20² v < 30 *)
                v < 30,
                1.08108,
                
                If[
                  (*30 ² v < 50*)
                  v<50,
                  
                  1.05231,
                  
                  If[
                    (*50 ² v < 60*)
                    v< 60,
                    
                    1.04456,
                    
                    If[
                      (*60 <= v < 70*)
                      v< 70,
                      
                      1.03883,
                      
                      If[
                        (*70 ² v < 80*)
                        v< 80,
                        
                        1.03437,
                        
                        If[
                          (*80 ² v < 90*)
                          v< 90,
                          
                          1.03086,
                          
                          If[
                            (*90² v < 110*)
                            v<110,
                          
                            
                            1.0256,
                            If[
                              (*110< v < 200*)
                              v < 200,
                              1.01468,
                              (*else 550>v³ 110*)
                              1.00509
                              ]
                            ]
                          ]
                        ]
                      ]
                    ]
                  ]
                ]
              ]
            ]
          ]
        ],
      (*else if v < .0001*)
      Print["volume to small"];
      (*return*) 
      -1
      
      ];

(*This function uses interpolation to get asure that the curvature used in \
the approximation is less than the given v within
    an interval that is small. The constant term COMPUTERERROR assures that \
the precision of the computer does not affect our lower bound*)
CurvaturefromVolSingle[v_Real, error_Real]:=Module[{
        curvature, smallK =  CurvatureEstimatorSingle[v], bigK= \
LargestSingCurvature, counter =0},
      
       curvature = smallK;
      
      (**)
      While[VolSphere[curvature] > v - COMPUTERERROR || VolSphere[curvature] \
< v-error ,
        
        curvature = (bigK+smallK)/2;
        counter = counter +1;
        
        (*this tightens the interval that curvatures can be in*)
        If[VolSphere[curvature] >  v - COMPUTERERROR, 
          
          smallK= curvature;
          
          ,If[VolSphere[curvature] < v - error,
              
              bigK  = curvature;
              
              ];
          ];
        ];
      (*Print["It took ", counter, "steps." ];*)
      curvature
      ];



(*this function returns curvature pair that is at least one box size over \
then returns the array with the new curvature pair and the changes for each \
curvature though this seems to do the same thing as CurvaturesfromVolDouble[] \
this doesn't limit the size of the change*)
MakeChangeOneBoxSizeOver[k1_Real, k2_Real, volume1_Real, BoxSize_Real, \
change1_Real] := Module[{minVolChange, maxVolChange,adjCurvature1, \
adjCurvature2,  adjChange1, finalArray, counter, slope, closestOverShot, \
closestUnderShot, lastCurvatureGuess},
      
      minVolChange = .8*BoxSize;
      maxVolChange = 1.7*BoxSize;
      
      adjCurvature1 = change1*k1;
      adjCurvature2 = k2;
      adjChange1 = change1;
      counter=0;
      
      closestOverShot= k2;
      closestUnderShot = k1;
      
      slope = (VolBub1A[adjCurvature1, adjCurvature2] - \
volume1)/(adjCurvature1 - k1);
      
      adjCurvature1 =  adjCurvature1 +  BoxSize/slope;
      
      (*Print["curvatures, k1 = ",adjCurvature1, " k2 is ", k2];*)
      
      While[ (VolBub1A[adjCurvature1,adjCurvature2]- volume1< minVolChange || \
VolBub1A[adjCurvature1,adjCurvature2]- volume1 > maxVolChange) && \
adjCurvature1 > 1  && counter < 10,
        
        If[VolBub1A[adjCurvature1,adjCurvature2]- volume1< minVolChange,
          
          (*Print["hi test 1 is ", VolBub1A[
          adjCurvature1,adjCurvature2]- volume1< minVolChange];*)
          
          lastCurvatureGuess = adjCurvature1;
          
          adjCurvature1 = (adjCurvature1 +  closestOverShot)/2;
          
          closestOverShot = lastCurvatureGuess;
          ];
        
        
        If[VolBub1A[adjCurvature1,adjCurvature2]- volume1> maxVolChange,
          
          
          (*Print["hi test 2 is ", VolBub1A[adjCurvature1,adjCurvature2]- 
          volume1> maxVolChange];*)
          
          lastCurvatureGuess = adjCurvature1;
          
          adjCurvature1 = (adjCurvature1 +  closestUnderShot)/2;
          
          closestUnderShot = lastCurvatureGuess;
          ];
        
        (*Print["trial", counter, " vol is \
",VolBub[adjCurvature1,adjCurvature2]];*)
        
        counter= counter+1;
        (*Print["counter = ", counter];
          Print["vol1 =", 
          VolBub1A[adjCurvature1,adjCurvature2] ];
          Print["change = \
", adjChange1];
          Print["curvature ", adjCurvature1 ];
          *)
        
        ];
      
      If[adjCurvature1 <= 1,
        
        Print["MakeChangeOneBoxSizeOver error curvature too small"];
         ];
      (*
        Print["vol1 =", VolBub1[adjCurvature1,adjCurvature2] ];
        Print["change = ", adjChange1];
        Print["curvature ", adjCurvature1];
        *)
      adjChange1 = adjCurvature1/k1;
      
      {adjCurvature1, adjChange1}
      ];


(*this function returns curvature pair that is at least one box size over \
then returns the array with the new curvature pair and the changes for each \
curvature though this seems to do the same thing as CurvaturesfromVolDouble[] \
this doesn't limit the size of the change*)
MakeChangeOneBoxSizeUp[k1_Real, k2_Real, volume2_Real, BoxSize_Real, \
change2_Real] := Module[{minVolChange, maxVolChange,adjCurvature1, \
adjCurvature2,  adjChange1, counter, slope, lastCurvatureGuess, \
closestUnderShot, closestOverShot},
      
      minVolChange = 1.*BoxSize;
      maxVolChange = 2.*BoxSize;
      
      adjCurvature1 = k1;
      adjCurvature2 = k2*change2;
      adjChange2 = change2;
      closestUnderShot=k2;
      closestOverShot = 1.; 
      (*Might want to change this!*)
      
      counter=0;
      
      (*Print["hi"];
        Print["test =", (VolBub2[adjCurvature1,adjCurvature2]- volume2< \
minVolChange || VolBub2[adjCurvature1,
              adjCurvature2]- volume2 > maxVolChange) && 
          adjCurvature2 > 1  && counter < 10];
        *)
      
      slope = (
      VolBub2A[adjCurvature1, adjCurvature2] - volume2)/(adjCurvature2 - k2);
      
      adjCurvature2 =  adjCurvature2 +  BoxSize/slope;
      
      While[ (VolBub2A[adjCurvature1,
      adjCurvature2]- volume2< minVolChange || \
VolBub2A[adjCurvature1,adjCurvature2]- 
          volume2 > maxVolChange) && adjCurvature2 > 1  && counter < 10,
        
        counter= counter+1;
        (*Print["counter = ", counter];
          Print["vol2 =", VolBub2[adjCurvature1,adjCurvature2] ];
          Print["change = ", adjChange2];
          Print["curvature ", adjCurvature1 ];
          *)
        
        If[VolBub2A[adjCurvature1,adjCurvature2]- volume1< minVolChange,
          
          lastCurvatureGuess = adjCurvature2;
          
          adjCurvature2 = (adjCurvature2 +  closestOverShot)/2;
          
          closestOverShot = lastCurvatureGuess;
          ];
        
        
        If[VolBub2A[adjCurvature1,adjCurvature2]- volume1> maxVolChange,
          
          lastCurvatureGuess = adjCurvature1;
          
          adjCurvature2 = (adjCurvature2 +  closestUnderShot)/2;
          
          closestUnderShot = lastCurvatureGuess;
          ];
        
         
        ];
      
      (*Print["counter = ", counter];*)
      
      If[adjCurvature2 <= 1,
        
        Print["MakeChangeOneBoxSizeUp error curvature too small"];
         ];
      (*
        Print["vol2 =", VolBub2[adjCurvature1,adjCurvature2] ];
        Print["change = ", adjChange2];
        Print["curvature ", adjCurvature2];
        *)
      adjChange2 = adjCurvature2/k2;
      
      {adjCurvature2, adjChange2}
      ];



(*This function tells whether the point is in given rectangle with lower left \
corner xvalueBox, yValueBox and width boxWidth and heigh boxHeight*)
IsPointinBox[xValuePoint_Real, yValuePoint_Real, xValueBox_Real, \
yValueBox_Real, boxWidth_Real, boxHeight_Real] := 
    (xValuePoint> xValueBox+ COMPUTERERROR && xValuePoint < (xValueBox + \
boxWidth) && yValuePoint > yValueBox + COMPUTERERROR && 
      yValuePoint < (yValueBox + boxHeight));



(*This function tells whether the point has a x value and a yalue bigger than \
a given x, and y value *)
IsPointToUpperRight[xValuePoint_Real, yValuePoint_Real, xValueBox_Real, \
yValueBox_Real] := 
    (xValuePoint> xValueBox) && yValuePoint > yValueBox ;

(*This function is supposed manipulate a given curvature pair for volumes \
vol1 and vol2 and get a curvature pair corresponding to volumes in inside a \
box with v1, v2 as its lower left corner*)
(*right now this method needs to dyanically choose its changing size to avoid \
infinite loops*)
(*Returns a two element list of curvatures*)
(*v1, v2 should correspond to the point in the lower right of the rectangle \
containing VolBub1A[k1, k2], VolBub2A[k1, k2]
    the curvatures k1, k2 should be a point in the middle of the rectangle*)


NextCurvaturePairinArray[ k1_Real, k2_Real,v1_Real, v2_Real, change1_Real, \
change2_Real, scalingFactor1_Real, scalingFactor2_Real] \
:=Module[{curvatureComponent11, curvatureComponent12, curvatureComponent21, \
curvatureComponent22, startingVolume1, startingVolume2, volumeComponent11, \
volumeComponent12, volumeComponent21, volumeComponent22, alpha, beta, \
volumeOneTarget, volumeTwoTarget, curvaturePair, volumeChangeMatrix, 
      startingVoltoTargetMatrix, alphaBetaMatrix, counter, \
oneBoxSizeOverArray, oneBoxSizeUpArray},
      
      (*adjustment arrays*)
      oneBoxSizeOverArray = 
        MakeChangeOneBoxSizeOver[k1, k2, 
      VolBub1A[k1,k2], change1, scalingFactor1];
      
      oneBoxSizeUpArray = MakeChangeOneBoxSizeUp[k1, k2, VolBub2A[k1, k2], \
change2,  scalingFactor2];
      
      (*Declaration of initial curvaturevectors*)
      
      curvatureComponent11 = oneBoxSizeOverArray[[1]];
      
      curvatureComponent12 = k2;
      curvatureComponent21 = k1;
      curvatureComponent22 = oneBoxSizeUpArray[[1]];
      counter=0;
      
      (*Print["c11 = ", curvatureComponent11];
        Print["c12 = ", curvatureComponent12];
        Print["c21 = ", curvatureComponent21];
        Print["c22 = ", curvatureComponent22];*)
      
      
      
      (*declaration of initial volumevectors*)
      startingVolume1 = VolBub1A[k1, k2];
      startingVolume2 = VolBub2A[k1, k2];
      volumeComponent11 = VolBub1A[
      curvatureComponent11, curvatureComponent12];
      volumeComponent12 = VolBub2A[curvatureComponent11, \
curvatureComponent12];
      volumeComponent21 = VolBub1A[
      curvatureComponent21, curvatureComponent22];
      volumeComponent22 = VolBub2A[curvatureComponent21, \
curvatureComponent22];
      
      (*Print["v11 = ", volumeComponent11];
        Print["v12 = ", volumeComponent12];
        Print["v21 = ", volumeComponent21];
        Print["v22 = ", volumeComponent22];*)
      
      (*declaration of target Volumes*)
      volumeOneTarget = v1 + .5*change1;
      volumeTwoTarget = v2 + .5*change2;
      
      (*Print["volumeOneTarget = ", volumeOneTarget];
        Print["volumeTwoTarget = ", volumeTwoTarget];
        *)
      
      (*declaration of Matricies*)
      startingVoltoTargetMatrix= {{volumeOneTarget -  
              startingVolume1}, {volumeTwoTarget - startingVolume2}};
      
      (*Print["first test = 
          ",IsPointinBox[volumeComponent11, volumeComponent12, v1+change1, \
v2, change1, change2]];
        *)
      If[IsPointinBox[volumeComponent11, volumeComponent12, v1, v2, change1, 
          change2],
        
        
        (*Print["v11 = ", volumeComponent11];
          Print["v12 = ", volumeComponent12];
          Print["change1 = ", change1];
          Print["change2 = ", change2];
          *)
        
        curvaturePair = {curvatureComponent11, curvatureComponent12},
        
        (*Print["2nd test = ",IsPointinBox[volumeComponent21, \
volumeComponent22, v1+change1, v2, change1, change2]];
          
          Print["v21 = ", volumeComponent21];
          Print["v22 = ", volumeComponent22];
          Print["change1 = ", change1];
          Print["change2 = ", change2];
          *)
        
        If[IsPointinBox[volumeComponent21, volumeComponent22, v1, v2, change1,
                 change2],
          
          (*Print["case II"];
            Print["v21 = ", volumeComponent21];
            Print["v22 = ", volumeComponent22];
            Print["change1 = ", change1];
            Print["change2 = ", change2];*)
          
          curvaturePair = {curvatureComponent21, curvatureComponent22},
          
          (*Print["while loop test", \
!IsPointinBox[volumeComponent11,volumeComponent12,v1+change1, v2, change1,
                    change2]];
            *)
          
          volumeChangeMatrix = {{volumeComponent11 - startingVolume1, \
volumeComponent12 - startingVolume2}, {volumeComponent21 -
          startingVolume1, volumeComponent22 - startingVolume2}};
          
          (*Print["volumeChangeMatrix = ",volumeChangeMatrix//MatrixForm];
            *)
          
          alphaBetaMatrix = Inverse[volumeChangeMatrix] . 
              startingVoltoTargetMatrix;
          (*Print["alphaBetaMatrix = ", alphaBetaMatrix//MatrixForm];
            *)
          
          alpha = alphaBetaMatrix[[1,1]];
          beta = alphaBetaMatrix[[2,1]];
          
          curvatureComponent11 = (curvatureComponent11-k1)*alpha + \
(curvatureComponent21-k1)*beta +k1;
          curvatureComponent12= (curvatureComponent12-k2)*
          alpha + (curvatureComponent22-k2)*beta + k2;
          
          (*Print["inside while loop"];
            Print["c11 = ", curvatureComponent11];
            Print["c12 = ", curvatureComponent12];
            Print["alpha = ", alpha];
            Print["beta = ", beta];*)
          
          volumeComponent11 = VolBub1A[
            curvatureComponent11, curvatureComponent12];
          volumeComponent12 = VolBub2A[
          curvatureComponent11, curvatureComponent12];
          
          (*Print["v11 = ", volumeComponent11];
            Print["v12 = ", volumeComponent12];*)
          
          ];
        
        curvaturePair = {curvatureComponent11, curvatureComponent12};
        (*Print["after loop"];
          Print["v11 = ", volumeComponent11];
          Print["v12 = ", volumeComponent12];
          Print["change1 = ", change1];
          Print["change2 = ", change2];
          
            Print["2nd 
              test = ",IsPointinBox[volumeComponent11, \
volumeComponent12, v1+change1, v2, change1, change2]];*)
        
        
        ];
      
      (*If[!
        IsPointinBox[VolBub1[curvatureComponent11, curvatureComponent12], \
VolBub2[curvatureComponent11,  curvatureComponent12], v1+change1, v2, \
change1, change2],
            
            curvaturePair = {-1, -1};
            ];
        *)
      curvaturePair
      ];


(*This returns a lower bound on the positive part of the Hutchings function \
for a given volume pair*)
LowerBoundOnPosHutchingFunction[vol1_Real, vol2_Real, error1_Real, \
error2_Real]:=
    
    2*AreaSphere[CurvaturefromVolSingle[vol1/2, error1]]+
      AreaSphere[CurvaturefromVolSingle[vol2, error2]]+
      AreaSphere[CurvaturefromVolSingle[vol1+vol2, error2]];




(*This returns an array of areas is indexed by its volume pair *)
ArrayBuilderForSingAreasandCurvatures[vMin_Real, vMax_Real, boxSize_Real, \
error_Real]:=Module[
      {AreaArray, volumeArray, n, g},
      
      g[i_Integer]= AreaSphere[CurvaturefromVolSingle[vMin+ i*boxSize, \
error]];
      
      n[i_Integer]=VolSphere[CurvaturefromVolSingle[vMin+ i*boxSize, error]];
      
      AreaArray= Array[g, Ceiling[(vMax-vMin)/boxSize], 0];
      volumeArray=Array[n, Ceiling[(vMax-vMin)/boxSize], 0];
      
      {volumeArray,AreaArray}
      ];



(*This returns an array of areas is indexed by its volume pair *)
ArrayBuilderForSingAreasandVolumes[vMin_Real, vMax_Real, boxSize_Real, \
error_Real]:=Module[
      {AreaArray, volumeArray, n, g},
      
      g[i_Integer]= AreaSphere[CurvaturefromVolSingle[vMin+i*boxSize, 
      error]];
      n[i_Integer]=vMin+i*boxSize;
      
      AreaArray= Array[g, Ceiling[(vMax-vMin)/boxSize], 0];
      volumeArray=Array[n, Ceiling[(vMax-vMin)/boxSize], 0];
      
      {volumeArray,AreaArray}
      ];


(*This returns an array of areas is indexed by its volume pair *)
ArrayBuilderForSingAreas[vMin_Real, vMax_Real, \
boxSize_Real, error_Real]:=Module[
      {AreaArray,  g},
      
      g[i_Integer]= AreaSphere[CurvaturefromVolSingle[vMin+i*boxSize, \
error]];
      
      AreaArray= Array[g, Ceiling[(vMax-vMin)/boxSize], 0];
      
      AreaArray
      ];




(*This code fills an array with curvature pairs. It moves on to the next \
region after verifying that 2 the area of the double bubble enclosing volumes \
v,w is less than the concave part of the Hutchings function for v,w*)
ArrayFillingProof[vMin_Real, wMin_Real,wMax_Real,rectangleHeight_Real, \
rectangleWidth_Real,startingCurvature1_Real, startingCurvature2_Real, \
adjustmentMainVol_Real,adjustmentSecondVol_Real] :=Module[{smallerVolume, \
largerVolume,nextRowStartingCurvatures, curvaturePair, failSafe,
       nextRowStartingPosition,counter, insideCounter, failingVolumeV, \
singAreasArray, singAreasArrayStart},
      
      failSafe=True;
      counter=0;
      insideCounter =1;
      
      singAreasArrayStart=(vMin-rectangleWidth)/2;
      singAreasArray= ArrayBuilderForSingAreas[(vMin-rectangleWidth)/2, \
2*(wMax+rectangleWidth), rectangleWidth/2, rectangleWidth/2];
      Print["Array completed, evaluating Hutchings Function"];
      (*Print["singAreasArray =", singAreasArray];*)
      
      (*Print["hi"];
        Print[failSafe];*)
      
      (*nextRowStartingPosition keeps tract of the x position that \
starts a new row
          initially it is set to 0 and then set to correct position*)
      nextRowStartingPosition = 0.0;
      nextRowStartingCurvatures = curvaturePair;
      
      
      curvaturePair = {startingCurvature1, startingCurvature2};
      
      
      smallerVolume = vMin;
      largerVolume=wMin+rectangleHeight;
      
      (*Print["first test", largerVolume ² wMax +rectangleHeight && \
failSafe];*)
      
      While[largerVolume ² wMax +rectangleHeight && failSafe,
        
        (*While the v1 is less than v2*)
        While[smallerVolume ²  largerVolume +rectangleWidth && failSafe ,
          
          failSafe = IsPointinBox[VolBubV[curvaturePair[[
      1]], curvaturePair[[2]]], 
        VolBubW[curvaturePair[[1]],curvaturePair[[2]]], smallerVolume, \
largerVolume, 2*rectangleWidth, 2*rectangleHeight]; 
          
          If[!failSafe,
            
            Print["( ",
              VolBubV[curvaturePair[[1]], curvaturePair[[2]]], ", ", 
                VolBubW[curvaturePair[[1]],
                  curvaturePair[[2]]], ") is not in the 
                  box defined by  (", smallerVolume,", ",  largerVolume,") \
and  (", rectangleWidth, ", ", rectangleHeight, ")" ];
            
            
            ];
          (*marker*)
          (*this sets up moving one row up*)
          If[ nextRowStartingPosition \[Equal]0 && .85 * largerVolume ² \
smallerVolume+rectangleWidth,
            nextRowStartingPosition = smallerVolume;
            nextRowStartingCurvatures = curvaturePair;
            
            ];
          
          (*Print["hi! 11"];*)
          (*If[insideCounter <= 5,
                Print["target volume pair = {", smallerVolume, ",", largerVolume, 
              "}"];
                Print["actual volume 
              
                            pair = {", \
VolBubV[curvaturePair[[1]],curvaturePair[[2]] ], ",", \
VolBubW[curvaturePair[[1]],curvaturePair[[2]] ], "}"];
                Print["curvatures are ", curvaturePair];
                insideCounter++;
                Print["failSafe ", failSafe];
                
                ];*)
          
          (*if there is a problem with the curvature function or the \
hutchings function is negative then the program should fail*)
          
          (*Print["v/2 from array =", \
singAreasArray[[Floor[((smallerVolume-rectangleWidth)/2-singAreasArrayStart)/(\
rectangleWidth/2)+1]]]];
            Print["index =", \
Floor[((smallerVolume-rectangleWidth)/2-singAreasArrayStart)/(rectangleWidth/\
2)+1]];
            
           
                     Print["v/2 not from array = " ,
                                  AreaSphere[CurvaturefromVolSingle[(
                            smallerVolume-rectangleWidth)/2, \
rectangleWidth/2]]];
            
            
            Print["w from array =", 
              singAreasArray[[Floor[
                2+(largerVolume-rectangleWidth-singAreasArrayStart)/(\
rectangleWidth/2)]]]];
            
            
                    Print["index =", \
Floor[2+(largerVolume-rectangleWidth-singAreasArrayStart)/(rectangleWidth/2)]]\
;
            
            
            
              Print["w not from array = " \
,AreaSphere[CurvaturefromVolSingle[largerVolume-
                rectangleWidth, rectangleWidth/2]]];
            
            Print["v+w from array =", \
singAreasArray[[Floor[1+(smallerVolume+largerVolume-rectangleWidth-
                      singAreasArrayStart)/(rectangleWidth/2)]]]];
           
                             Print["index =", \
Floor[1+(-rectangleWidth+smallerVolume+largerVolume-singAreasArrayStart)/(\
rectangleWidth/2)]];
            
            Print[
            "v+w not from array = " \
,AreaSphere[CurvaturefromVolSingle[smallerVolume+largerVolume-rectangleWidth, \
rectangleWidth/2]]];
            
            Print["v + w =", smallerVolume+largerVolume-rectangleWidth];
            Print["v = ", smallerVolume];
            Print["w = ", largerVolume-rectangleWidth];
            
            *)
          
          If[curvaturePair[[1]] < 1 \
||2*singAreasArray[[Floor[((smallerVolume-rectangleWidth)/2-\
singAreasArrayStart)/(rectangleWidth/2)+1]]]+singAreasArray[[
                              Floor[2+(largerVolume-rectangleWidth-\
singAreasArrayStart)/(rectangleWidth/
                              2)]]]+singAreasArray[[Floor[1+(smallerVolume+\
largerVolume-rectangleWidth-singAreasArrayStart)/(rectangleWidth/2)]]]+ - \
2*AreaDblBubble[curvaturePair[[1]], curvaturePair[[2]]] < 0,
            
            Print["curvaturePair =",  curvaturePair];
            Print["hutchings function = ", \
2*singAreasArray[[Floor[((smallerVolume-rectangleWidth)/2-
            singAreasArrayStart)/(rectangleWidth/2)+1]]]+singAreasArray[[\
Floor[2+(largerVolume-rectangleWidth-singAreasArrayStart)/(rectangleWidth/2)]]\
]+singAreasArray[[Floor[1+(smallerVolume+largerVolume-rectangleWidth-\
singAreasArrayStart)/(rectangleWidth/2)]]] - 2*AreaDblBubble[
                curvaturePair[[1]], curvaturePair[[2]]]];
            failSafe = False;
            failVolumeV =smallerVolume;
            
            
            Print["upper 
              bound on dbl bubble", 2*AreaDblBubble [curvaturePair[[1]], \
curvaturePair[[2]]]];
             Print["lower bound on concave",singAreasArray[[
                              Floor[((smallerVolume-rectangleWidth)/2-\
singAreasArrayStart)/(rectangleWidth/2)+1]]]+singAreasArray[[Floor[2+(\
largerVolume-
                  rectangleWidth-singAreasArrayStart)/(rectangleWidth/2)]]]+\
singAreasArray[[Floor[1+(smallerVolume+largerVolume-rectangleWidth-
            singAreasArrayStart)/(rectangleWidth/2)]]]];
            
            
            ];
          
          (*Print["inside while loop"];*)
          curvaturePair = NextCurvaturePairinArray[curvaturePair[[1]], \
curvaturePair[[2]], smallerVolume+rectangleWidth, largerVolume, 
          rectangleWidth, rectangleHeight, adjustmentMainVol, \
adjustmentSecondVol];
          
          (*Print["upper bound on dbl bubble", 2*AreaDblBubble \
[curvaturePair[[1]], curvaturePair[[2]]]];
             Print["lower bound on \
concave",LowerBoundOnPosHutchingFunction[smallerVolume, largerVolume, \
rectangleWidth, rectangleHeight]];
            *)
          
          
          
          smallerVolume = smallerVolume+ rectangleWidth;
          
          
           ];
        
        If[failSafe,
          (*Print["smallerVolume", smallerVolume];*)
          smallerVolume = nextRowStartingPosition;
          
          nextRowStartingPosition =0.0;
          counter++;
          insideCounter= 1;
          
          If[counter \[Equal]100,
            Print["w at ", largerVolume];
            counter=0;
            ];
          
          
          
          (*Print["SETTING the curvature pair"];
            Print["Volumes before {", \
VolBubV[nextRowStartingCurvatures[[1]],nextRowStartingCurvatures[[2]]], " , \
",VolBubW[nextRowStartingCurvatures[[1]],
            nextRowStartingCurvatures[[2]]], "}" ];
            
            Print["target box { (", smallerVolume ",", \
smallerVolume+rectangleHeight , ") X (
                 ", (largerVolume+rectangleHeight), ", ",  \
(largerVolume+2*rectangleHeight)") }"];
            *)
          largerVolume = largerVolume+ rectangleHeight;
          
          While[!IsPointinBox[VolBubV[curvaturePair[[1]],
              curvaturePair[[2]]], VolBubW[curvaturePair[[1]],curvaturePair[[
            2]]], smallerVolume, largerVolume, 2*rectangleWidth, \
2*rectangleHeight],
            curvaturePair = NextCurvaturePairinArray[curvaturePair[[1]], \
curvaturePair[[2]], smallerVolume, largerVolume+rectangleHeight, \
rectangleWidth, rectangleHeight, adjustmentMainVol, adjustmentSecondVol];
            ];
          (*Print["VOL1 =",VolBubV[curvaturePair[[1]], curvaturePair[[2]]] ];
            Print["VOL2 =",VolBubW[curvaturePair[[1]], curvaturePair[[2]]] ];
            *)
          
          
          failSafe=IsPointinBox[VolBubV[curvaturePair[[1]],curvaturePair[[2]]]\
, VolBubW[curvaturePair[[1]],curvaturePair[[2]]], smallerVolume, largerVolume, \
2*rectangleWidth, 2*rectangleHeight]&&smallerVolume/largerVolume².85;
          
          If[!failSafe,
            
            Print["error at end ( ",VolBubV[
                  curvaturePair[[1]], curvaturePair[[2]]], ",
                 ", VolBubW[curvaturePair[[1]],curvaturePair[[2]]], ")
                   is not in the box defined 
                    by  (", smallerVolume,", ",  largerVolume,") 
                and  (", rectangleWidth, ", ", rectangleHeight, ")" ];
            
            
            ];
          
          
          
          (*Print["test ",IsPointinBox[VolBubV[curvaturePair[[1]],
          curvaturePair[[
          2]]], VolBubW[curvaturePair[[1]],curvaturePair[[2]]], smallerVolume,
             largerVolume, rectangleWidth, rectangleHeight]];*)
          ];
        
        ];
      
      (*Print["while loop test ", smallerVolume ²  largerVolume && failSafe, \
failSafe];
        Print["first part of test", smallerVolume ²  largerVolume];
        *)
      
      
      If[failSafe,
        Print["done"],
        
        Print["failed at (",smallerVolume, ", ", largerVolume, ")"];
        Print["ratio =",smallerVolume/largerVolume];
        ];
      
      Print["failSafe = ", failSafe];
      ];

\end{verbatim}

\end{document}